\documentclass[preprint,12pt]{elsarticle}

\newcommand\diag{\text{diag}}

\usepackage[english]{babel}
\usepackage[utf8x]{inputenc}
\usepackage{amsmath} 
\usepackage{amsthm}
\usepackage{amsfonts}
\usepackage{amssymb}
\usepackage{mathtools}
\usepackage{xcolor}
\usepackage{float}

\definecolor{newcolor}{rgb}{.8,.349,.1}

\usepackage{paralist}

\usepackage{tikz}
\usetikzlibrary{shapes.geometric, arrows}
\usetikzlibrary{positioning,automata}

\tikzstyle{startstop} = [rectangle, rounded corners, minimum width=2cm, minimum height=0.8cm,text centered, draw=black]
\tikzstyle{io} = [trapezium, trapezium left angle=78, trapezium right angle=102, minimum height=0.8cm, text centered, draw=black]
\tikzstyle{process} = [rectangle, minimum width=1cm, minimum height=0.8cm, text centered, draw=black]
\tikzstyle{decision} = [diamond, minimum width=1cm, minimum height=0.4cm, text centered, draw=black, aspect=2]
\tikzstyle{arrow} = [thick,->,>=stealth]

\tikzset{IBMRRN/.style = {
     base/.style = {draw=black, 
                    inner sep=1mm, outer sep=0mm,
                    text width=2.5cm, minimum height=0.8cm,
                    align=flush center},
       io/.style = {base, trapezium,
                    trapezium left angle=78, trapezium right angle=-78,
                    trapezium stretches=false, 
                    },
                    }}

\usepackage{hyperref} 

\newcommand{\midd}{\;\|\;}
\newcommand{\abs}[1]{\left | #1 \right |}
\newcommand{\norm}[1]{\| #1 \|}

\newcommand{\rel}{\mathcal{R}}
\newcommand{\re}[2]{\rel(#1\midd #2)}
\newcommand{\rer}[2]{\mathcal{H}(#1\midd #2)}
\newcommand{\pFIM}[1]{\mathcal{I}(#1)}
\newcommand{\pFIMH}[1]{\mathcal{I}_{\mathcal{H}}(#1)}
\newcommand{\expt}[2]{\mathbb{E}_{#2}\left[#1\right]}

\newcommand{\prob}[1]{\text{Pr}\left\{ #1 \right\}}
\newcommand{\var}[2]{\text{Var}_{#2}\left[#1\right]}

\newcommand\latt{\mathcal{L}}
\newcommand\Dt{{\Delta t}}
\newcommand\tr[1]{\text{Trace}\left( #1 \right)}
\newcommand\CG{\mathbf{\varPi}}
\newcommand\CGp{\mathbf{\varGamma}}
\newcommand\comp{\perp}

\newtheorem{thm}{Theorem}[section]
\newtheorem{lemm}[thm]{Lemma}
\newtheorem{prop}[thm]{Proposition}
\newtheorem{rem}[thm]{Remark}

\def\VIZ#1{(\ref{#1})}      
\def\FISHER#1{{\mathcal I}\big(#1\big)}
\def\FISHERR#1{{\mathcal I}_{\mathcal{H}}\big(#1\big)}

\def\INIT{\nu}
\def\INITAPP{\nu^{\theta}}

\begin{document}


\begin{frontmatter}

\title{Data-driven, Variational Model Reduction of high-dimensional Reaction Networks}

\author[umass]{Markos A. Katsoulakis\corref{cor1}}
\ead{markos@umass.math.edu}

\author[umass]{Pedro Vilanova}
\ead{pedro.vilanova@gmail.com}

\address[umass]{Department of Mathematics and Statistics, University
of Massachusetts  Amherst, Amherst MA 01002, USA}

\cortext[cor1]{Corresponding author. We use alphabetical convention in author's name order.}

\begin{abstract}
In this work we present new scalable, information theory-based variational methods for the efficient  model reduction of high-dimensional deterministic and stochastic reaction networks. The proposed  methodology combines, (a) information theoretic tools for sensitivity analysis that allow us to identify the proper coarse variables of the reaction network, with (b) variational approximate inference methods for training  a best-fit reduced model. This approach takes advantage of  both physicochemical modeling and data-based approaches and allows to construct optimal parameterized reduced dynamics in the number of variables, reactions and parameters, while controlling the information loss due to the reduction. We demonstrate the effectiveness of our model reduction method  on several complex, high-dimensional chemical reaction networks arising in biochemistry.
\end{abstract}



\end{frontmatter}

\pagestyle{myheadings}
\thispagestyle{plain}

\setcounter{tocdepth}{1}

\section{Introduction}

The modeling and simulation of complex biochemical systems typically involves  non-linear and high-dimensional dynamical systems,  
in terms of both state variables and parameters \cite{SUTTON2015190,III201501}. Of particular importance would be a simpler model able to capture key characteristics of the biochemical system and therefore more amenable for analysis, parameter identification, statistical inference, and eventually design and optimization. Model reduction techniques seek to obtain such models, but in many cases the required computational work may quickly become prohibitive. Additionally, modeling goals usually constrain these reduction methods,  for example, the biochemical meaning of the state variables is likely to be difficult to interpret if non-linear (and even linear) coarse-graining transformations are applied during the model reduction process. 
The most widely applied reduction methods can be roughly classified as timescale exploitation approaches \cite{III201501,briggs1925note,rou01, kooi02, radulescu2008robust, petrov2007reduction,Schneider2000,maas1992simplifying,vallabhajosyula2005conservation,zobeley2005new, surovtsova2009accessible}, reduction methods based on sensitivity analysis \cite{liu2004sensitivity,DEGENRING2004729, apri2012complexity, turanyi1990sensitivity, tomlin1995reduced, liu2004sensitivity, maurya2005reduced,jayachandran2014optimal}, optimization based methods \cite{maurya2005reduced, maurya2009mixed, hangos2013model,locke2005extension,anderson2011model,prescott2012guaranteed}, and lumping methods \cite{dano2006reduction,dokoumetzidis2009proper,koschorreck2007reduced,sunnaaker2011method}. 
Finally, an important class of model reduction methods are based on maximum entropy techniques
 for the closure of the equation that determines the time evolution of the probabilistic description of the stochastic reaction network (see for instance \cite{majda2005information, CONSTANTINO2017139, lee2009moment,gillespie2009moment,grima2012study}). 
In particular, the method proposed in this work can be considered as a combination of sensitivity and optimization-based methods, using an information theory approach. We refer to Section \ref{sec:comparison} for a discussion on the connections between our method and the current literature.

It is well-known that, when the population sizes of the biochemical system become small, a deterministic formulation of the dynamics is inadequate for understanding many important properties of the system (see for instance \cite{erdi, Wilkinson:12}). Therefore, our method is primarily concerned with stochastic, Markovian models of reaction networks.
In order to estimate distances between corresponding probability distributions for our stochastic models and  train models from data,   information metrics such as the relative entropy are natural choices, \cite{Cover:91}.
The relative entropy (Kullback-Leibler divergence) of two probability measures $P$ and $Q$ 
is given by 
\begin{equation}\label{relent:intro}
\re{P}{Q} = \expt{ \log \frac{dP}{dQ} }{{P}} \, ,
\end{equation}
where $\frac{dP}{dQ}$ is the Radon-Nikodym derivative for $P$ absolutely continuous 
with respect to $Q$ (see for instance \cite{Cover:91}).  
Entropy-based analytical tools have
proved essential for deriving rigorous model reductions of interacting particle
models to the so-called hydrodynamic limits, \cite{Landim}.

In a closely related direction, information metrics provide systematic, practical, and widely used 
tools to build  approximate statistical models of reduced complexity through  
variational inference methods \cite{MacKay:2003,bishop,stuartKL}  for  
machine learning \cite{Wainwright:2008,Blei:2013,bishop}, 
and coarse-graining of complex systems at equilibrium \cite{Shell2008,shell2010,NoidReview:2011,Bilionis:2012,zabaras2013,Foley}. However, dynamics are of critical importance in reaction networks and such earlier works on equilibrium coarse-graining  are not directly applicable.
Here we build on earlier work \cite{KP13, HKKP16, Kalligiannaki2016}, where coarse-graining methods for dynamics and non-equilibrium steady states were derived for molecular dynamics and Kinetic Monte Carlo methods. In particular, in order to address model reduction for dynamics it is essential to consider time series data that include temporal correlations, necessitating the use of information theory methods for probability distributions in path-space, i.e. the space of all possible time series. 

In this work we develop a method that, given a reaction network with path distribution $P_{0:T}$ , i.e., a probability over the set of all possible dynamics over a time interval, finds an element from a parameterized family of distributions that is close to $P_{0:T}$.
The novelty of this work is two-fold, (a) it extends the coarse-graining path-space information theory methods of \cite{KP13, HKKP16} to achieve model reductions of complex reaction networks, and  (b) obtains the most efficient model reduction in (a) by applying a path-space sensitivity analysis technique  presented in \cite{PK2013,PKV:2013}, thus identifying the most sensitive model parameters. 
More specifically,  in (a) we consider the path probability distribution $P_{0:T}$ of a high-dimensional reaction network, i.e., a probability over the set of all possible dynamics (time series)  on a time interval $[0,T]$; then we seek  an element  from a parameterized family of reduced models distributions that is closest to $P_{0:T}$ with respect to a loss function, which in this case is the relative entropy:
\begin{equation}
\label{eq:min_prob_intro}
\min _{\theta \in \Theta} \re{P_{0:T}}{Q^\theta_{0:T}} \, . 
\end{equation}
Here by $Q^\theta_{0:T}$ we denote the parameterized path probability distribution, where $\theta$ belongs to a certain parametric space $\Theta$. 
As we will show in this paper, $\re{P_{0:T}}{Q^\theta_{0:T}}$ in \eqref{eq:min_prob_intro} turns out to be a  parameter-dependent computable quantity that can be fitted (or trained) by means of the  time series data of the high-dimensional model $P_{0:T}$. 

Furthermore, in Step (b) in order to  determine sensitive model parameters, we use  
 the Hessian of the relative entropy $\re{P^c_{0:T}}{P^{c+\epsilon}_{0:T}}$  where the vector $c$ corresponds to the parameters of the 
 high-dimensional model $P_{0:T}=P^c_{0:T}$ and $\epsilon$ is any vector perturbation of $c$,  i.e., the pathwise Fisher Information Matrix (pFIM), \cite{DKPP16, PKV:2013}. This pFIM is block-diagonal and scales linearly with the number of parameters, thus is computationally an efficient tool for sensitivity analysis of reaction networks.  In turn, this sensitivity analysis identifies and ranks   sensitive model parameters, but crucially for the optimization in   \eqref{eq:min_prob_intro}, determines  a  family of reduced coarse variables and associated models  $Q^\theta_{0:T}$ by retaining only the corresponding reactions and species to just the sensitive parameters. In this fashion, step (b)  enables  step (a) in order to obtain the best-fit reduced model  by optimizing the loss function in  \eqref{eq:min_prob_intro},  over only the sensitive parameters $\theta$ and subsequently the proper family of reduced models $Q^\theta_{0:T}$. 
Overall,  the path-space model reduction  methodology introduced earlier in 
\cite{KP13, HKKP16}
is here augmented to include the sensitivity analysis and uncertainty quantification,  information-theoretic tools described earlier  that allow us to identify the proper coarse variables, a long-standing open question in the  coarse-graining literature. This  latter was made feasible here, both due to the use of the pFIM and the inherent sparse structure of reaction networks.

The resulting method provides  a simple, efficient and principled model reduction method targeted to high-\! dimensional deterministic and stochastic reaction networks. This method enables significant model reductions in the number of variables, reactions and parameters, while preserving the original dynamics, especially in the case of models in ``sloppy'' regimes, i.e., when most of the information contained in the pFIM is accumulated in a reduced number of parameters, for a given time interval. This reduction method  construct models in such a way that the information contained in the reduced model, relative to the full model, satisfies a user defined information threshold. Furthermore, the quality of a given reduction is assessed by comparing suitable distances between the mean-field of the full and the reduced models, and aim to satisfy a user defined error tolerance.
One of the main advantages of our reduction method is its feasibility in terms of computational work. From the algorithmic point of view, to construct the reduced model in step (b) using the pFIM (not including the minimization step \eqref{eq:min_prob_intro}) a linear amount of computational work on the number of parameters plus the number of state variables is usually required (per timestep). As a comparison, note that the work required for classical sensitivity analysis is of the order of the number of parameters times the number of state variables. In a nutshell, our method can be visually described as in Figure \ref{fig:method}.

We demonstrate the efficiency of our method by reducing three well-known models present in the biochemical literature: A protein homeostasis network in a sloppy regime from \cite{proctor11}, a Epidermal Growth Factor Receptor (EGFR) model from \cite{Kholodenko99}, a mammalian circadian clock model from \cite{Leloup03}, and finally two stochastic pure jump models: a 
 growth factor receptor derived from \cite{Kholodenko99}, and a circadian clock model derived from \cite{Leloup03}. In the first case, we consider a particular regime studied by the authors, in which it is possible to obtain a substantial reduction both in the number of state variables and in the number of parameters and reaction channels. In the second example we analyze a non-sloppy model of signalling phenomena, i.e. has two different regimes in time, and still able to obtain significant reductions. In the third example, we reduce a model whose dynamics are dominated by non-trivial oscillatory behaviour. Finally, both stochastic models are reduced to obtain remarkable agreements even in the presence of non-Gaussian distributions. In all the examples, even though the models were carefully designed by researchers, significant model reductions are achieved. 

One of the key elements  of our approach is that we combine two different modeling philosophies: the physically based behavior description of a biochemical system, and the data-based approach, i.e., the analysis of time series data. The physics is modeled by using the reaction network formalism, and in particular we focus on the stoichiometry matrix, which describes the relation between state variables and reaction paths of the process and does not depend on the type of the model (discrete or continuous, deterministic or
stochastic). The physical modeling point of view is then emphasized in the selection of the most important reactions and state variables and the construction of the reduced stoichiometry. The data-based approach is emphasized in the data fit step i.e. the minimization problem \eqref{eq:min_prob_intro} via a suitable simplified loss function that is derived in Section \ref{sec:loss_fct}.
Three main challenges were considered in the reduction method presented in this work. The first one, is to preserve the structure and therefore the physical model of the original network. This means not only to incorporate the reaction network formalism, but in particular to exploit its stoichiometry in order to identify the most relevant variables, parameters and reaction channels. This effectively reduces data requirements in comparison to alternative machine learning techniques that do not take into account the scientific domain knowledge and therefore substantially accelerates identification and training steps. In particular, our method allows to optimize the parameter vector by considering the drift associated to the reduced model instead of the full model, dramatically reducing the complexity of the optimization step (see Theorem 4.1).
Second challenge was to preserve physicochemical understanding of the reduced model, which is of crucial importance in mathematical modeling. User confidence in model predictions is directly linked to the conviction that the model accounts for the correct variables, parameters, and applicable physicochemical laws. The goal of obtaining understandable models which are readily interpretable becomes significantly more difficult in the high-dimensional case. Computational intensive and involved model reduction techniques may achieve more parsimonious reductions, but the lack of interpretability makes them insufficient specifically in the high-dimensional case. 
Finally, we aimed to obtain a method able to provide, in a robust manner, reduced models within a prescribed accuracy and information loss, requiring the least amount of computational work. The iterative nature of our approach, allows to subsequently construct improved reduced models by increasing the number of variables, parameters and reactions channels. Then, a hierarchy of reduced models for model selection is obtained. Even though our method uses the stoichiometry of the network to ontain the reduced model, choose sensitive species, further structural and chemical information may be extracted from it. In that respect, our method can be considered as a knowledge-domain-aware scientific machine learning technique that makes use of well established machine learning methods such as variational and approximate inference.

The paper is structured as follows. In the next section we present main notations and assumptions regarding reaction networks. In Section \ref{sec:inf_methods} we review main information-based tools used in the variational model reduction approach for stochastic systems, such as the pathwise Fisher information matrix, scalable information-based sensitivity indexes and its mean-field approximation. In Section \ref{sec:var_inf} we recap the variational inference idea applied to coarse-graining and we present the core of the contribution contained in this work: a scalable loss function for data training. In Section \ref{sec:MR} we present our information-based reduction method for high dimensional reaction networks. This reduction method is driven by the pathwise Fisher information matrix to find the most sensitive parameters and reaction channels, and then uses the stoichiometry of the network to build a family of parameterized models to finally use the loss function to fit the resulting model to time series data. In sections \ref{sec:exp} and \ref{sec:stoc_exp} we apply our method to three models taken from the literature and two proposed stochastic pure jump models obtaining relevant reductions in all of them.
In Section \ref{sec:comparison} we discuss closely related state of the art on model reduction for biochemical reaction networks.
We finally draw main conclusions of this work in Section \ref{sec:conc}. 

\section{Models for reaction networks}
\label{sec:RNs}
The purpose of this section is to present main notations and assumptions regarding reaction networks.
A Reaction Network is a dynamical system whose state can be described by a $d$-dimensional time dependent vector, denoted by $X(t)=(X_1(t),X_2(t),...,X_d(t))\in \latt$, and a finite number, $J$, of possible transitions of the system. Each transition is a pair $(\nu_j,a_j)$, $j{=}1,2,...,J$, formed by a $d$-dimensional state-change vector $\nu_j$ and a non-negative function $a_j=a_j(x;c)$ that depends on the state of the system, $x \in \latt$, and a $K$-dimensional vector of model parameters $c \in \mathcal{C}$. Each state-change vector describes the effect of the transition on the state of the system. If the state of the system at time $t$ is $X(t){=}x$ and the $j$-th transition is taken, then the next state is $x{+}\nu_j$. The function $a_j$ models the rate at which the $j$-th reaction occurs. The matrix $\nu$ is defined as the matrix whose $j$-column is $\nu_j$, and $a$ is the column vector whose components are $a_j$. 
In this work, we assume $\nu_j \in \mathbb{Z}^d$, $\latt = \mathbb{R}^d$, and $\mathcal{C}=\mathbb{R}^{K}$. More importantly, we focus on high-dimensional systems, i.e., $J{\gg} 1$, $d{\gg} 1$, and $K{\gg} 1$.

In biochemical reaction networks, $\nu_j$ is called stoichiometric vector, $a_j$ is called propensity function, the pair $(\nu_j,a_j)$ is called reaction channel, and $X(t)$ accounts for the population of $d$ interacting species $(S_1,S_2,...,S_d)$. The stoichiometry of each reaction channel can be represented as
$$
\nu_{in,1}S_1+\nu_{in,2} S_2 +\ldots + \nu_{in,d}S_d \rightarrow \nu_{out,1}S_1+\nu_{out,2}S_2+\ldots +\nu_{out,d}S_d
$$
where $\nu_{in,i}, \nu_{out,i} \in \mathbb{N}$ for $i{=}1,2,...,d$. The left hand side of the arrow represents the species that take part as the input in the reaction (called reactants), and the right hand side represents the species that take part as the output (called products). We then have $\nu = \nu_{out}-\nu_{in}$. The stoichiometry models the law of conservation of mass where the total mass of the reactants equals the total mass of the products. Stoichiometry describes the static, algebraic structure of the network of reactions. It can be considered as the framework within each chemical motion take place.
Here we focus on the stoichiometry of the network to construct a reduced model.
Having the full stoichiometry information allows to identify when a reactant is consumed in a reaction, and distinguish it from a catalytic reactant, which is not consumed.  In some cases, the stoichiometry determines the form of the propensity function, e.g. in the case of the law of mass action, which is the proposition that the rate of a chemical reaction is directly proportional to the product of the activities or concentrations of the reactants. 
In this work, we assume that, for every reaction channel, vectors $\nu_{in}$ and $\nu_{out}$ are sparse. That is, $|\{i: \nu_{in,i}\neq 0\}| \ll d$ and $|\{i: \nu_{out,i}\neq 0\}| \ll d$, for $j{=}1,2,...,J$. This is typically the case in biochemical reaction networks.

Each reaction channel depends on a vector of model parameters. 
Here we denote with $\varphi$ the function that maps, for each parameter, the set of reaction channels that explicitly depend on that parameter, 
\begin{equation}
\label{eq:varphi}
\varphi: \{1,2,\ldots,K\} \rightarrow \wp(\{1,2,\ldots,J\})\, ,
\end{equation}
where $\wp(A)$ denotes the set of all possible subsets of $A$. For example, in the linear parameter dependent case, the map $\varphi$ is the identity that maps the index $k$ to the singleton $\{k\}$.
In this work, we assume that the reaction channels depend on a reduced number of parameters, that is, for a given propensity function $a_j$, we have
\begin{equation}
\label{eq:sp_pars}
a_j(x;c)=a_j(x;c_{k_1},c_{k_2},...,c_{k_M}) \, ,
\end{equation}
where $k_1,...,k_M \in \{1,2,...,K\}$ and $M {\ll} K$ for $j{=}1,2,...,J$. This is typically the case in biochemical reaction networks. 

Next, we discuss stochastic and deterministic models for RNs. The stochastic models are of two types: discrete state, called Stochastic Reaction Networks (SRNs), and continuous state, usually called the Chemical Langevin equation (CLE). The first one is a counting stochastic process, and usually models the number of particles of different species interacting in a fixed volume, while the second is a diffusion process that models concentrations instead of number of particles. The deterministic model is usually called reaction-rate system (of differential equations) and models the average concentration of the species. By means of a proper renormalization of the variables by the size of the volume in which the species interact, it is possible to obtain a relation between the stochastic intensity and the deterministic reaction rate. In this work we assume the scalings are properly done, whenever it may be needed (see \cite{pedrodiss} and references therein).

\subsection{Stochastic reaction networks}
Stochastic Reaction Networks (SRNs) are a class of continuous-time Markov
chains that describe the stochastic evolution of a system of $d$ interacting species, where
$X:\mathbb{R}_+ \times \Omega \to \mathbb{N}^d$ models the number of particles of each species present in the system at time $t$. 

The probability that reaction $j$ occurs during an infinitesimal interval
$(t,t+\Delta t]$, when the state of the system is $X(t)=x$, is given by
\begin{align}\label{eq:trProba}
 \prob{X(t+ \Delta t)=x+ \nu_j \bigm| X(t)=x}=a_j(x;\cdot)\Delta t + o(\Delta t) \, .
\end{align}

We assume $a_j(x;\cdot){=}0$ for $x$ such that $x{+}\nu_j \notin \latt$, that is, the process never leaves the domain $\latt$. This is sometimes not the case in some biochemical models used in practice, but it can be suitably enforced by means of smooth indicator functions.
The total propensity (or rate) $a_0(x;c):=\sum_j a_j(x;c)$ is in fact the  waiting time for the process departing from state $x$. 
SRNs can be characterized by the following representation \cite{kurtzmp} 
\begin{equation}
\label{eq:kurtz}
X(t)= x_{0}+\sum_{j=1}^{J} Y_j \left( \int_0^t  a_{j}(X(s)) \, d s \right) \nu_j  \, ,
\end{equation}
where $Y_j:\mathbb{R}_+{\times} \Omega \to\latt$ are independent unit-rate Poisson
processes. 
A typical feature of biochemical systems is that the modeled reaction network
 has thousands of species and/or reaction channels together with  different time scales coming
from the orders of magnitude disparity between the propensity of each reaction channel, making an exact simulation of a SRN unfeasible from the computational point of view. 
The tau-leap method \cite{Gillespie:01} approximates \eqref{eq:kurtz} by applying a forward Euler discretization, that is, sampling a batch of events by means of a Poisson random variable at each time-increment. Although several improvements of
the basic tau-leap algorithm have been proposed, methods for high-dimensional systems are still under active research \cite{Cao:06,Tian:04, Chatterjee:05, chernoff, chernoffML}.

\subsection{Chemical Langevin equation and the mean-field equation}
A number of approximations to the pure jump process \eqref{eq:trProba} have been developed in order to reduce the computational work required to sample a trajectory of the system. For example, the reaction-rate ODE system (or mean-field) approximation ignores the stochastic fluctuations and yields a
deterministic system that approximates the mean populations of the species \cite{Gardiner:85,
vanKampen:06}. Stochastic counterparts such as the
Chemical Langevin equation \cite{Gillespie:00} or the linear noise approximation \cite{Kurtz:72} can be applied in order
to improve the accuracy of the simulation. 
%
The reaction-rate ODE system (or mean-field) can be formally obtained by linearizing the infinitesimal generator of the SRN, to obtain
\begin{align}%
\label{eq:rrODE}
\left\{ \!
\begin{array}{l@{\;}c@{\;}l}
dz(t) &=& \nu a(z(t);c) dt , \quad t \in \mathbb{R}_+ \\
z(0) &=& x_0
\end{array}
\right. \, .
\end{align}
A second order expansion gives the so called chemical Langevin (CLE)  approximation 
\begin{align}%
\label{eq:lang}
\left\{ \!
\begin{array}{l@{\;}c@{\;}l}
dY(t) &=& \nu a(Y(t);c)dt +  \nu \sqrt{\diag(a(Y(t);c))}dW(t), \quad t \in \mathbb{R}_+ \\
Y(0) &=& x_0
\end{array}
\right. \, ,
\end{align}
where $Y(t)$ is a diffusion process, $W(t)$ is a $\mathbb{R}^J$-valued Wiener process with independent components, and $\diag(v)$ of a vector $v$ is a square matrix whose diagonal is $v$ and zero everywhere else (see \cite{pedrodiss} and references therein).

\emph{Euler discretization of the CLE.}
Usually, specialized solvers are used by the modeller to obtain time series data of a complex reaction network by running suitable software like CHEMKIN \cite{kee1980chemkin} or TChem \cite{safta2011tchem}. 
In this work, we aim to obtain model reductions by using available data. For this reason, and without losing any methodological generality, we focus on describing the dynamics of the reaction network by means of the following Euler discretization of the CLE \eqref{eq:lang} for a time-step of size $\Dt$
\begin{equation}
\label{eq:ecle}
x_{k+1} = x_k + b(x_k) \Dt + \sigma(x_k) \Delta W_k \, ,
\end{equation}
where 
$$b(x) := \nu a(x;c) \quad \text{ and } \quad  \sigma(x) := \nu \sqrt{\diag(a(x;c))} \, ,
$$
with $\Delta W_k \sim \mathcal{N}(0,\Dt I)$ independent Gaussian increments. The density of its transition probability, given $x \in \mathbb{R}^d$ to $x' \in \mathbb{R}^d$, is a multivariate Gaussian random variable that depends on $\Dt$ and can be written
\begin{equation}
\label{eq:microp}
p(x,x') = \frac{1}{Z_\Dt(x)} \exp \left\{ -\frac{1}{2\Dt}(x'-m_\Dt(x))^{tr} \Sigma^{-1}(x)(x'-m_\Dt(x)) \right\} \, ,
\end{equation}
where $m_\Dt(x) := x-b(x) \Dt$, $Z_\Dt(x) := \sqrt{(2 \pi \Dt)^d \det(\Sigma(x))}$ and $\Sigma(x) := \sigma(x) \sigma^{tr}(x)$. 

This Euler discretization of the CLE allows to obtain the following time series 
\begin{equation}
\label{eq:data}
(x_0,x_1,x_2,...,x_T) \, , \quad \text{ where } x_i{=}(x_{i,1},x_{i,2},...,x_{i,d}) \text{ for } i{=}0,1,...,T \, ,
\end{equation}
for $(\Dt_i)_{i=1}^T$ on the time interval $[0,T]$. Notice that similar time series can also be obtained by sampling the representation \eqref{eq:kurtz} or by the numerical integration of the reaction-rate ODEs \eqref{eq:rrODE}. In this work, we do not consider measurement errors.

\section{Path-space information methods in discrete time}
\label{sec:inf_methods}
The purpose of this section is to review main information-based tools used in the variational model reduction approach for stochastic systems introduced in  \cite{HKKP16,KP13}.

\subsection{Information-based model reduction}
In this section we briefly recap the main idea of information-based model reduction.
Given a reaction network with path distribution denoted by $P_{0:T}$ (the full model), we aim to find an element from a parameterized family of distributions that is close to this original distribution in a suitable distance. Let $Q^\theta_{0:T}$ denote the parameterized path distribution, where $\theta$ belongs to a certain parametric space $\Theta$.  
Using the relative entropy as the distance, we aim to solve
$$
\min _{\theta \in \Theta} \re{P_{0:T}}{Q^\theta_{0:T}} \, ,
$$ 
to find the optimal representation of $P_{0:T}$ in terms of $Q^\theta_{0:T}$. We notice that $P_{0:T}$ also depends on a model parameter vector, i.e., $P^c_{0:T}$, but for simplicity we avoid this notation here. Recall that the relative entropy is not symmetric so we may be also interested in the problem 
$$
\min _{\theta \in \Theta} \re{Q^\theta_{0:T}}{P_{0:T}} \, ,
$$
typically addressed in the variational inference literature, [REF7].
Here we choose to work with the former one, primarily due to the relative simplicity  of implementation.
The relative entropy between the two path distributions $P_{0:T}$ and $Q^\theta_{0:T}$ (or Kullback-Leibler divergence), on the same measurable space, is given by
$$
\re{P_{0:T}}{Q^\theta_{0:T}} = \expt{\log \frac{dP_{0:T}}{dQ^\theta_{0:T}}}{P_{0:T}}\, ,
$$
provided $P_{0:T}$ is absolutely continuous with respect to $Q^\theta_{0:T}$. From an information theory perspective, the relative entropy  quantifies the loss of information when $Q^\theta_{0:T}$ is used instead of $P_{0:T}$.  One  notable analytical advantage of the relative entropy is that it 
reduces to minimizing the expectation of a single distinguished observable given by 
$\log \frac{dP_{0:T}}{dQ^\theta_{0:T}}$.
We refer to \cite{PK2013,PKV:2013,DKPP16} and references therein for additional details, in particular for relative entropy for discrete time Markov Chains.
%
In our setting of model reduction (or coarse-graining), the path distribution $P_{0:T}$ is associated with the original model (which will be mapped to the coarse space in order to be able to compute the relative entropy) and $Q^\theta_{0:T}$ is associated with the approximating reduced model. The parameterized family of distributions $Q^\theta_{0:T}$ is considered in order to construct the best approximation to the original distribution $P_{0:T}$. This approximation is then fitted using entropy-based criteria over a set of data coming from the model with  distribution $P_{0:T}$, to find the best possible Markovian approximation of the projected original model. By training this single functional given by the relative entropy instead of an observable-based quantity of interest, we obtain a reliable parameterization that gives rise to transferability properties applicable to any observable. 

\subsection{Pathwise relative entropy}
\label{sec:time_series}
Here we present the pathwise distribution of the original model, $P_{0:T}$, and the parameterized one, $Q^\theta_{0:T}$, for a discrete-time Markovian time-homogeneous process that generates the time series $(x_i)_{i=0}^T$ (see \eqref{eq:data}).
Let $p(x,x')$ denote its transition probability function for $x,x'\in \latt$.
In virtue of the Markov property, the path space probability distribution $P_{0:T}$ for the time series $\{x_i\}_{i=0}^T$, starting from the distribution 
 $\INIT(x)$, is given by
\begin{equation}
\label{eq:pdist_CLE}
P_{0:T}\big(x_0,\ldots, x_T \big) = \nu(x_0) p(x_0,x_1) \ldots  p(x_{T-1},x_T)\, .
\end{equation}
Then the probability density function at the time instant $i$ is denoted as $\nu_{i}(x)$ and given by 
\begin{equation*}
\nu_{i}(x) = \int_\latt\cdots\int_\latt \nu(x_0)p(x_0,x_1)\ldots p(x_{i-1},x) dx_0\ldots dx_{i-1} \, , \quad i{=}1,2,...,T \, ,
\end{equation*}
where $\nu_0(x):= \nu(x)$.

In the same line, consider a parameterized  transition probability
density function $q^{\theta}(x,x')$, which depends on the parameter vector $\theta \in \Theta$ for $x, x' \in \latt$. Its path space probability distribution, starting from $ \nu^\theta(x)$, is given by
\begin{equation}
\label{eq:ppdist_CLE}
Q^\theta_{0:T}\big(x_0,\ldots, x_T \big) =  \nu^\theta(x_0) q^\theta(x_0,x_1) \ldots  q^\theta(x_{T-1},x_T)\, .
\end{equation}
The pathwise relative entropy can be decomposed as (see Appendix)
\begin{equation}
\re{P_{0:T}}{Q^\theta_{0:T}}= \re{\INIT}{\INITAPP} + \sum_{i=1}^T \re{P_i}{Q^\theta_i} \ ,
\label{eq:pRE}
\end{equation}
where the following quantity
\begin{equation}
\re{P_i}{Q^\theta_i}
= \expt{\int_\latt p(x,x')\log \frac{p(x,x')}{q^{\theta}(x,x')}dx'}{\nu_{i-1}} \, ,
\label{eq:iRE}
\end{equation}
can be interpreted as the ``instantaneous relative entropy''. 

\subsection{Pathwise Fisher information matrix of a parameterized distribution}
\label{sec:pFIM}
In this work, we employ fast parametric screening with controlled accuracy to detect and discard (eliminate or fix) parameters considered insensitive. This screening step is 
based on sensitivity indexes that can be bounded by the pathwise Fisher information matrix (pFIM),
which is presented in this section. We stress that in the high-dimensional RN case, fast screening is
especially important due to the large number of parameters that increase by orders of magnitude the computational work when compared to the simulation of the model. We refer to \cite{PKV:2013, PLOS:AKP} for further details.

Consider a vector $\epsilon \in \mathbb{R}^K$ such that $c{+}\epsilon$ is a small perturbation of the model parameter vector $c$ with non-negative components. 
Assume that the instantaneous pathwise relative entropy $\re{P^{c}_i}{P^{c+\epsilon}_i}$ is smooth with respect to $c$. Then, in combination with its non-negativity property, it can be Taylor-expanded around $c$ (see Theorem \ref{thm:pFIM} in the Appendix) to obtain  
\begin{equation}
\label{eq:RE_FIM}
\re{P^c_i}{P^{c+\epsilon}_i}= \frac{1}{2} \epsilon^{tr} \FISHERR{P^c_i} \epsilon + O(|\epsilon|^3) \, ,
\end{equation}
where
\begin{equation}
\label{eq:iFIM}
\FISHERR{P^c_i} =  \expt{\int_\latt p^c(x,x') \nabla_c \log p^c(x,x') \nabla_c \log p^c(x,x')^{tr} dx'}{\nu_{i-1}^c} \, 
\end{equation}
is the instantaneous Fisher information matrix associated to the instantaneous relative entropy. 
The Fisher information is a measure of the amount of information that a random variable contains regarding a set of parameters. An appealing property is that the FIM is independent of the perturbation vector, $\epsilon$, and contains up to third order accuracy the sensitivity information as it is quantified by
the relative entropy. 
Consequently, the pathwise FIM, i.e., the Hessian of the
pathwise relative entropy at $c$ is given by
\begin{equation}
\label{eq:pFIM}
\FISHER{P^c_{0:T}} = \FISHER{\nu^c} + \sum_{i=1}^T  \FISHERR{P^c_i} \ ,
\end{equation}
where $\FISHER{\nu^c} = \expt{\nabla_c \log\nu^c(x)\nabla_c \log\nu^c(x)^{tr}}{\nu^c}$
is the FIM of the initial distribution.

\emph{Block diagonal structure of the pFIM.} 
As previously mentioned, reaction channels of biochemical reaction networks typically depend on a reduced number of parameters, which determine a block diagonal structure of the pFIM. This parametric dependence allows to reduce the computational work of the quantities to estimate from $O(K^2)$ to $O(K)$. This reduction may be essential in the  high-dimensional case. In Figure \ref{fig:sp_pFIM} we show two examples of the block structure of the pFIM.

\begin{figure}[h!]
\centering
\begin{minipage}{0.3\hsize}
	\includegraphics[width=\textwidth]{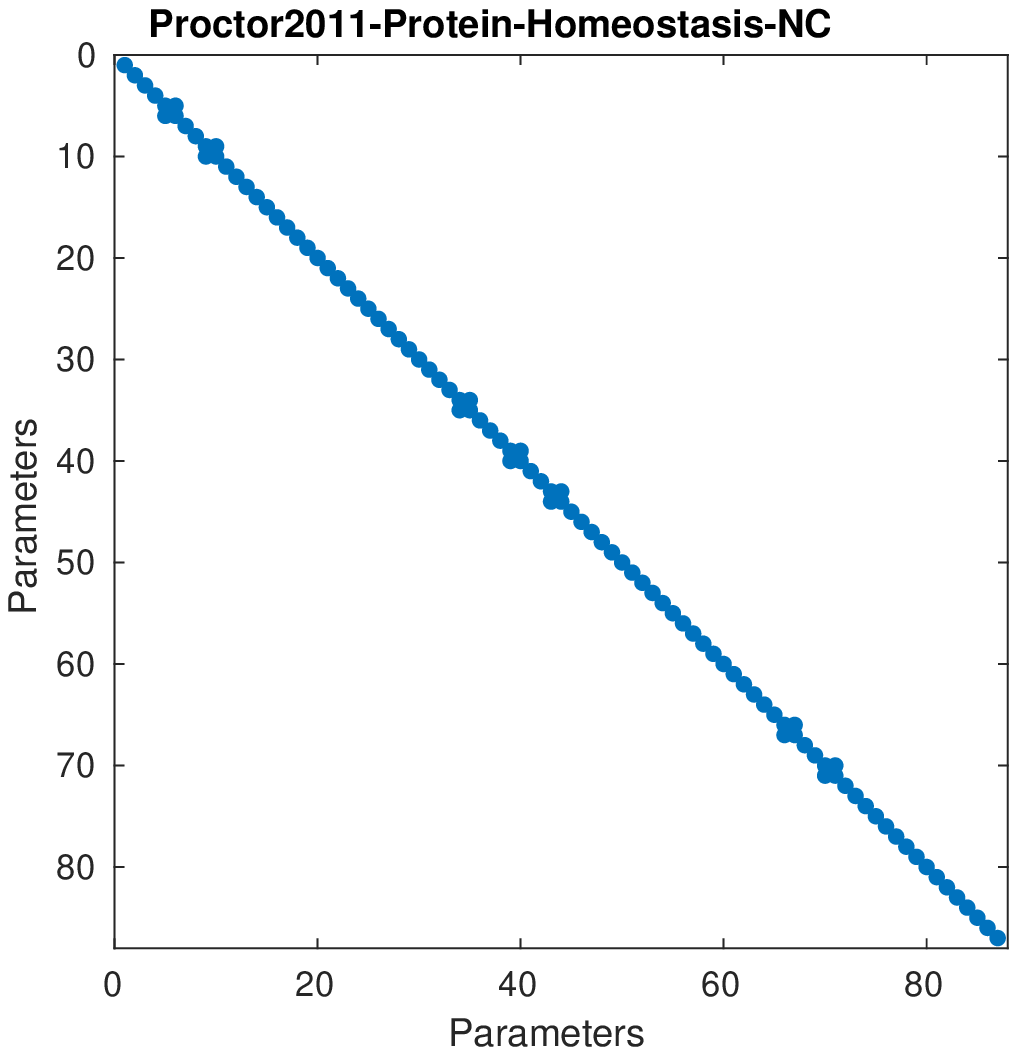}
\end{minipage}
\hfill
\begin{minipage}{0.3\hsize}
	\includegraphics[width=\textwidth]{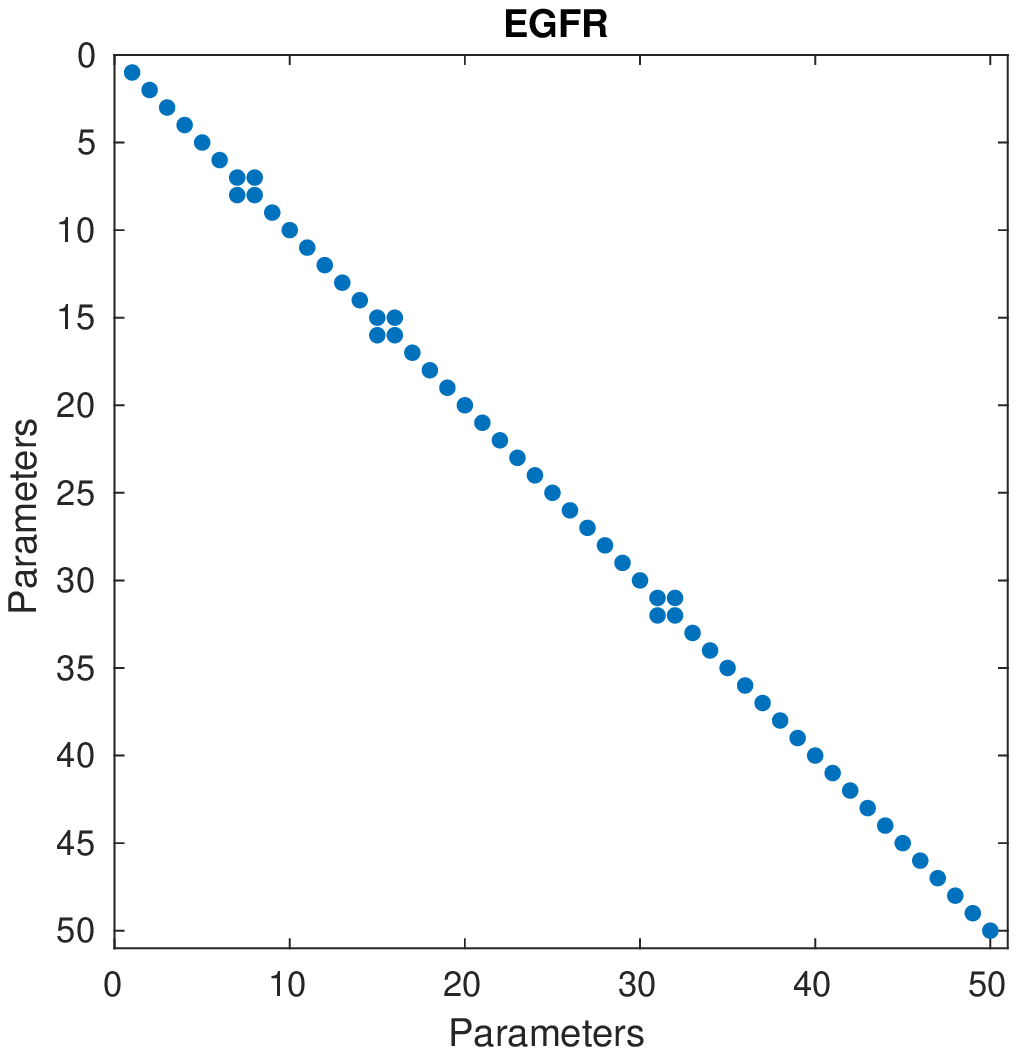}
\end{minipage}
\hfill
\begin{minipage}{0.3\hsize}
	\includegraphics[width=\textwidth]{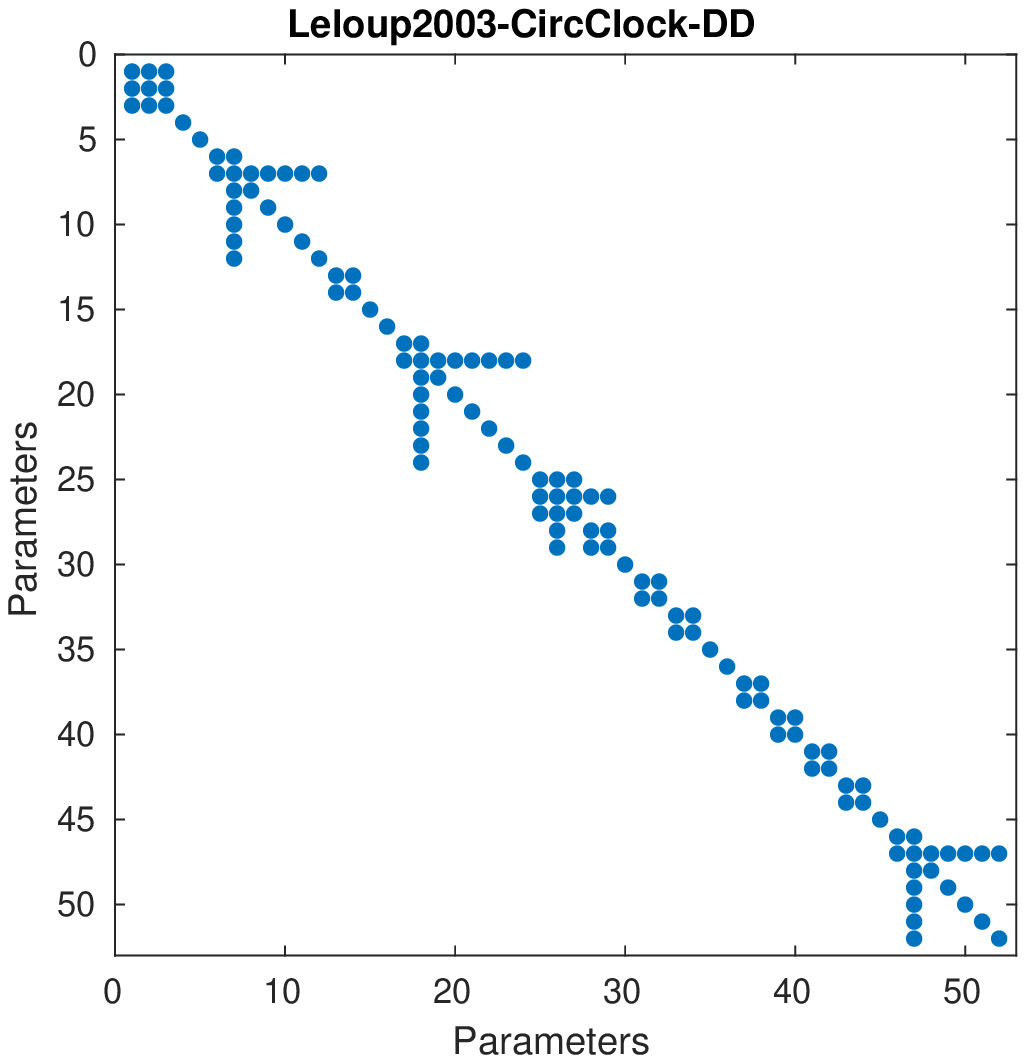}
\end{minipage}
\caption{\label{fig:sp_pFIM} Pathwise FIM block structure in the three models presented in Section \ref{sec:exp}. Protein Homeostasis model \cite{proctor11}, Epidermal Growth Factor Receptor (EGFR) model \cite{Kholodenko99} and Mammalian Circadian Clock model \cite{Leloup03}. We notice that the sparsity of the pathwise FIM allows to reduce the computational work from $O(K^2)$ to $O(K)$.}
\end{figure}

%

\begin{rem}[Fisher information for model reduction]
Fisher information matrices has been used before in the context of model reduction for reaction networks  (see \cite{III201501} and references therein). However, these information matrices are different from our pathwise FIM. In \cite{III201501} the matrix is computed based on the adjoint system \eqref{eq:adjoint}, so the computational work issue still remains. In \cite{chu2008parameter} the authors assume that the quantities of interest (measurements) are affected by identically distributed and independent Gaussian noise. Then, the Fisher information matrix of these independent measurements is computed. Similar comments apply to \cite{cintron2009sensitivity,yue2006insights}. We stress here that our pFIM is computed by taking into account the Markovian dynamics of the process and therefore, taking into account its intrinsic noise (see \eqref{eq:pFIM}).
\end{rem}

\begin{rem}[Pathwise FIM for Stochastic Reaction Networks]
The pFIM over an interval $[0,T]$ is given by
$$
\pFIM{P_{[0,T]}^c} = \pFIM{\nu^c} + \int_{0}^{T} \pFIMH{P_{t}^c} dt \, ,
$$
where $\pFIM{\nu^c}$ is the FIM of the initial distribution and the process $P^c_t$ can be viewed as the instantaneous pFIM given by
\begin{equation}
\label{eq:pFIM_SRN}
\pFIMH{P_{t}^c} = \expt{\sum_{j=1}^J a_j(X_{t-};c) \nabla_c \log a_j(X_{t-};c) \nabla_c \log a(X_{t-};c)^{tr}}{P^c_{[0,t]}} \, ,
\end{equation}
where $X_{t-}$ denotes the left side limit at time $t$.
The transition probabilities for the embedded discrete time Markov chain defined as $(Z_n)_{n \in \mathbb{N}}:=X(t_n)$, where $t_n$ is the time of the $n$-th jump, are given by
\begin{equation}
\label{eq:emb}
p(x,x{+}\nu_j;c):=\frac{a_j(x;c)}{a_0(x;c)}\, , \quad  j{=}1,2,...,J \,  ,
\end{equation}
for $X(t_n){=}x \rightarrow X(t_{n+1}){=}x{+}\nu_j$, assuming $a_0(x;c){>}0$.
\end{rem}

\begin{rem}
\label{rem:log}
In many applications, model parameters differ by
several orders of magnitude, so we perform relative parameter perturbations. This is done by perturbing the logarithm of the parameters. Using the chain rule on $\nabla_{\log c} p^c$ we obtain 
\begin{equation}
\left(\pFIM{P_{0:T}^{\log c}}\right)_{k,l} =  c_k c_l \left( \pFIM{P_{0:T}^c} \right)_{k,l}\, , \quad k,l=1,...,K \ .
\end{equation}

Note that \eqref{eq:RE_FIM} continues to be valid for the logarithmic scale. From now, we use relative parameter perturbations, i.e., we use $\pFIM{P_{0:T}^c}$ to denote $\pFIM{P_{0:T}^{\log c}}$.
\end{rem}

\subsection{Scalable information-based sensitivity indexes}
\label{sec:SB}
In this section we discuss a scalable information-based inequality that connects observables of the system with the pFIM, and we present a reliable sample-free implementation. 

Relative entropy provides a rigorous and computationally tractable  methodology for parameter sensitivity analysis of complex, stochastic dynamical systems focusing on the sensitivity of the entire probability
distribution.
However, in most simulations of biochemical networks, the main interest are 
observables such as population means, variances or time averages, as well as autocorrelations or extinction times. Therefore, it is
reasonable to connect parameter sensitivities with observables.
Given a pathwise distribution $P^c_{0:T}$, let 
\begin{equation}
D_{F,k} := \frac{\partial \expt{F}{P^c_{0:T}}}{\partial c_k} \, ,
\end{equation}
be the classical local sensitivity index, then the following information-based bound
\begin{equation}
\label{eq:SBo}
\abs{D_{F,k}} \leq \sqrt{\var{F}{P^c_{0:T}}}\sqrt{(\pFIM{P^c_{0:T}})_{k,k}} \, ,
\end{equation}
can be obtained by rearranging the generalized Cramer-Rao bound for estimators of the form $$F((x_i)_{i=1}^T)=\frac{1}{T}\sum_{i=1}^T f(x_i) \Dt_i \, ,$$ with $f$ a suitable observable. Here $(\pFIM{P^c_{0:T}})_{k,k}$ denotes the $k$-th diagonal element of the pathwise Fisher Information Matrix (pFIM).
For a given observable, this inequality can be used as an indicator that allows to classify --even in the presence of a very high-dimensional parameter space-- insensitive parameters, in the sense that a small pFIM diagonal value suggests relatively low SIs. In the same way, large pFIM values suggest high SIs. However, notice that large or small pFIM values does not imply sensitive or insensitive parameters respectively. We refer to \cite{PKV:2013,PLOS:AKP,DKPP16} for further details.
In this work, we focus instead on the following normalized local sensitivity index
\begin{equation}
\mathcal{D}_{F,k}:= \frac{\frac{\partial \expt{F}{P^c_{0:T}}}{\partial c_k}}{\sqrt{\var{F}{P^c_{0:T}}}} \, ,
\end{equation}
which has the advantage of capturing how ``noisy'' each observable is, and therefore, weights the sensitivity accordingly, meaning that perturbations in the expected value of $F$ are less significant when its variance is large. Then we readily have the sensitivity bound
\begin{equation}
\label{eq:SB}
\abs{\mathcal{D}_{F,k}} \leq \sqrt{(\pFIM{P^c_{0:T}})_{k,k}} \, .
\end{equation}
In this work, since we focus on model reduction, we mainly consider the observable $f(x)=x$, that is, $F$ is the time average of the state variables (species)
\begin{equation}
\label{eq:time_avg}
F((x_i)_{i=1}^T):=\frac{1}{T}\sum_{i=1}^T x_i \Dt_i \, .
\end{equation} 

Despite of this, further observables may be considered in the last step of our method, in which the reduced model may be augmented with additional parameters depending on particular quantities of interest. 

\subsection{Mean-field estimation of the pFIM}
\label{sec:MF_approx}
Here we discuss a classical mean-field type of approximation to avoid sampling of \eqref{eq:iFIM}. 
For such cases in which sampling is not feasible, either due the high dimension of the system or due to the fact that the available data is limited, a deterministic approximation is usually the only alternative. The linear noise approximation can be used to 
efficiently compute the pFIM, while maintaining controlled bias
in the statistical estimators. First notice that, under certain conditions and properly scaling the state variables, the CLE can be  written as
\begin{equation}\label{eq:approx_MFE}
Y(t) = z(t) + \eta \xi(t)
\end{equation}
where $z(t)$ is the deterministic mean-field part that satisfies \eqref{eq:rrODE}, $\xi(t)$ is a zero-mean external noise process and  $\eta$ is the amplitude of this stochastic term, which is proportional to the inverse square root of the reactant
populations \cite{Kurtz:72, Kurtz:81, Gillespie:00,vanKampen:06,PKV:2013}. Thus, for large populations, the dominant part of the stochastic process is the deterministic term
whose dynamics are governed by the ODE system \eqref{eq:rrODE}.
Assuming the process starts at a fixed value, the diagonal elements of the pFIM \eqref{eq:pFIM_SRN} are approximated using
\VIZ{eq:approx_MFE}
to get
\begin{equation}
\begin{aligned}
(\pFIM{P^c_{0:T}})_{k,k} & \approx \sum_{i=1}^T \sum_{j=1}^J a_j(z_i;c) \left(\frac{\partial \log a_j(z_i;c)}{\partial_{c_k}}\right)^2  \Dt_i \, .
\end{aligned}
\end{equation}
Here the sequence $(z_i)_{i=0}^T$ corresponds to the mean-field part of \eqref{eq:approx_MFE}. 
This approximation is usually valid for large populations and usually cannot capture complex dynamics
such as bistability nor exit times nor rare events. Here we use the time series data \eqref{eq:data} in place of the sequence $(z_i)_{i=0}^T$.

Due to the parametric sparsity usually found in biochemical models, i.e., \eqref{eq:sp_pars}, we can assume that the  computational work required to compute these quantities is of the order $O(K)$ per timestep. 
In contrast, most of the sensitivity-based reduction methods require to solve the classical parametric sensitivity adjoint system \eqref{eq:adjoint}, which requires work on the order of $O((K{+}1){\times}d)$, where $d$ is the number of state variables, as presented in the next remark.

\begin{rem}[Classic parametric sensitivity system]
\label{rem:adjoint}
The classic parametric sensitivity analysis for deterministic reaction networks consists of solving the coupled system: 
\begin{equation}
\label{eq:adjoint}
\left\{ \!
\begin{aligned}
dz  &=  b(z;c)dt \\
ds_k  &= \frac{\partial b}{\partial z} s_k dt + \frac{\partial b}{\partial s_k}dt \, , \quad k{=}1,...,K \, , \,\, t \in \mathbb{R}_+ \, ,
\end{aligned}
\right. \, ,
\end{equation}
where $b(z;c) = \nu a(z;c)$, $z \in \mathbb{R}^d$ are the state variables, and $c \in \mathbb{R}^K$ is the parameter vector, for the sensitivity indexes 
\begin{equation}
\label{eq:SIs_adj}
s_k := \frac{\partial z}{\partial c_k} \, , \quad k{=}1,...,K \, .
\end{equation}
The first line of the coupled system \eqref{eq:adjoint} is the reaction-rate ODE system, which approximates $\expt{F}{P_{0:T}}$, i.e. the expected value of a pathwise quantity of interest of the form $\frac{1}{T}\int_0^T f(z(s))ds$, where $f$ is a suitable observable. The computational work required for integrating this system is of the order of $(K{+}1)\times d$  per timestep, which renders this method unfeasible for high-dimensional systems. The alternative Green's function method allows to solve this system using  work of order of $d \times d$. For $d\gg 1$ and $K\gg 1$ this method turns to be impractical especially taking into account that the system is usually stiff.

\end{rem}

\section{Variational model reduction in path-space}
\label{sec:var_inf}
In this section we present one of the main results of this work: a simplified loss function that allows to find a solution to the variational inference problem \eqref{eq:opt_pr} in the macroscopic space state by using available microscopic data time series \eqref{eq:data}.

\subsection{Variational inference and coarse-graining}
Here we briefly recap the variational inference approach. Variational inference is a machine learning technique for approximating probability distributions \cite{Jordan1999, MAL-001}, and is widely used to approximate posterior densities in Bayesian models, as an alternative strategy to Monte Carlo sampling. Compared to Monte Carlo sampling, tends to be faster and more scalable to large datasets \cite{blei17}. Instead of using sampling, the main idea of variational inference is to apply optimization, and therefore for complex models it provides a relevant alternative approach. 

In this work, we aim to construct a suitable reduced  parametric model and to find the optimal set of parameters that minimizes the information loss in an efficient and principled manner. This is achieved by means of a loss function based on the following relative entropy minimization
\begin{equation}
\label{eq:opt_pr}
\theta^* := \theta_{0:T}^* := \arg \min_{\theta \in \Theta} \re{P_{0:T}}{Q^\theta_{0:T}} \, .
\end{equation}

We seek for this optimal solution $\theta^*$ based on the following first order optimality condition,
\begin{equation*}\label{eq:opt}
\nabla_\theta \re{P_{0:T}}{Q^{\theta^*}_{0:T}}  = 0 \, ,
\end{equation*}
whose solutions reveal the local optima of the relative entropy in the microscopic state space.
Notice that if the relative entropy is a strictly convex function then there is a unique global minimum. This clearly  depends  on the choice of 
the parameterized model (see \cite{HKKP16}). For example, if the parameterized model depends linearly on $\theta$ then there is a unique global minimum and the problem reduces to solve a linear system.
In this work, the parameterized model is constructed by using the original propensity functions, assumed to be an analytical function.

In what follows, we present two such loss functions; the first one is equivalent to solving the problem \eqref{eq:opt_pr} (see Theorem \ref{thm:argmin}), and the second one is an upper bound but is computationally less demanding.

\subsection{Building the approximating parametric family}
Since the main goal of this work is to construct a reduced model for a high-dimensional RN, we start by considering the application of the linear map  $\CG:\mathbb{R}^d \rightarrow \mathbb{R}^{\bar d}$. 
Examples of these maps used in the literature for molecular systems (called coarse-graining maps) include the mapping to the centers of mass of groups of particles, or a projection on a specific set of particles, among others. Without losing generality, in this work we consider the following map. For $x\in \mathbb{R}^d$ let $\CG$ be such that
\begin{equation}
\label{eq:var_map}
x \mapsto (\CG x, \CG^{\comp} x) = (\bar x, \hat x)
\end{equation}
where $\bar x \in \mathbb{R}^{\bar d}$ is the macroscopic state, and $\hat x \in \mathbb{R}^{d - \bar d}$. If we denote with the same symbol $\CG \in \mathbb{R}^{\bar d\times d}$ the matrix representation of the map, then $\CG$ together with $\CG^\comp$ form  
a permutation matrix.
This $\CG$ map determines the macroscopic state, $\bar x$, as a function of the full model state $x$. Further details of this map can be found on Section \ref{sec:sel_var}.

Assuming the data is generated by means of the Euler discretization of the CLE (see \eqref{eq:data}) is clear that $x_i$ is a multivariate Gaussian random variable for $i{=}1,2,...,T$. 
We can then split the transition density in the microscopic state space (see \eqref{eq:pdist_CLE}) as
$$
p(x,x') = p^{(1)}(x,\bar x')p^{(2)}(x,\bar x'|\hat x')\, , \quad \text{ for } x,x' \in \mathbb{R}^d , \,  \bar x' {\in} \mathbb{R}^{\bar d}\, ,
$$
and the associated microscopic parameterized transition density (see \eqref{eq:ppdist_CLE}) as 
\begin{equation}
\label{eq:micro_dens}
q^\theta(x,x') = r(x'|\bar x') p^\theta(\bar x, \bar x')\, , \quad \text{ for } x,x'{ \in} \mathbb{R}^d, \, \bar x' {\in} \mathbb{R}^{\bar d} \, ,
\end{equation}
where $r$ is a reconstruction density independent of $\theta$, and $p^\theta$ is the  density of the reduced model. This reconstruction density essentially recovers the lost degrees of freedom and, together with $p^\theta$, approximates the microscopic density. This reconstruction density serves as an auxiliary tool that connects the reduced dynamics with the full dynamics on the same space. The important fact of this density is that it is independent of $\theta$ and can be arbitrarily chosen. In this work, we do not focus on this auxiliary density, since it is not relevant in our optimization procedure. 
We have then the original transition density $p$ which is associated with the path distribution $P_{0:T}$. For the Euler CLE case, the forumla is given by equation \eqref{eq:pdist_CLE}. Then, we have the parameterized approximation to the original microscopic path distribution $P_{0:T}$, i.e., $Q_{0:T}^\theta$, which is associated to the transition density $q^\theta$. For the Euler CLE case, the formula is given by equation \eqref{eq:ppdist_CLE}. Finally, we have the transition density $p^\theta$ (see \eqref{eq:micro_dens}), which is associated with the path distribution of the reduced model, which lives in the macroscopic space. The main goal of this work is to construct a reduced model which corresponds to this transition density.
In Section \ref{sec:MR} we show how to construct this reduced model, 
which can be described by a $\bar d$-dimensional time dependent state vector $\bar x =\CG x$, and $\bar J$ reaction channels $((\bar \nu_j, \bar a_j))_{j=1}^{\bar J}$, where $\bar \nu_j$ is a $\bar d$-dimensional state change vector properly constructed by considering the full model stoichiometry, and $\bar a_j= \bar a_j(\bar x;\theta)$, $j{=}1,2,...,\bar J$ its propensity functions, with $\theta \in \Theta$ a $\bar K$-dimensional vector of model parameters. The definition of $\bar \nu$ is given in \eqref{eq:bar_nu} while th definition of $\bar a$ is given in \eqref{eq:bar_a}.

In this work, propensity functions in the reduced model have the same functional form of the corresponding full model propensity functions. Two main reasons apply. First, from the modeling point of view, usually it is desirable that the functional form of the propensity functions belong to the same parametric space as the full model. In the biochemical literature, propensity functions usually include mass-action and Michaelis-menten type of expressions, but may in general include closed form expressions. As a consequence the dependence of the propensity functions on the parameters is usually non-linear. Second, in the context of biochemical reaction networks, each state variable is associated with a different species and in some cases with physical location. Therefore, even a linear transformation may be difficult to interpret and thus render the reduction meaningless. 

\subsection{Loss function for time-series data training}
\label{sec:loss_fct}
The following theorem allows to connect the optimization problem \eqref{eq:opt_pr} with the macroscopic space of state variables. This theorem also provides a  pathwise relative entropy representation for the Euler CLE case and shows  computable quantities that can be trained by means of the microscopic data.

\begin{thm}
\label{thm:argmin}
Let $P_{0:T}$ be the path distribution of the Euler CLE approximation of a RN $((\nu_j,a_j))_{j=1}^J$, and $Q^\theta_{0:T}$ be the path distribution of a parameterized approximation. Let $\CG: \mathbb{R}^d \rightarrow \mathbb{R}^{\bar d}$ be a linear map as in \eqref{eq:var_map}. Then, we have 
\begin{equation}
\label{eq:opt_pple}
\arg \min _{\theta \in \Theta} \re{P_{0:T}}{Q^\theta_{0:T}} = \arg \min_{\theta \in\Theta} \left[ R_{0:T}(\theta) + M_{0:T}(\theta)\right]
\end{equation}
where $$R_{0:T}(\theta) := \sum_{i=1}^T R_i(\theta), \quad M_{0:T}(\theta) := \sum_{i=1}^T M_i(\theta) \Dt_i\,, \quad \Dt_i {>} 0 \, ,$$ and
\begin{align}
\label{eq:ayb}
R_i(\theta) &:= 
\frac{1}{2} \expt{ \tr{\CG \Sigma(x) \CG^{tr} \bar \Sigma^{-1}(\bar x;\theta)} - \log \det \left(  \CG \Sigma(x) \CG^{tr} \bar \Sigma^{-1}(\bar x;\theta) \right) }{\nu_{i-1}} \,  ,
\\
\nonumber M_i(\theta) &:=  \frac{1}{2}\expt{ (\bar b(\bar x;\theta) -\CG b(x))^{tr} \bar \Sigma^{-1}(\bar x;\theta)(\bar b(\bar x;\theta) -\CG b(x))}{\nu_{i-1}}\, .
\end{align}
Here $\Sigma(x)=\sigma(x)\sigma^{tr}(x)$, where $\sigma(x) = \nu \sqrt{\diag (a(x))}$ and $b(x) = \nu a(x)$ are the diffusion and drift coefficients of the RN $((\nu_j,a_j))_{j=1}^J$ respectively. Moreover, $\bar \Sigma(\bar x;\theta)=\bar \sigma(\bar x;\theta)\bar \sigma^{tr}(\bar x;\theta)$ with $\bar \sigma$ and $\bar b$ the reduced RN $((\bar \nu_j,\bar a_j))_{j=1}^{\bar J}$ diffusion and drift coefficients respectively.
\end{thm}

The proof of this Theorem is presented in the Appendix. Notice that this result allows to optimize the parameter vector $\theta$ by considering the drift associated to the reduced model, $\bar b(\bar x;\theta) = \bar \nu \bar a(\bar x;\theta)$, and the projected microscopic drift, $\CG \nu a$, instead of the microscopic drift, since the pathwise quantities $R_{0:T}(\theta)$ and $M_{0:T}(\theta)$ can be computed in the macroscopic state space. 

The macroscopic pathwise quantity 
\begin{equation}
\label{eq:lossf}
F_{0:T}(\theta) : =R_{0:T}(\theta) + M_{0:T}(\theta)
\end{equation}
can be then interpreted as a pathwise loss function whose minimization allows to determine the optimal value of the approximating distribution set of parameters. 
We notice that the $R_{0:T}$ term acts as a penalization term for $M_{0:T}$ by the  discrepancy of $\bar \Sigma$ from $\CG \Sigma \CG^{tr}$. 

\emph{A simplified loss function.}
Next, we derive a simplified loss function from \eqref{eq:lossf} that substantially simplifies the numerical method search space, and therefore the associated  computational work to find its solution. First we notice that 
\begin{equation}
\label{eq:ineq_R}
R_i(\theta)  \geq \frac{\bar d}{2}, \quad \text{ for } \theta \in \Theta \, ,
\end{equation}
with equality if and only if $\bar \Sigma = \CG\Sigma \CG^{tr}$. This is proved in Proposition \ref{prop:ineq_R} (see Appendix).
Replacing in \eqref{eq:lossf} $\bar \Sigma$ by $\CG \Sigma \CG^{tr}$ we obtain
\begin{align*}
F_{0:T}(\theta)\Big|_{\bar \Sigma = \CG \Sigma \CG^{tr}} & =R_{0:T}(\theta) + M_{0:T}(\theta)\Big|_{\bar \Sigma = \CG \Sigma \CG^{tr}} \\
&= T \frac{\bar d}{2} + \sum_{i=1}^T M_i(\theta)\Big|_{\bar \Sigma = \CG \Sigma \CG^{tr}} \Dt_i  \\
&=  T \frac{\bar d}{2} + \frac{1}{2}  \sum_{i=1}^T \expt{\norm{\bar b(\bar x; \theta - \CG b(x))}^2_{\CG}}{\nu_{i-1}} \Dt_i \, ,
\end{align*}
where $\norm{\cdot}^2_\CG$ is defined as  
\begin{equation}
\label{eq:norm}
\norm{z}^2_\CG := z^{tr} (\CG \Sigma \CG^{tr})^{-1} z \, .
\end{equation}

Finally, ignoring the constant $T \frac{\bar d}{2}$ we have the following simplified loss functional
\begin{equation}
\label{eq:losse}
E_{0:T}(\theta) :=  \frac{1}{2}  \sum_{i=1}^T \expt{\norm{\bar b(\bar x; \theta) - \CG b(x))}^2_{\CG}}{\nu_{i-1}} \Dt_i \, .
\end{equation}

By applying the same approximation as in Section \ref{sec:MF_approx}, i.e. \eqref{eq:approx_MFE} we obtain
\begin{equation}
\label{eq:losse_approx}
E_{0:T}(\theta) \approx \hat{E}_{0:T}(\theta) := \frac{1}{2} \sum_{i=1}^T \norm{\bar b(\CG z_i;\theta)-\CG b(z_i)}^2_\CG \Dt_i \, .
\end{equation}
As before, the sequence $(z_i)_{i=0}^T$ corresponds to the mean-field part of \eqref{eq:approx_MFE}, replaced by the time series data \eqref{eq:data}. 

We then solve the following optimization problem 
\begin{equation}
\label{eq:opt_losse}
\min_{\theta \in \Theta} \hat{E}_{0:T}(\theta) \, ,
\end{equation}
by considering steepest descent and Newton-Raphson
type of methods, recalling that the Hessian of the relative entropy is the Fisher Information Matrix. In this work, we do not explore further this topic, and we use instead standard optimization packages to find a numerical solution, in particular MATLAB 2016b.

We emphasize that the CLE formulation is only required to derive the loss function which allows to find an approximate solution of the variational problem \eqref{eq:opt_pr}. The loss function is suitable to be applied in non-Langevin regimes, as we show in the pure jump stochastic models in Section \ref{sec:stoc_exp}.

\begin{rem}
The loss functions \eqref{eq:lossf} and \eqref{eq:opt_losse}  may be regularized by using penalization terms like Tichonov or $L^1$ to simplify the optimization procedure or to obtain a more parsimonious model. For example we can solve
\begin{equation}\min_{\theta \in \Theta}  \frac{1}{2} \sum_{i=1}^T \norm{\bar b(\CG z_i;\theta)-\CG b(z_i)}^2_\CG \Dt_i + \lambda \norm{\theta}^2_\CG \, ,
\end{equation}
where $\lambda$ is a parameter to determine (see \cite{efron2016computer, aster2005parameter, kaipio2005statistical} for additional references).
\end{rem}

\begin{rem}
One of the computational novelties of our method lies in the derivation and implementation of a pathwise force matching condition (termed Dynamic Force Matching) such as \eqref{eq:opt_losse} and Theorem \ref{thm:argmin}, where the norm is now $\norm{\cdot}^2_\CG$ instead of the Euclidean one. The classic force matching method can be viewed as a particular case (see for instance \cite{annurev-biophys,Izvekov} and references therein). \end{rem}

\begin{rem}
The optimization principle in \eqref{eq:opt_pple} can be further extended to obtain an improved time-dependent optimal parameterization, which is obviously more demanding in terms of computational work. That is, obtain  
$$
\theta_{0:T}^* = (\theta^*_i)_{i=1}^T, \quad i{=}1,2,...,T \, ,
$$
by optimizing 
$$
\arg \min_{(\theta_i)_{i=1}^T }\frac{1}{2}\sum_{i=1}^T\expt{\norm{\bar b(\bar x_i;\theta_i) - \CG b(x_i))}^2_\CG} {\nu_{i-1}}\Dt_i \, .
$$
\end{rem}


\begin{rem}
In the case that the matrix $\bar \Sigma$ is singular, the conclusions of Theorem \ref{thm:argmin} are still valid, by means of the Moore-Penrose generalized inverse instead of the inverse, and the pseudo-determinant instead of the determinant in \eqref{eq:ayb}. We refer to \cite{Wilkinson:12} and references therein for additional details.
\end{rem}

In the next Section we present a procedure to construct a reduced model, i.e., how to choose $\bar \nu$ and $\bar a$, such that $\bar b = \bar \nu \bar a$ and $\bar \sigma = \bar \nu \sqrt{ \diag(\bar a)}$ are the drift and diffusion terms of the reduced model.

\begin{rem}
The drift of the reduced model, $\bar b$, may be alternatively constructed by using the fact that most physical models of the form have only a few relevant terms that define the equations of motion of the system (i.e. the RHS of the system $\dot{x}(t)=f(x)$). These methods are called sparse identification of nonlinear dynamics (see for instance \cite{Brunton3932, Schmidt81}).
Since the construction can be considered on reduced space by virtue of Theorem \ref{thm:argmin}, this certainly mitigates one of the main drawbacks of the aforementioned methods, that is, when data is high-dimensional and therefore the state dimension of the dictionary of candidate functions is usually prohibitively large.
\end{rem}

\section{Model reduction procedure}
\label{sec:MR}
In this section we present an information-based model reduction procedure for high-dimensional reaction networks. The goal is to 
construct a suitable reduced model in a principled and efficient manner using the tools developed in Section \ref{sec:SB}. This reduction is achieved in several steps as follows. In Section \ref{sec:summ_red} we present a summary of the steps, in Section \ref{sec:sel_pars} we show how to identify the most sensitive parameters, in sections \ref{sec:sens_chans} and \ref{sec:sel_var} to use the stoichiometry of the network to choose the associated variables and reaction channels. In Section \ref{sec:model_red} we show how to construct a paramterized family of reduced models to then in Section \ref{sec:fitting} how to find an optimal model. Finally, in Section \ref{sec:augment} we show how to validate and iteratively improve the reduced model obtained in the previous steps.

\subsection{Summary of the reduction method}
\label{sec:summ_red}
The reduction method presented here can be described in 6 steps as follows.

\begin{inparadesc}
\item \emph{Step 1. Selecting parameters and reaction channels.} The reduction method starts with the estimation of the full model pFIM \eqref{eq:pFIM} by using the available time series \eqref{eq:data} to determine the parameters that accumulate at least $\kappa\%$ (usually $97\%$ or $99\%$) of the total information. This criterion determines a set of parameter indexes $\mathcal{P} \subseteq \{1,2,...,K\}$ to include in the reduced model and later train (in Step 4). By means of the $\varphi$ map, \eqref{eq:varphi}, we determine a set of reaction channel indexes to include in the reduced model, $\mathcal{J}_\mathcal{P}$. This set is if indexes correspond to reactions that include at least one sensitive parameter as per the pFIM diagonal. Finally, notice that  parameter map $\CGp$ \eqref{eq:defvargamma} is also determined by $\mathcal{P}$.

\item \emph{Step 2. Selecting variables.} Given the set of reaction channels $\mathcal{J}_\mathcal{P}$ from Step 1, define a state variable map $\CG$ that contains every species that takes part in the stoichiometry of any reaction $j \in \mathcal{J}_\mathcal{P}$. The set of species (or state variables) that satisfies the aforementioned relation, are selected to be the species of the reduced model, and denoted by $\mathcal{S}_\mathcal{P}$. 

\item \emph{Step 3. Construction of a parameterized family of reduced models.} Given $\mathcal{J}_\mathcal{P}$ and $\mathcal{S}_\mathcal{P}$, construct a family of candidate reduced models of the form $((\bar \nu_j, \bar a_j))_{j=1}^{\bar{J}}$, parameterized by the sensitive parameters of the full model in the sense \eqref{eq:SB}, and taking into account the stoichiometry structure.  This step includes the definition of $\bar \nu_j$ and $\bar a_j = \bar a_j(\bar x;\theta)$, where $\theta$ is the parameter vector of the reduced model. 

\item \emph{Step 4. Model training.} By means of the loss function \eqref{eq:losse_approx} the reduced model parameters, denoted by $\theta$, are fitted to the time series data \eqref{eq:data}. This optimization also provides a goodness-of-fit in terms of entropy loss.

\item \emph{Step 5. Validation.} Determine if further reaction channels should be added to the reduced model, by means of computing a suitable distance on the time series data \eqref{eq:data}. 

\item \emph{Step 6. Iteration.} Finally, previous steps can be subsequently iterated to obtain either alternative or further reduced models that may provide a better fit.
\end{inparadesc}

The model reduction process (Step 1--6) is depicted in Figure \ref{fig:method}.

\begin{figure}[h!]
\centering
\begin{tikzpicture}[IBMRRN, node distance=1.5cm]
\node (model) [io] {Full Model: $((\nu_j,a_j))_{j=1}^J$};
\node (data) [io, left of=model, xshift=5.5cm] {Data: $(x_i)_{i=0}^T$ \eqref{eq:data}};
\node (pFIM) [process, below of=data] {Compute diag(pFIM) \eqref{eq:pFIM}};
\node (chans) [process, below of=pFIM] {Determine $\mathcal{J}_{\mathcal{P}}$ \eqref{eq:sens_chans}};
\node (info) [io, left of=chans, xshift=-4cm] {Information  threshold \eqref{eq:pFIMthr}};
\node (maps) [process, below of=chans] {Construct $\CGp$ \eqref{eq:defvargamma} and $\CG$ \eqref{eq:defpi}};
\node (nu) [io, below of=info] {Stoichiometry: $\nu_{in}$, $\nu_{out}$};

\node (RM) [process, below of=maps] {{\small Construct parameterized family of reduced models:} $((\bar \nu_j,\bar a_j(\cdot;\theta)))_{j=1}^{\bar J}$};
\node (train) [process, below of=RM] {Loss function \eqref{eq:losse_approx} training to get $\theta^*$};

\node (dec1) [decision, below of=train] {\tiny{Reduction ok?}};
\node (tol) [io, right of=dec1, xshift=2cm] {user defined error TOL};

\node (end) [io, below of=dec1] {Output: $((\bar \nu_j,\bar a_j(\cdot;\theta^*)))_{j=1}^{\bar J}$};

\draw [arrow] (model) -- (pFIM);
\draw [arrow] (data) -- (pFIM);
\draw [arrow] (pFIM) -- (chans);
\draw [arrow] (info) -- (chans);
\draw [arrow] (nu) -- (maps);
\draw [arrow] (chans) -- (maps);
\draw [arrow] (maps) -- (RM);
\draw [arrow] (RM) -- (train);
\draw [arrow] (train) -- (dec1);
\draw [arrow] (tol) -- (dec1);
\draw [arrow] (dec1) -- node[right] {yes} (end);

\draw [arrow] (dec1) node[above] {\hspace{-85pt} no} -| node{} ++(-8,1) node[above] {\hspace{45pt} iterate} |- (info) ;

\end{tikzpicture}
\caption{\label{fig:method} Main steps in the information-based model reduction method for  high-dimensional reaction networks method (see  Section \ref{sec:summ_red}).}
\end{figure}
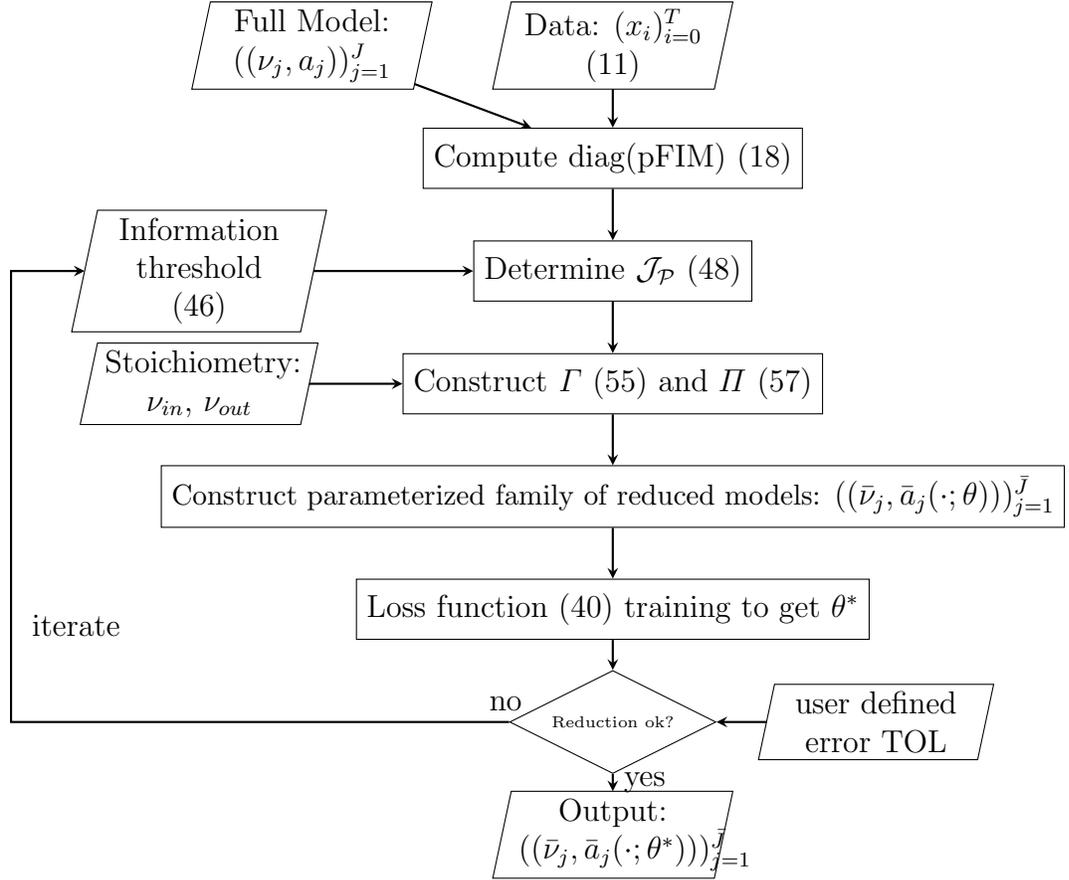

\subsection{Parameter selection (step 1)}
\label{sec:sel_pars}

By means of the information-based bound \eqref{eq:SB}, we focus our sensitivity analysis of the full model on
\begin{equation}
\label{eq:dk}
\xi_k := (\pFIM{P_{0:T}})_{k,k} \, ,
\end{equation}
which are the diagonal elements of the pFIM \eqref{eq:pFIM} of the full model, which has pathwise distribution $P_{0:T}$ \eqref{eq:pRE} and parameter vector $c$. This corresponds to perturbations of full model parameters in the canonical directions, where $\xi_k$ corresponds to parameter $c_k$.

Let $\xi_{\sigma(1)}>\xi_{\sigma(2)}>...>\xi_{\sigma(K)}$ be the ordering from highest to lowest of the diagonal elements $\xi_k$ where $\sigma$ denotes a sorting permutation of the indices. So $\sigma(1)$ is the index of the parameter with the largest diagonal value in the pFIM. In the case that the pFIM is diagonal, this is equivalent to the sorted eigenvalues of the pFIM.
Given the ordered set of diagonal elements, $\{\xi_{\sigma(1)},\xi_{\sigma(2)},...,\xi_{\sigma( K)}\}$, we say that the percentage of information contained in the corresponding parameters $\{\sigma(1),\sigma(2),...,\sigma(\bar K)\}$ is 
\begin{equation}
\label{eq:pFIMpctg}
\Xi_{\bar K}:=\frac{1}{\tr{\pFIM{P_{0:T}}}}\sum_{\ell=1}^{\bar K} \xi_{\sigma(\ell)} \, ,
\end{equation}
where $\xi_k$ is defined in \eqref{eq:dk}.

Now define 
\begin{equation}
\label{eq:set_p}
\mathcal{P} :=\{\sigma(1),\sigma(2),...,\sigma(\bar K)\}\subseteq \{1,2,...,K\}
\end{equation} \
as the set of parameter indexes intended to keep in the reduced model, such that
\begin{equation}
\label{eq:pFIMthr}
\Xi_{\bar K}\geq \kappa \, ,
\end{equation}
where $\kappa$ is a user defined threshold that represents $(\kappa\times 100)\%$ of the total information preserved in the reduced model as per the trace of the full model pFIM. We collectively call these parameters as \emph{pFIM-sensitive parameters}. In the examples studied in Section \ref{sec:exp}, keeping the parameters  whose pFIM values accumulate at least $95\%$ of the total sum is enough to include all sensitive parameters for the time average of the species, in the sense of \eqref{eq:adjoint}. 

More precisely, the reduction of the parameter space is considered as 
the application of a full rank linear map $\CGp : \mathbb{R}^K \rightarrow \mathbb{R}^{\bar K}$ such that
\begin{equation}
\label{eq:par_map}
c \mapsto (\CGp c, \CGp^{\comp,1} c, \CGp^{\comp,2} c) = (\theta, u, \tilde c) \, ,
\end{equation}
where $c \in \mathbb{R}^K$ is the full model parameter vector, $\theta \in \Theta := \mathbb{R}^{\bar K}$ is the vector of sensitive parameters, $u \in \mathbb{R}^{\bar K'}$ is the vector of insensitive parameters but still relevant for the reduction as explained below, and $\tilde c$ is the vector of irrelevant parameters. 

This map determines a partition in the full model parameter space, identifying sensitive parameters that will be included as parameters in the reduced model i.e., $\theta = \CGp c$, insensitive but relevant parameters that will be included as constants in the reduced model i.e., $u=\CGp^{\comp,1}c$, and irrelevant parameters that will be eliminated from the reduced model, i.e., $\tilde c = \CGp^{\comp,2}c$. 

The distinction between $\theta$ and $u$ is relevant when fitting the reduced model to the given time series data \eqref{eq:data} (see Step 4). This fitting is achieved by minimizing the information loss between the full model and the parameterized reduced model, and the number of dimensions of this optimization problem is given by the number of parameters of the reduced model. The linear map $\CGp$ is depicted in Figure \ref{fig:CGp}.

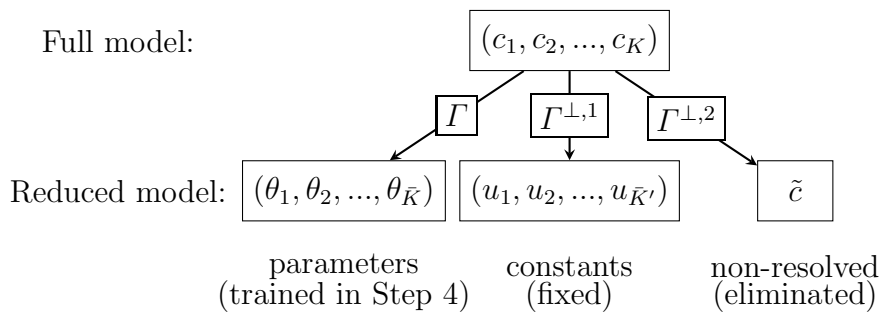
\begin{figure}[h!]
\centering
\begin{tikzpicture}[node distance=2cm]
\node (pars) [process] {$(c_1,c_2,...,c_K)$};

\node (theta) [process, below of=pars, xshift=-3cm] {$(\theta_1,\theta_2,...,\theta_{\bar K})$};
\node (u) [process, below of=pars] {$(u_1,u_2,...,u_{\bar K'})$};
\node (ct) [process, below of=pars, xshift=3cm] {$\tilde c$};

\node (c1) [below of=theta, yshift=1cm,text centered] {parameters};
\node [below of=c1,  yshift=1.6cm, text centered] {(trained in Step 4)};
\node (c2) [below of=u, yshift=1cm,text centered] {constants};
\node (c22) [below of=c2,  yshift=1.6cm, text centered] {(fixed)};
\node (c3) [below of=ct, yshift=1cm,text centered] {non-resolved};
\node (c22) [below of=c3,  yshift=1.6cm, text centered] {(eliminated)};

\node [left of=pars, xshift=-4cm] {Full model:};
\node [left of=theta,  xshift=-1cm] {Reduced model:};

\draw [arrow] (pars) -- node {\fcolorbox{black}{white}{$\CGp$}} (theta);
\draw [arrow] (pars) -- node {\fcolorbox{black}{white}{$\CGp^{\comp,1}$}} (u);
\draw [arrow] (pars) -- node {\fcolorbox{black}{white}{$\CGp^{\comp,2}$}}  (ct);

\end{tikzpicture}
\caption{\label{fig:CGp}Schematic representation of the linear $\CGp$ mapping, from full model parameters to reduced model parameters.}
\end{figure}



\begin{rem}
Further information can be extracted from the pFIM \eqref{eq:pFIM}, for example, the
spectral analysis reveals the most and the least sensitive directions
of the system around $c$, which corresponds to the eigenvector
with the largest/smaller smaller eigenvalue. 
In the same direction, parameter
identifiability can also be studied. Parameter identifiability is satisfied when all the eigenvalues of the pFIM
are above a certain threshold (see \cite{Komorowski:11}). For example, when the determinant one of the blocks is zero then the corresponding linear combinations of  parameters are
non-identifiable. 
\end{rem}

\subsection{Reaction channel selection (Step 1)}
\label{sec:sens_chans}
Here we describe how reaction channels are selected in order to keep the most sensitive ones in the sense explained below. Notice that this is not the only way to select the relevant reaction channels. Given a set of sensitive parameter indexes $\mathcal{P}$ (eq. \eqref{eq:set_p}),
the set of sensitive reaction channels is defined as the full model reaction channels that contain at least one sensitive parameter (indexed by $\mathcal{P}$), i.e.,
\begin{equation}
\label{eq:sens_chans}
\mathcal{J}_{\mathcal{P}} :=\bigcup_{k \in \mathcal{P}}\varphi(k)\, ,
\end{equation} 
where $\mathcal{J}_\mathcal{P}$ is the set of indices of reaction channels that depends on the set of parameter indexes $\mathcal{P}$, with $|\mathcal{J}_{\mathcal{P}}|=\bar J$. That is, if $j \in \mathcal{J}_\mathcal{P}$ then exists $k \in \mathcal{P}$ such that the propensity function $a_j$ explicitly depends on the parameter $c_k$. The formal definition of the reduced model propensity function is given in Section \ref{sec:model_red}, equation \eqref{eq:bar_a}.

\subsection{Selecting variables (Step 2)}
\label{sec:sel_var}
Here we describe how the matrix representation of the state variables full rank linear map $\CG : \mathbb{R}^d \rightarrow \mathbb{R}^{\bar{d}}$, such that
\begin{equation}
\label{eq:var_map_det}
x \mapsto (\CG x, \CG^{\comp,1} x, \CG^{\comp,2} x) = (\bar x, \bar y, \tilde x) \, ,
\end{equation}
is constructed. Here $\bar x \in \mathbb{R}^{\bar d}$ is the state space of the reduced model, i.e., the microscopic variables that take part in the stoichiometry of the selected reaction channels; $\bar y \in \mathbb{R}^{\bar d '}$ is the set of microscopic state variables that are not part of the stoichiometry of the selected reaction channels; and $\tilde x \in \mathbb{R}^{d-\bar d -\bar d'}$ is the set of microscopic state variables that are not included in the reduced model. 

Given the set of sensitive reaction channels $\mathcal{J}_\mathcal{P}$ (eq. \eqref{eq:sens_chans}) , the map $\CG$ is in fact a projection on a smaller number of variables $\bar x$, corresponding to the species that take part in the stoichiometry of the sensitive reaction channels. In this way, the state variables included in the reduced model are determined by $\mathcal{J}_{\mathcal{P}}$ and the stoichiometry structure of the full model, that is, $\nu_{in}$ and $\nu_{out}$. 
We denote with $\mathcal{S}_{\mathcal{P}}$ the set of indexes of variables that take part on the stoichiometry of the sensitive reaction channels $\mathcal{J}_\mathcal{P}$, i.e., 
\begin{equation}
\label{eq:sens_spec}
i \in \mathcal{S}_\mathcal{P} \iff 
\,\, \text{exists} \,\, j {\in} \mathcal{J}_{\mathcal{P}} \,\, \text{ for which } \,\, (\nu_{in})_{i,j} > 0 \text{ or } (\nu_{out})_{i,j} > 0 \, .
\end{equation}

These variables are selected because are part of the stoichiometry of the sensitive reaction channels, and therefore are considered necessary for constructing the reduced model.

Since a propensity function  may depend on any state variable, and not only the ones that take part in the stoichiometry, we define the complement of $\CG $ by $\CG^{\comp,1}$ and $\CG^{\comp,2}$. The first one takes into account, for each $j\in \mathcal{J}_\mathcal{P}$, the state variables that are included in the respective propensity function $j$ and not in the stoichiometry reaction channel $j$. In this work, those state variables are not included in the reduced model as variables, but time averages, as explained below in Section \ref{sec:model_red}. Those variables are in fact relevant but not necessary for constructing the reduced model, and therefore modeled as constant values. The linear map $\CG$ is depicted in Figure \ref{fig:CG}.


\begin{figure}[h!]
\centering
\begin{tikzpicture}[node distance=2cm]
\node (pars) [process] {$(x_1,x_2,...,x_d)$};

\node (theta) [process, below of=pars, xshift=-3cm] {$(\bar x_1,\bar x_2,...,\bar x_{\bar d})$};
\node (u) [process, below of=pars] {$(\bar y_1,\bar y_2,...,\bar y_{\bar d'})$};
\node (ct) [process, below of=pars, xshift=3cm] {$\tilde x$};

\node [below of=theta, yshift=1cm,text centered] {variables};
\node (c2) [below of=u, yshift=1cm,text centered] {constants};
\node [below of=c2, yshift=1.6cm,text centered] {(time averages)};
\node (c3) [below of=ct, yshift=1cm,text centered] {non-resolved};
\node [below of=c3, yshift=1.6cm,text centered] {(eliminated)};

\node [left of=pars, xshift=-4cm] {Full model:};
\node [left of=theta,  xshift=-1cm] {Reduced model:};

\draw [arrow] (pars) -- node {\fcolorbox{black}{white}{$\CG$}} (theta);
\draw [arrow] (pars) -- node {\fcolorbox{black}{white}{$\CG^{\comp,1}$}} (u);
\draw [arrow] (pars) -- node {\fcolorbox{black}{white}{$\CG^{\comp,2}$}}  (ct);

\end{tikzpicture}
\caption{\label{fig:CG} Schematic representation of the full model state variables to reduced model state variables.}
\end{figure}
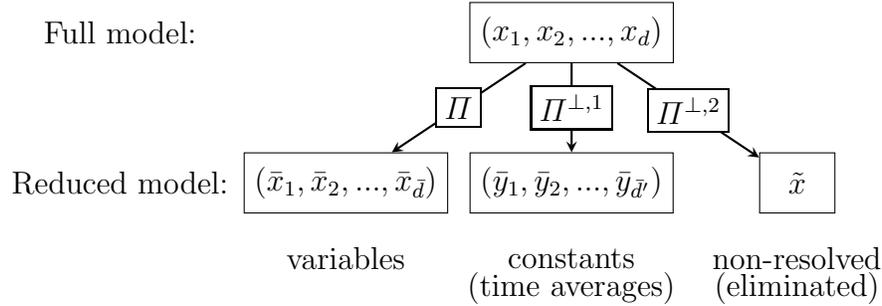





\subsection{Construction of a parameterized family of reduced models (Step 3)}
\label{sec:model_red}
Here we present a particular method to construct the parametric family of reduced models.
Let $(\nu_j,a_j)_{j=1}^J$ be the full RN, as defined in Section \ref{sec:RNs}. Let $\mathcal{P}$ be a set of parameters to include in the reduced model, and $\mathcal{J}_\mathcal{P}$ be the corresponding set of reaction channels (from Step 1 and Step 2). Let $\bar x = \CG x$ denote the macroscopic state, and recall that $x$ can be decomposed as $(\CG x,\CG^{\comp,1} x,\CG^{\comp,2} x) = (\bar x,\bar y, \tilde x)$ \eqref{eq:var_map_det}.
Let $\theta = \varGamma c$ denote the reduced model parameter vector, and recall that the microscopic parameter vector can be decomposed as $(\varGamma c,\varGamma^{\comp,1} c, \varGamma^{\comp,2} c) = (\theta,u,\tilde c)$ \eqref{eq:par_map}. 


The reduced model stoichiometry is defined as follows. For $i \in \mathcal{S}_{\mathcal{P}}$ and $j\in \mathcal{J}_{\mathcal{P}}$ as
\begin{equation}
\label{eq:bar_nu}
(\bar \nu_{in})_{k,l} := (\nu_{in})_{i,j} \,\, \text{ for } \,\, k =1,2,...,\bar d, \,\, l =1,2,..., \bar J \, ,
\end{equation}
where the indexes $k=k(i)$ and $l=l(j)$ preserve the order of $i$ and $j$, and similarly for $\nu_{out}$. Then $\bar \nu = \bar \nu_{out} - \bar\nu_{in}$.

Let $x_{0:T}$ denote the time average of the microscopic data $(x_i)_{i=0}^T$ \eqref{eq:data},
\begin{equation}
\label{eq:timeavg}
x_{0:T} := \frac{1}{T}\sum_{i=1}^T x_i \Dt_i \, .
\end{equation}
for a non-uniform time step $\Dt_i$.

We define the reduced model propensity functions as follows. For $j \in \mathcal{J}_\mathcal{P}$ there is a corresponding $k =k(j) \in \{1,2,...,\bar J\}$ such that
\begin{equation}
\label{eq:bar_a}
\bar{a}_{k}(\bar{x};\theta) := a_j((\bar x, \CG^{\comp,1} x_{0:T});(\theta, u)) \,, \end{equation}
where $\bar{x} \in \mathbb{R}^{\bar{d}}, \theta \in \mathbb{R}^{\bar K}$. 

Some observations apply. The reduced model propensity functions are defined to have the same functional form than its full model counterparts, but evaluated at $(\bar x,\CG^{\comp,1} x_{0:T})$ and its parameters evaluated at $(\theta, u)$ \eqref{eq:par_map}. 
Since $j\in \mathcal{J}_{\mathcal{P}}$, the full model propensity function $a_j$ depends explicitly on state variables $\bar x = \CG x$ and $\bar y = \CG^{\comp,1} x$ only. For the same reason, it only depends explicitly on parameters $\theta = \varGamma c$ and constants $u = \varGamma^{\comp,1} c$.
Here we choose $\bar y$ to be equal to the projection of the time average of the data, $\CG^{\comp,1} x_{0:T}$.
This is a natural choice consistent with long-term behaviour assuming that exists the time-average limit of the mean field approximation of the full model, $z(t)$, in the sense that $\lim_{T\rightarrow +\infty} \frac{1}{T}\int_0^T y(t)dt = z_\infty$, where $z_\infty$ is a constant.


It is important to notice that this reduction method focuses on the stoichiometry, i.e., structure of the reaction network, and not on the propensity functions. This allows to take advantage of the natural sparsity of the $\nu$ matrix and to use the full model function space for the reduced model propensity functions.


\subsection{Model training (Step 4)} 
\label{sec:fitting}
Given the times series data $(x_i)_{i=0}^T$ \eqref{eq:data} with non-uniform step $\Dt_i$, we numerically find $\theta^*$ by maximizing the loss function \eqref{eq:losse_approx} 
\begin{equation}
\label{eq:opt_e}
\arg \min_{\theta}\frac{1}{2}\sum_{i=1}^T \norm{\bar b(\bar x_i;\theta) - \CG b(x_i))}^2_\CG \Dt_i \, ,
\end{equation}
where 
$
\bar{b}(\bar x;\theta) = \bar{\nu} \bar{a}(\bar x ;\theta) 
$
and the norm $\norm{\cdot}^2_\CG$ as in \eqref{eq:norm}. Even though the loss function $E_{0:T}(\theta)$, as given by formula \eqref{eq:losse} is based on the Langevin approximation, we can plug-in any time series as \eqref{eq:data}, by using the mean-field approximation discussed in Section \ref{sec:MF_approx}.

%
%
%
%
%
%
%

\subsection{Validation and improvement of the reduced RN (Step 5 and Step 6)}
\label{sec:augment}
In this step, we assess the quality of the mean-field trajectories generated by a fitted reduced model (as obtained in Step 4) and the mean-field trajectories generated by the full model. This comparison can be also performed by using the times series data \eqref{eq:data} instead of the full model mean-field. We opt for using the former one. Also, we discuss how the reduced model may be expanded by adding reaction channels of a given species of interest.
Suppose that, in a given a fitted reduced model with a set of species $\mathcal{S}$, exists a species of interest $i \in \mathcal{S}$ such that 
$$\sup_{t \in [0,T]} \frac{\abs{z_i(t)-\bar z_i(t)}}{z_i(t)} >TOL \, ,
$$
where $z_i(t)$ denotes the full model mean-field trajectory of the $i$-th state component, and, similarily $\bar z_i$, is the mean-field trajectory generated with the reduced model, for a user defined relative tolerance $TOL > 0$. 
Since the diagonal of the pFIM determines a linear order on the parameters, any reduced model that can be obtained by changing the information threshold $\kappa$ are nested in terms of species and reaction channels. We now discuss an idea to obtain models that are not attached to this nested hierarchy, focusing on a one step analysis, that can be easily extended to a multi-step one by traversing the graph defined by the stoichiometry structure of the full model. 
Let $\mathcal{J}_i$ be the set of reaction channel indexes such that for $j \in \mathcal{J}_i$, we have 
\[
(\nu_{in})_{i,j} > 0 \quad \text{ or } \quad (\nu_{out})_{i,j} > 0 \, ,
\]
that is, $\mathcal{J}_i$ denotes the set of reaction channels in which species $i\in \mathcal{S}$ takes part in the  stoichiometry. Then, if $\mathcal{J}_i$\textbackslash{}$\mathcal{J}_p$ is not empty, a new state variables map can be defined as in Section \ref{sec:sel_var}, equation \eqref{eq:sens_spec}, but using $\mathcal{J}_{\mathcal{P}} \cap \mathcal{J}_i$ instead of $\mathcal{J}_{\mathcal{P}}$ in \eqref{eq:sens_spec}. 
A simple illustration of this extension idea is presented in the first example of Section \ref{ex:PH}.

\begin{rem}[Other strategies to construct the parameterized family]
As long as the state map $\CG$ is linear and orthogonal, the general form of the loss function applies, assuming the data comes from an Euler discretization of the CLE. Otherwise, the loss function needs to be rederived from the relative entropy as we did in Theorem \ref{thm:argmin}.
To construct the state map, $\CG$, one can use our pFIM approach, or other criteria for variable selection.
\end{rem}

\begin{rem}[Computational work considerations of the reduced model construction]
One of the main advantages of the reduction method presented in this work is its feasibility in terms of computational work. From the algorithmic point of view, there are three main computational steps to construct the reduced model (not including the fitting step):
\begin{enumerate}
\item Computing the pFIM \eqref{eq:pFIM} and determine the sensitive parameter set $\mathcal{P}$: As previously mentioned, in virtue of its block diagonal structure 
 this step can be achieved with a computational work roughly of the order of the number of parameters of the full model, $K$.
\item Parameter mapping: Work to construct $\CGp$ \eqref{eq:defvargamma} is of order $K$ and to construct $\CGp^{\comp,1}$ \eqref{eq:defvargamma_comp} is linear on $|\mathcal{J}_{\mathcal{P}}|$, which is of order of the number of reaction channels, $J$.
\item State variable mapping: Work to construct $\CG$ \eqref{eq:defpi} is linear on $|\mathcal{J}_{\mathcal{P}}|$ which is of order $J$. Work to construct $\CG^{\comp,1}$ \eqref{eq:defpi_comp} is linear on $|\mathcal{J}_{\mathcal{P}}|$ which is of order $J$, and requires additional storage of the order of the dimension of the state space, $d$.
\end{enumerate}
\end{rem}

\section{Numerical experiments}
\label{sec:exp}
In this section we demonstrate the applicability of our method by analyzing three models taken from the literature: A protein homeostasis network in a sloppy regime \cite{proctor11}, a Epidermal Growth Factor Receptor (EGFR) model \cite{Kholodenko99}, and a Mammalian Circadian clock model that presents dymanics dominated by oscillations  \cite{Leloup03}. Moreover, we present two stochastic pure jump models: a circadian clock model which is based on \cite{Leloup03}, and a growth factor receptor, based on \cite{Kholodenko99} The parameter values, as well as initial conditions, quantities of interest, and time horizon of the study are taken from the model database \cite{databaseURL,BioModels2010} (manually curated section). In the first case, we consider a particular regime studied in the corresponding paper, in which is possible to obtain a substantial reduction both in the number of state variables (or species) and in the number of parameters and reaction channels. 
In the second example we analyze a non-sloppy model of signalling phenomena, and finally, in the third example, we reduce a model whose dynamics are dominated by non-trivial oscillatory behaviour. Finally, in Section \ref{sec:stoc_exp} we scale the deterministic Mammalian Circadian clock model and the EGFR to obtain in each case a stochastic pure jump model. The first one is clearly not in a Langevin nor mean-field regime. The second one is still well approximated by the deterministic model. Both stochastic models exhibit non-Gaussian distributions for some of the species.
In all the examples, even though the models were carefully designed by researchers, significant model reductions are achieved.

\subsection{Matrix representations of the linear maps}

Here we briefly describe the matrix representations of $\CGp$ and $\CG$ used in the numerical experiments.
The matrix representation of $\CGp$ is constructed (row-wise) as follows. The $k$-th row is defined as
\begin{equation}
\label{eq:defvargamma}
(\varGamma)_{k,\cdot}:=\mathbf{e}_{l}\,\, \text{ for }\,\, k=1,2,...,\bar K \, ,
\end{equation}
for parameter index $l \in \mathcal{P}$ and $\mathbf{e}_l$ is the $l$-th element of the standard basis of dimension $K$, assuming 
the row index $k$ preserves the order of index $l$. Similarily, the $k$-th row of the matrix representation of the map $\CGp^{\comp,1}$ is defined as 
\begin{equation}
\label{eq:defvargamma_comp}
(\varGamma^{\comp,1})_{k,\cdot}:=\mathbf{e}_{l}\,\, \text{ for }\,\, k=1,2,...,\bar K' \, ,
\end{equation}
for each parameter index  $l \in \{1,2,...,K\}$ such that $l \not \in \mathcal{P}$ but exists $j \in \mathcal{J}_\mathcal{P}$ such that $j \in \varphi(l)$.
Finally, $\tilde c_l$ is an component of $\tilde c$ if exists $l \in \{1,2,...,K\}$, $l \not \in \mathcal{P}$ and there is no $j \in \mathcal{J}_\mathcal{P}$ such that $j \in \varphi(l)$.

Now, let $\mathcal{P}$ be a set of selected parameters, with $|\mathcal{P}|= \bar{K}$, as defined in Section \ref{sec:sel_pars}, and 
let $\mathcal{J}_{\mathcal{P}} =\bigcup_{k \in \mathcal{P}}\varphi(k)$ be the set of indices of reaction channels associated with the set of parameter indexes $\mathcal{P}$. The $k$-th row of $\CG$ is then defined as
\begin{equation}
\label{eq:defpi}
(\CG)_{k,\cdot}:=\mathbf{e}_i\,\, \text{ for } \,\, k=1,2,...,\bar d \, ,
\end{equation}
for each species index $i \in \{1,2,...,d\}$ such that
\begin{equation*}
\,\, \text{exists} \,\, j {\in} \mathcal{J}_{\mathcal{P}} \,\, \text{ for which } \,\, (\nu_{in})_{i,j} > 0 \text{ or } (\nu_{out})_{i,j} > 0 \, ,
\end{equation*}
where 
$\mathbf{e}_i$ is the $i$-th element of the standard basis of dimension $d$, assuming that the row index $k=k(i)$ preserves the order of index $i$ as in \eqref{eq:defvargamma}. We denote with $\mathcal{S}_{\mathcal{P}}$ the set of indexes of species that take part on the stoichiometry of the sensitive reaction channels $\mathcal{J}_\mathcal{P}$.

The map $\CG$ does not take into account the case in which a propensity function $a_j$ depends explicitly on a particular component of the state variable $x$, say $x_i$, and $i \not \in \mathcal{S}_{\mathcal{P}}$. 
Since the functional expression of the propensity functions of the sensitive reaction channels are preserved in the reduced model, a suitable complement, $\CG^{\comp,1}$, is defined as follows.

For each reaction channel index $ j \in \mathcal{J}_\mathcal{P}$ such that 
\begin{equation*}
\text{exists} \,\, i{\in} \{1,2,...,d\}, \,\, \text{ such that } \,\, a_j = a_j(x_i;\cdot) \,\, \text{ and } \,\, i \not \in \mathcal{S}_\mathcal{P} \, ,
\end{equation*}
define
\begin{equation}
\label{eq:defpi_comp}
(\CG^{\comp,1})_{k,\cdot}:=\mathbf{e}_i\,\, \text{ for } \,\, k=1,2,...,\bar d' \, ,
\end{equation}
where $()_{k,\cdot}$ denotes the $k$-th row of the matrix, $\mathbf{e}_i$ is the $i$-th element of the standard basis of dimension $d$, and $a_j=a_j(x_i;\cdot)$ denotes the fact that the propensity function $a_j$ is an explicit function of $x_i$. Here we assume that the row index $k=k(i)$ preserves the order of index $i$ as in \eqref{eq:defvargamma}.

\subsection{Protein homeostasis}
\label{ex:PH}
\subsubsection*{Model description}
Loss of protein homeostasis is
the common link between neuro-degeneration disorders which are characterized by the
accumulation of aggregated protein and neuronal cell death. The authors in \cite{proctor11} examined the role of
both $Hsp70$ and $Hsp90$ under three different regimes: no stress, moderate stress and high
stress. This reaction network consists of 52 species and 80 reactions with propensities being of mass-action kinetics type. The reaction constants, the initial population and time interval, $[0,T]$, $T{=}10$, are taken from \cite{proctor11} and \cite{databaseURL}. The variables represent particle counts. This model can be well approximated by a diffusion (CLE approximation \eqref{eq:lang}), since the particle count is large for every species.
The authors study a regime of \emph{no stress}, in which most of the classical sensitivity indexes as per \eqref{eq:adjoint} are close to zero. Recall that we consider that a system is in a sloppy regime when most of the information contained in the pFIM is accumulated in a reduced number of parameters (see Figure \ref{fig:PH_pFIM}). In this work, we study the Protein Homeostasis model in this regime as an example of substantial model reduction, both in the number of species and in the number of parameters and reaction channels. Every plot shown refers to the time interval $[0,T]$.

\subsubsection*{Selecting parameters and reaction channels (Step 1)}
We start this example by noticing that 10 out of 87 parameters accumulate at least $95\%$ (precisely, $96.524\%$) of the total information as per the pFIM diagonal \eqref{eq:pFIMpctg}, as shown in Figure \ref{fig:PH_pFIM}. Total information of a set of parameters is a natural and intuitive measure that allows to choose a meaningful set of parameters for information-based model reduction. 
\begin{figure}[h!]
\centering
\begin{minipage}{0.46\hsize}
	\includegraphics[width=\textwidth]{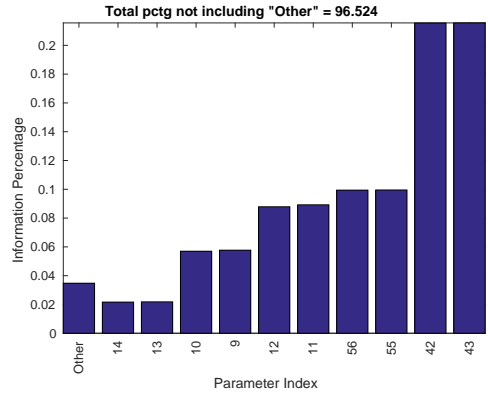}
\end{minipage}
\caption{\label{fig:PH_pFIM} Total information as per the pFIM diagonal (see \eqref{eq:pFIMpctg} and \eqref{eq:pFIMthr}). It can be seen that the parameters that represent at least $95\%$ of the total information as per the pFIM (10 parameters), contains the most sensitive parameters, as shown in Figure \ref{fig:PH_SIs}.}
\end{figure}

As a comparison, in Figure \ref{fig:PH_SIs} we plot the sensitivity indexes of the time average of the species (see \eqref{eq:time_avg}), computed by using the adjoint deterministic method (see \eqref{eq:adjoint}). The time average is considered as a pathwise quantity of interest for assessing the quality of the reduced model for replicating the full dynamics. The axis containing the parameter indexes, are sorted by using the diagonal values of the pFIM (see \eqref{eq:SB}).  The group of parameters with largest sensitivity indexes as per the adjoint method appears in the rightmost side of the plot. 

\begin{figure}[h!]
\centering
\begin{minipage}{0.54\hsize}
	\includegraphics[width=\textwidth]{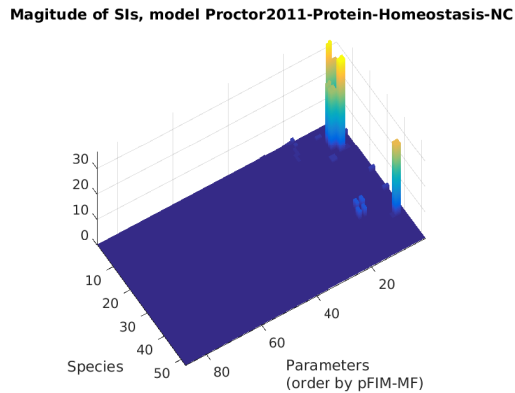}
\end{minipage}
\caption{\label{fig:PH_SIs} Sensitivity indexes for the time average of the species (see \eqref{eq:adjoint}). Parameter axis is ordered by the diagonal of the pFIM. The computational work required to compute these indexes is of the order of Parameters${\times}$Species.}
\end{figure}
By comparing both figures, we can note that the aforementioned 10 parameters, that represent more than $95\%$ of the total information as per the pFIM diagonal, include the most sensitive parameters as per the adjoint method, for the time average of the species.  This stresses the fact that by considering a large  amount of the total information as per the pFIM diagonal, sensitive parameters in the classical sense \eqref{eq:adjoint}, for the time average, is likely to be classified as pFIM sensitive. We empirically observed this fact in several examples contained in the EMBL-EBI BioModels database  \cite{databaseURL}.
Notice, however, that the two parameters that accumulate most of the information as per the pFIM diagonal (parameters $c_{42}$ and $c_{43}$) are not sensitive with respect to the time average (see the two rightmost part of the plot in Figure \ref{fig:PH_SIs}). 
Finally, we recall that in the case of high-dimensional models, the computational work required to compute the sensitivity indexes by the adjoint and related methods is usually prohibitive (see Remark \ref{rem:adjoint}). On the other hand, for computing the pFIM, a computational work of the order of the number of parameters is required (see Section \ref{sec:pFIM}). This can be readily seen by comparing Figure \ref{fig:PH_SIs} with Figure \ref{fig:PH_pFIM}.

In Figure \ref{fig:PH_graph} (left pane) we show the pathwise information geometry for the full model, that is, how the total information is structured in the network described by the stoichiometry. This plot distinguishes sensitive reaction channels (see \eqref{eq:sens_chans}). We show reaction channels and species in a circumference, where black dots depicts reaction channels and red dots, species. A link between one reaction channel and one species shows the stoichiometry structure, that is, which species takes part in the stoichiometry of which reaction channels 
where the color of the link represents the total information as per the pFIM diagonal of that particular reaction channel. That is, 
\begin{equation}
\label{eq:link_color}
j\text{-th link color} = \frac{1}{\sum_{\ell=1}^J \sum_{k \in \mathcal{K}_\ell}\xi_k} \sum_{k \in \mathcal{K}_j}\xi_k \, ,
\end{equation}
where $\xi_k= (\pFIM{P_{0:T}})_{k,k}$ (eq. \eqref{eq:dk}), and $\mathcal{K}_j:=\{k{\in} \{1,2,...,K\}:j{\in} \varphi(k)\}$. Finally, $\varphi$ is defined in \eqref{eq:varphi}.

\begin{figure}[h!]
\centering
\begin{minipage}{0.48\hsize}
	\includegraphics[width=\textwidth]{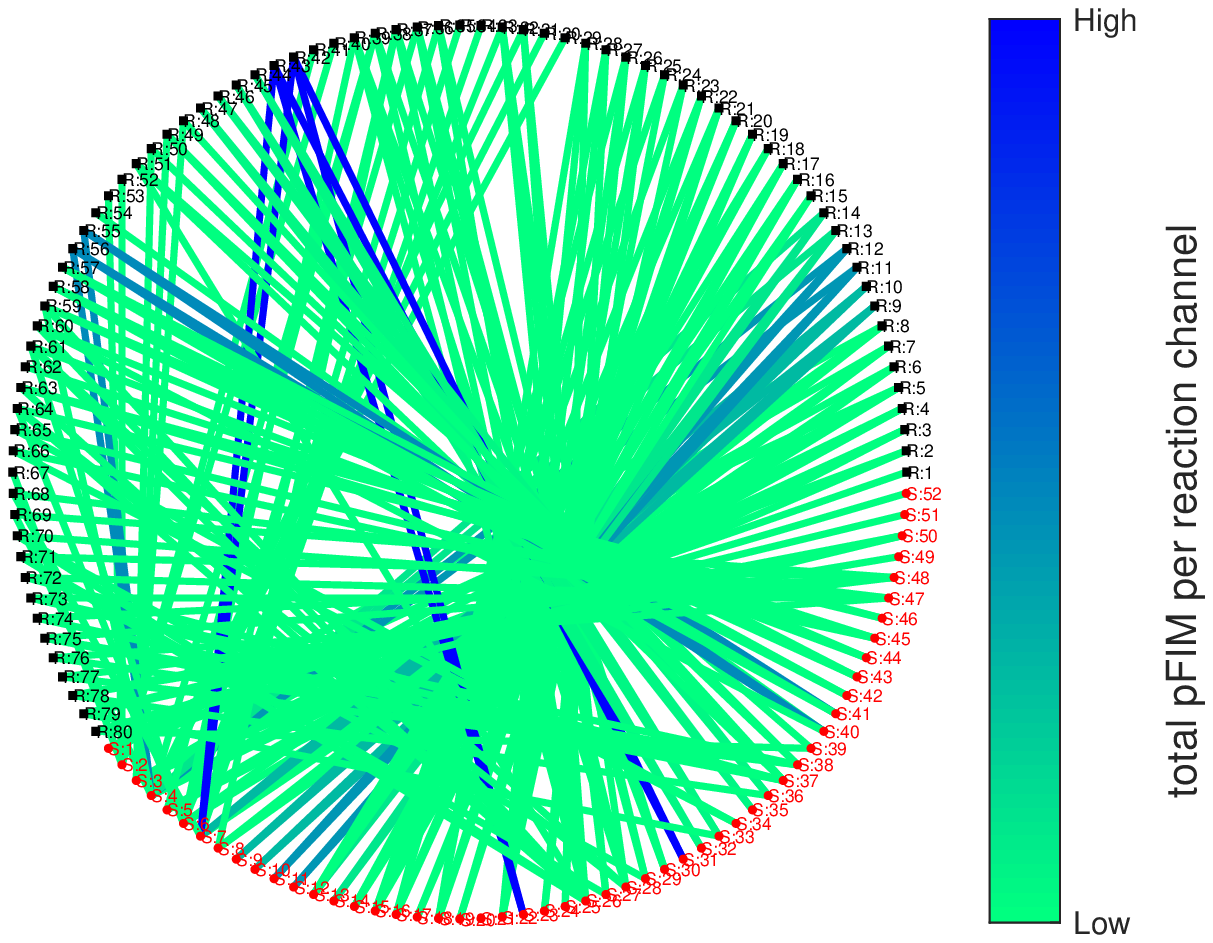}
\end{minipage}
\hfill
\begin{minipage}{0.48\hsize}
	\includegraphics[width=\textwidth]{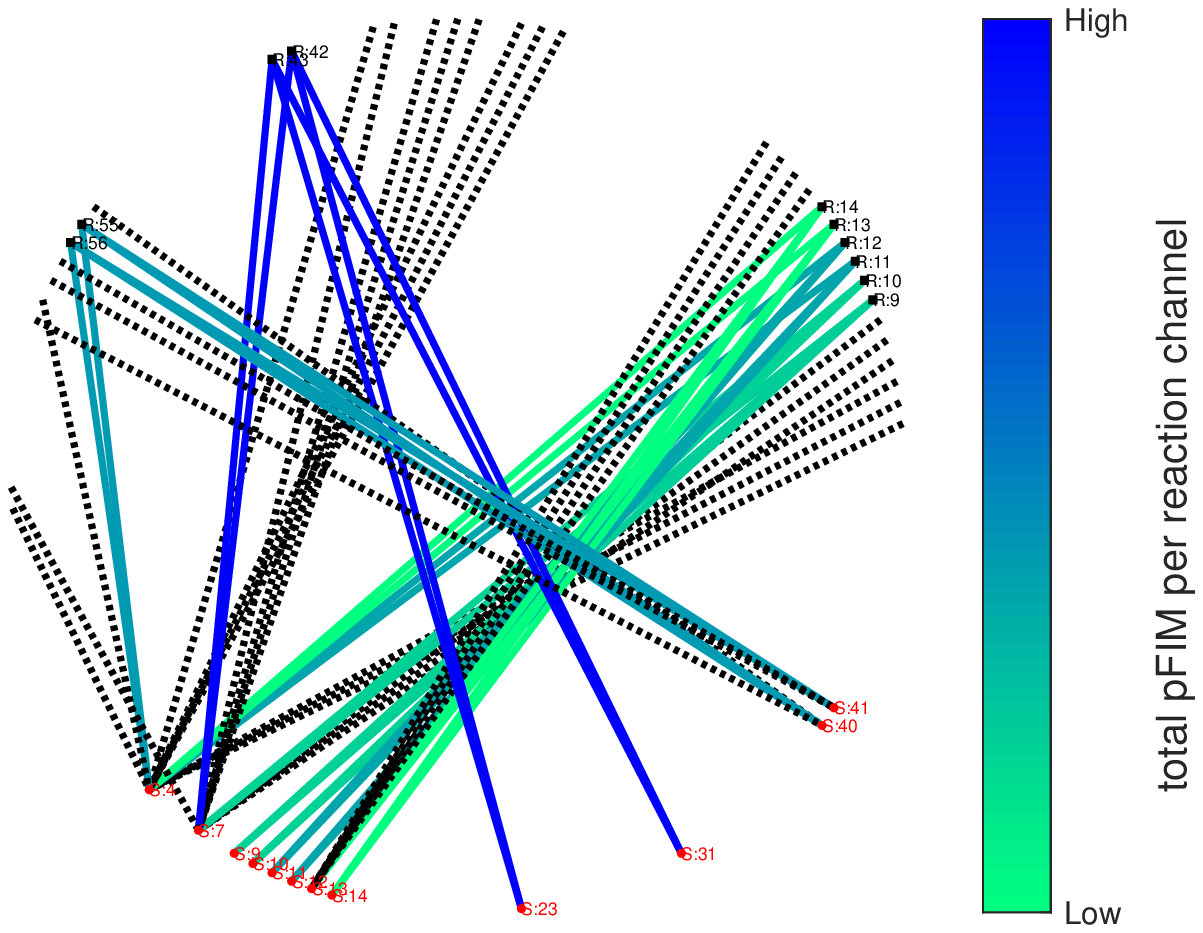}
\end{minipage}
\caption{\label{fig:PH_graph} Stoichiometry and information. Black dots: reaction channels, red dots: species. A link shows stoichiometry of each reaction channel. The color of the link shows the total information of the corresponding reaction channel (see \ref{eq:link_color}). Left: Full model. Right: Reduced model, including at least $95\%$ of the total information as per the pFIM diagonal \eqref{eq:pFIMpctg}. In the plot, species indexes are shown, while species names are given in Figure \ref{fig:PH-trajs-full-red}.}
\end{figure}

\subsubsection*{Selecting variables (Step 2)}
In Figure \ref{fig:PH_graph} (right pane) we show only the sensitive reaction channels \eqref{eq:sens_chans} and species \eqref{eq:sens_spec}, for $\kappa=0.95$. Dotted lines shows, for the sensitive species (i.e., included in the reduced model), the stoichiometry relations with reaction channels not included in the reduced model. 
For example, species $S_4$ (see right pane of Figure \ref{fig:PH_graph}) takes part in the stoichiometry of reaction channel $j{=}4$ (R4 in Figure \ref{fig:PH_graph}) but this reaction channel is not included in the reduced model because contains no pFIM-sensitive parameters.

In this plot we can identify an almost isolated sensitive subsystem in the stoichiometry of this network. That is, this network has a few sensitive reaction channels that affect only a few species, and those species are mainly affected by those reaction channels. This latter is shown by dotted black links in right pane of Figure \ref{fig:PH_graph}).

Species $Hsp70$ and $Hsp90$ ($S_7$ and $S_4$ respectively, see Figure \ref{fig:PH-trajs-full-red}), which are the main quantities of interest measured in \cite{proctor11}, are already chosen in the first iteration of our method for $\kappa{=}0.95$. This is because both species take part in the stoichiometry of a sensitive reaction channel for the specified $\kappa$. Note that any particular species that may be of interest, even if it is sensitive, can be also included in the reduction procedure presented in this work. 

\subsubsection*{Reduced model (Step 3 and Step 4)}
In Figure \ref{fig:PH-trajs-full-red} we compare in a log scale the mean-field trajectories of the species included in the reduced model for $\kappa{=}0.95$, and the respective mean-field trajectories of the same species in the full model. We also show a comparison between the time average on $[0,T]$ of the full model versus the corresponding reduced one. Names of the species included in the reduced model are given in both plots of Figure \ref{fig:PH-trajs-full-red}.

To visually compare how the amount of information as per the pFIM included in the reduced model affects the resulting mean-field trajectories, in Figure \ref{fig:PH-93} we show, as in Figure \ref{fig:PH-trajs-full-red}, a comparison of the mean-field trajectories of the full model but in this case versus a reduced model for $\kappa{=}0.93$. It can be  clearly seen that the trajectory of one of the species in the reduced model has very different dynamics than in the full model. This explains why the information loss and the pathwise distance are much larger than the case with $\kappa{=}0.95$ (see Table \ref{tab:PH}).

\begin{figure}[h!]
\begin{minipage}{0.48\hsize}
	\includegraphics[width=\textwidth]{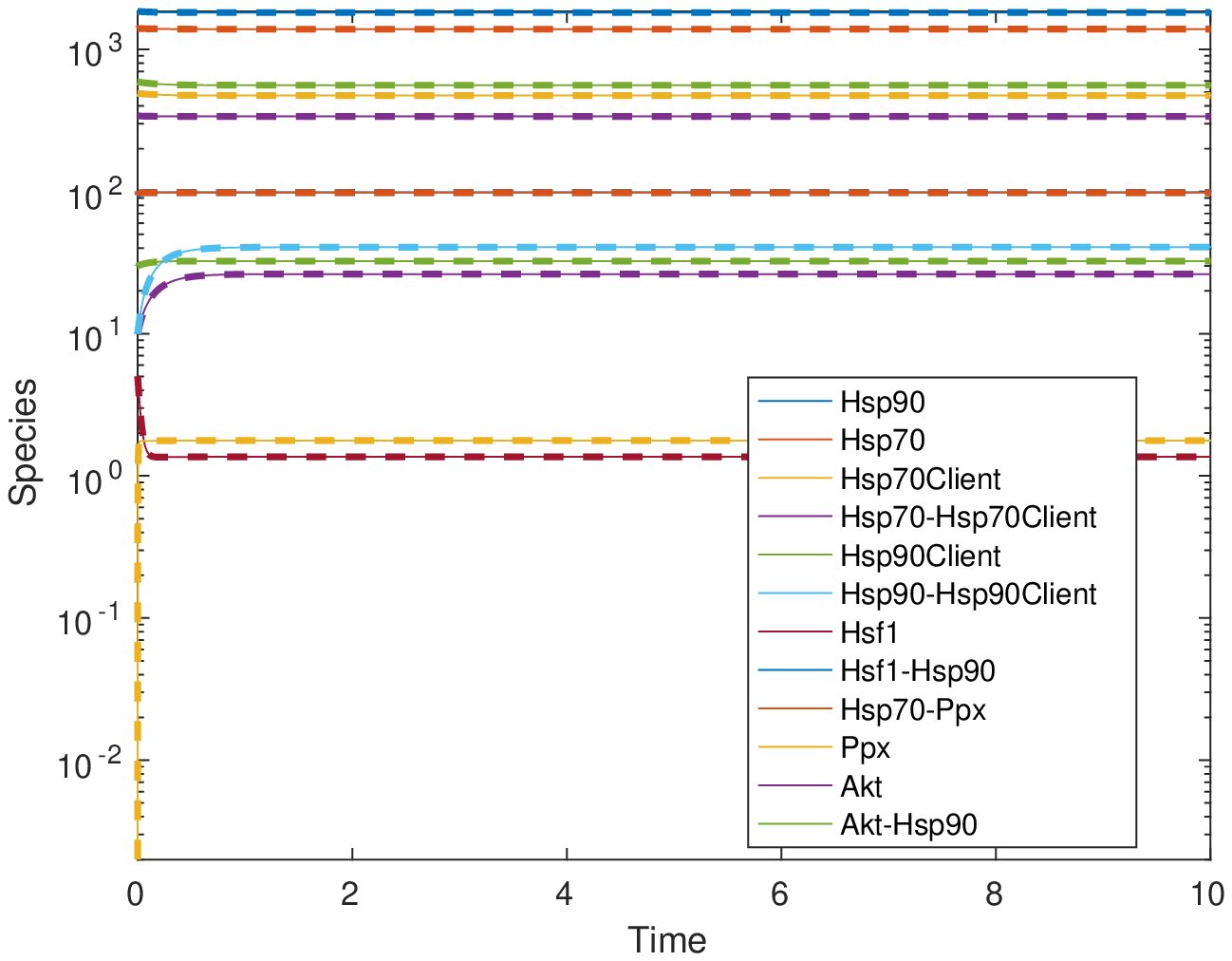}
\end{minipage}
\hfill
\begin{minipage}{0.48\hsize}
	\includegraphics[width=\textwidth]{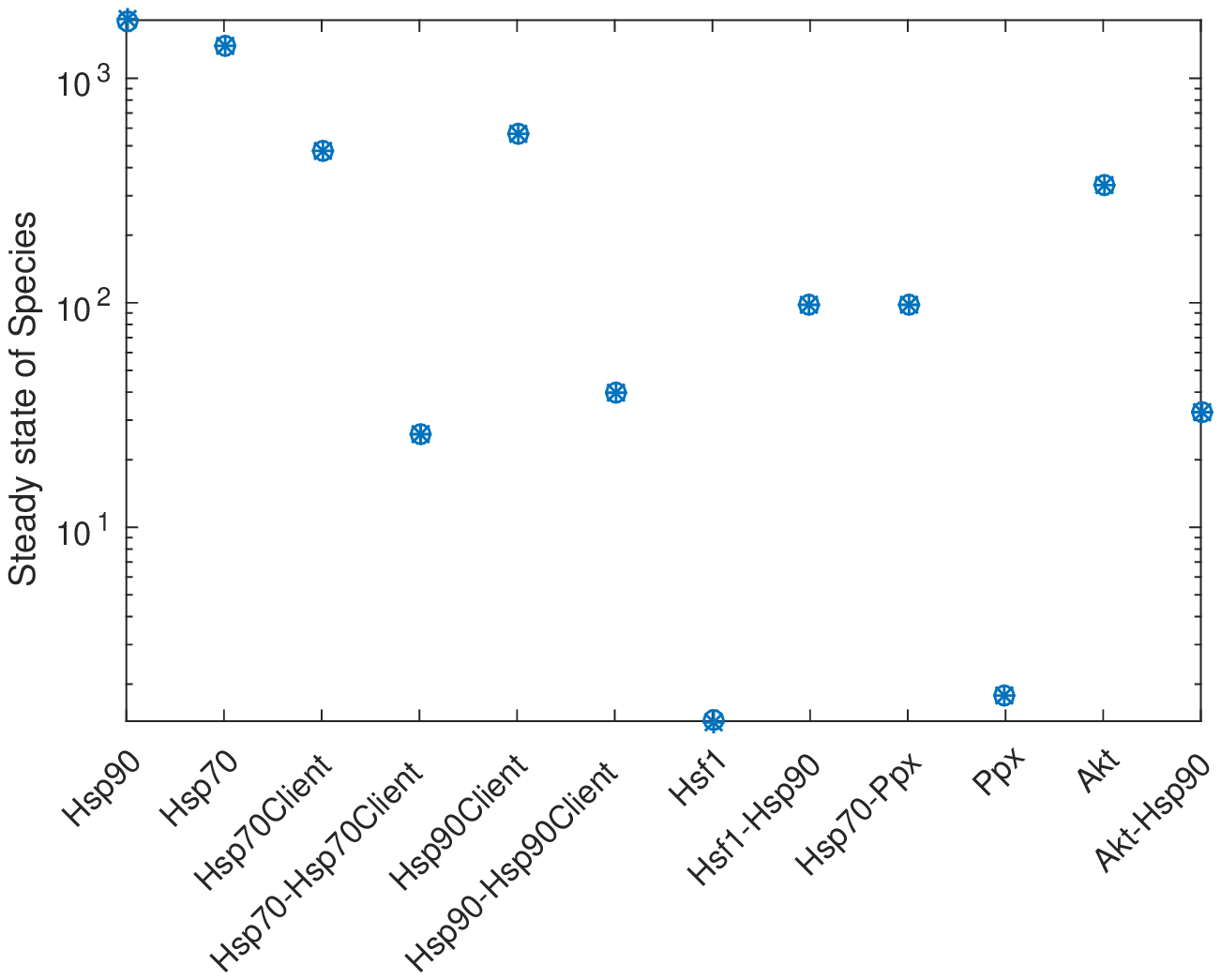}
\end{minipage}
\caption{\label{fig:PH-trajs-full-red} Left: Mean-field trajectories of the species indexed by $\mathcal{S}_\mathcal{P}$. Solid lines correspond to the reduced model and thick dashed lines to the full model at $95\%$ of total information. Right: Time average of species in the reduced model (circle), vs. the full model (star) at $95\%$ of total information as per the pFIM diagonal.}

\end{figure}
\begin{figure}[h!]
\begin{minipage}{0.48\hsize}
	\includegraphics[width=\textwidth]{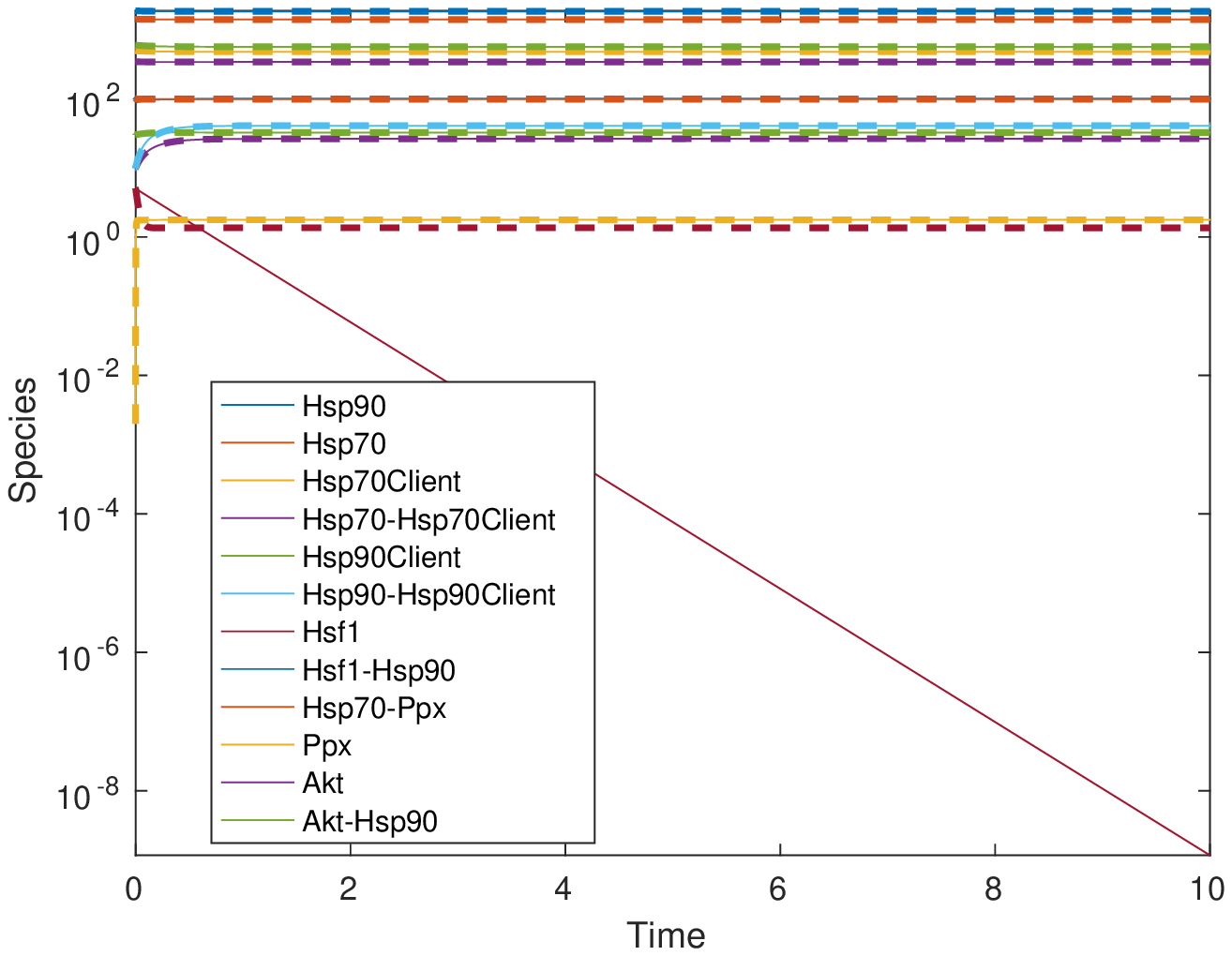}
\end{minipage}
\hfill
\begin{minipage}{0.48\hsize}
	\includegraphics[width=\textwidth]{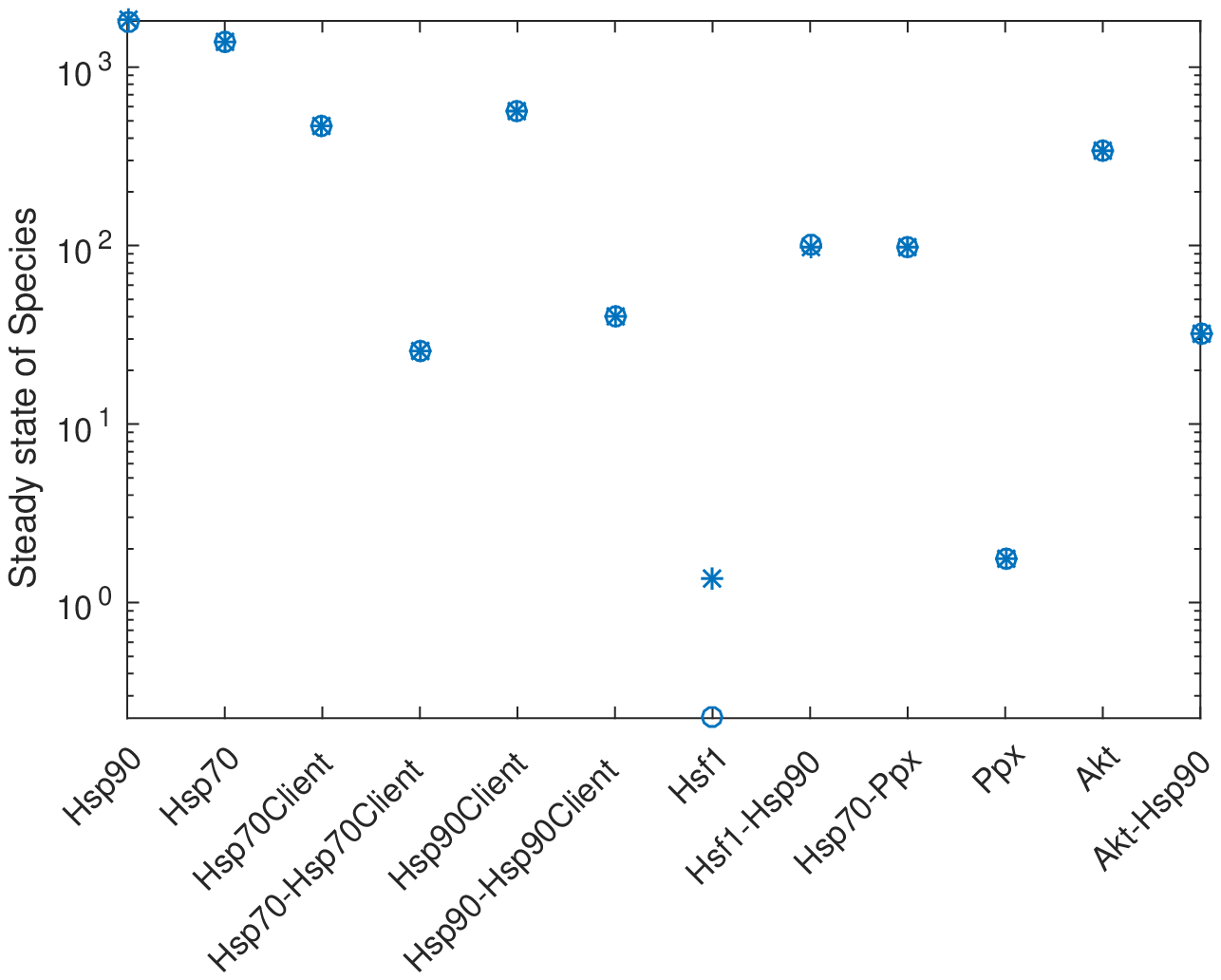}
\end{minipage}
\caption{\label{fig:PH-93} Mean-field trajectories and time average of species indexed by $\mathcal{S}_\mathcal{P}$, at $93\%$ of total information. Dynamics of species $Hsf1$ is different in the reduced model. Compare the trajectories versus a reduced model with $95\%$ if total information (Figure \ref{fig:PH-trajs-full-red}).}
\end{figure}

In Table \ref{tab:PH} we show different reduced models, according to total information as per the pFIM (first column). In the second, third and fourth columns we show the number of reaction channels, parameters and species in the reduced model respectively. Column ``Loss'' shows the value of the loss function \eqref{eq:losse} at which $\theta^*$ is achieved, by numerically solving problem \eqref{eq:opt_e}. 
When the information as per the pFIM diagonal increases from $95\%$ to $97\%$, the loss function at the optimal value also increases. 
Notice that the loss function value  decrease (by a large factor) when the minimum total information to keep increases from $93\%$ to $95\%$. This is due to the fact that an additional reaction channel is included in the reduced model while preserving the same number of species. This suggests that this reaction channel is crucial to correctly represent the dynamics of those species.

\begin{table}[h!]
\centering
\begin{footnotesize}\begin{tabular}{|c|c|c|c|c|}
\hline \rule{0pt}{2.6ex}
\textbf{pFIM \% \eqref{eq:pFIM}}&$\bar{J}$&$\bar{K}$&$\bar{d}$&\textbf{Loss \eqref{eq:losse}}\\\hline
93&9&9&12&194.87\\\hline
95&10&10&12&0.0168709\\\hline
97&11&11&15&6.95971\\\hline
99&13&13&18&6.97222\\\hline
\end{tabular}
\end{footnotesize}
\caption{\label{tab:PH} Different reduced models as per total diagonal pFIM information. Observe the non-monotone behaviour of the loss function \eqref{eq:losse} as the number of species increase. The trajectories associated with the first two models ($93\%$ and $95\%$) are shown in Figure \ref{fig:PH-93} and Figure \ref{fig:PH-trajs-full-red} respectively.}
\end{table}

\subsubsection*{Validation (Step 5)}
In this step we assess the quality of the trajectories generated by different reduced models, and we compute two validation distances. In order to perform a  detailed comparison, in table \ref{tab:PHp} we show different models as per the total number of parameters kept (ordered by total information as per the pFIM diagonal). Column ``path-dist'' shows the following pathwise distance
\begin{equation}
\label{eq:pdist}
\max_{i \in \mathcal{O}} \, \sup_{t \in [0,T]} \frac{\abs{z_i(t)-\bar z_i(t)}}{z_i(t)} \, ,
\end{equation}
where $\mathcal{O}$ is a fixed set of species, $z$ denotes the full model mean-field trajectories, and $\bar z$ the mean-field trajectories generated with the reduced model. By means of fixing the set of species, we can compare reduced models with different number of species. If $z_i(t)=0$ then the quotient is defined to be equal to $\bar z_i(t)$.
Finally, column ``SS-dist'' shows the following time average distance 
\begin{equation}
\label{eq:sdist}
\max_{i \in \mathcal{O}} \frac{\abs{z_{0:T,i}- \bar z_{0:T,i}}}{ z_{0:T,i}} \, ,
\end{equation}
where $\mathcal{O}$ is a fixed set of species, $z_{0:T}$ is the time average $z_{0:t}:=\frac{1}{T}\int_0^T z(s) ds$ of the full model mean-field, and $\bar z_{0:T}$ its corresponding counterpart in the reduced model.

It is important to note that, as the total information increases ($93\% \rightarrow 95\% \rightarrow 97\% \rightarrow 99\%$), reduced models may have additional species. In order to compare them, in Table \ref{tab:PHp} we restrict the comparison to the set of species of the model that has 10 parameters (line 4 in the table).

\begin{table}[h!]
\centering
\begin{footnotesize}\begin{tabular}{|c|c|c|c|c|}
\hline \rule{0pt}{2.6ex}
$\bar{J}$&$\bar{K}$&$\bar{d}$&\textbf{path-dist \eqref{eq:pdist}}&\textbf{SS-dist \eqref{eq:sdist}}\\\hline
7&7&10&0.530&0.128\\\hline
8&8&10&0.041&0.002\\\hline
9&9&12&0.034&0.005\\\hline
10&10&12&0.002&0.001\\\hline
\end{tabular}
\end{footnotesize}
\caption{\label{tab:PHp} Different reduced models as per total number of parameters in the reduced models. Distances shown in the last two columns are given in \eqref{eq:pdist} and \eqref{eq:sdist} respectively, where $\mathcal{O}$ is given by the set of species associated with the model that has 10 parameters (line 4 in the table). In such a way, we compare for the same species trajectories.}
\end{table}

\subsubsection*{Iteration (Step 6)}
We first point out that, a first model selection is carried out by by running the reduction procedure with different information thresholds (see Table \ref{tab:PH} and \ref{tab:PHp}).
In this section we show a further strategy to improve a given reduced model.
Suppose that it is required by the user to keep in the reduced model at least $99\%$ of the total information. This reduced model, obtained by applying Steps 1-5, is the one shown in last line of Table \ref{tab:PH}). By inspecting the distance between the full model mean-field trajectories versus the reduced ones, we can observe that the reduced model dynamics of species Jnk-P (species index 33) is departing from the full dynamics (see Figure \ref{fig:PH-ss-99}). 
\begin{figure}[h!]
\begin{minipage}{0.48\hsize}
	\includegraphics[trim=14pt 0pt 32pt 20pt,clip=true,width=\textwidth]{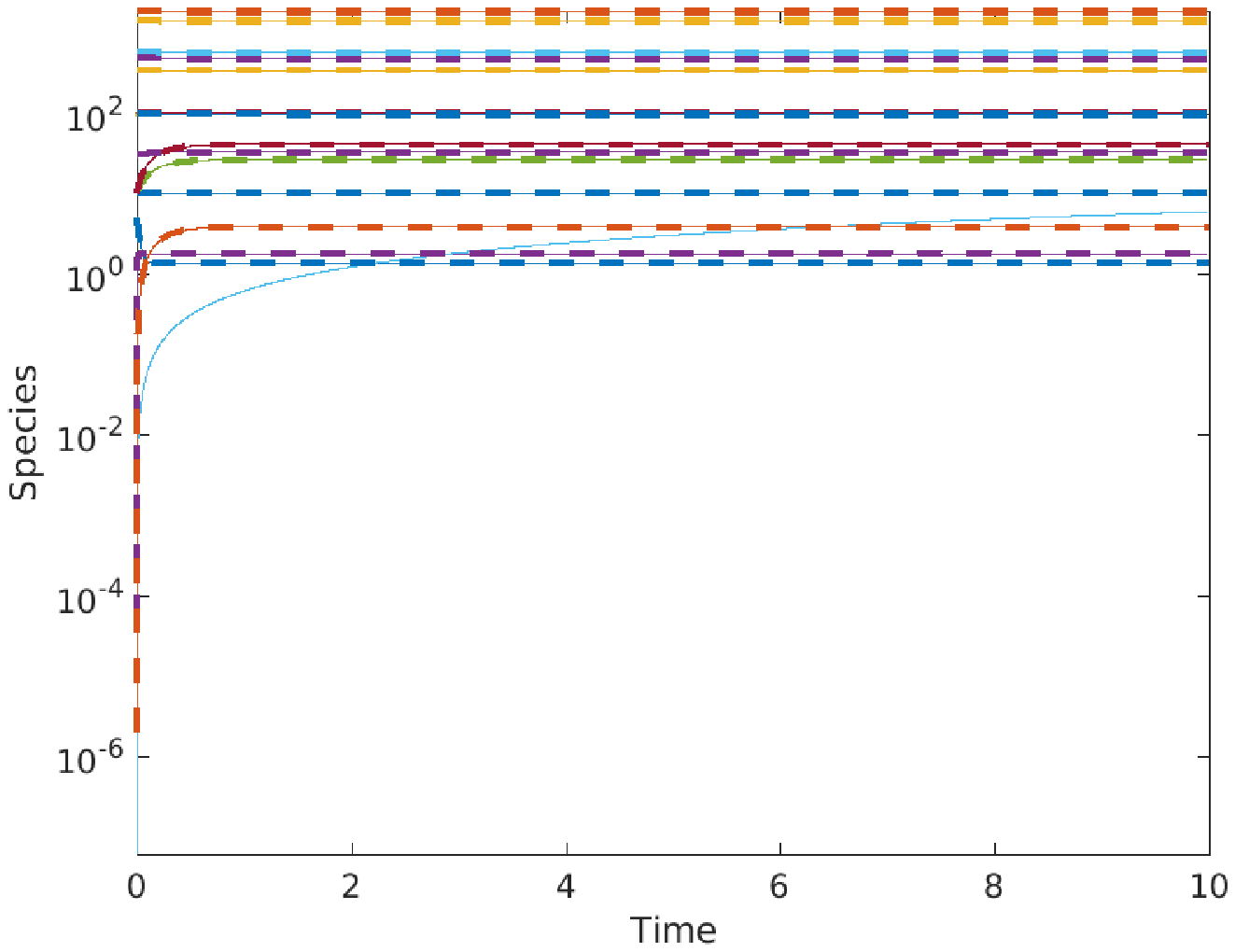}
\end{minipage}
\hfill
\begin{minipage}{0.48\hsize}
	\includegraphics[width=\textwidth]{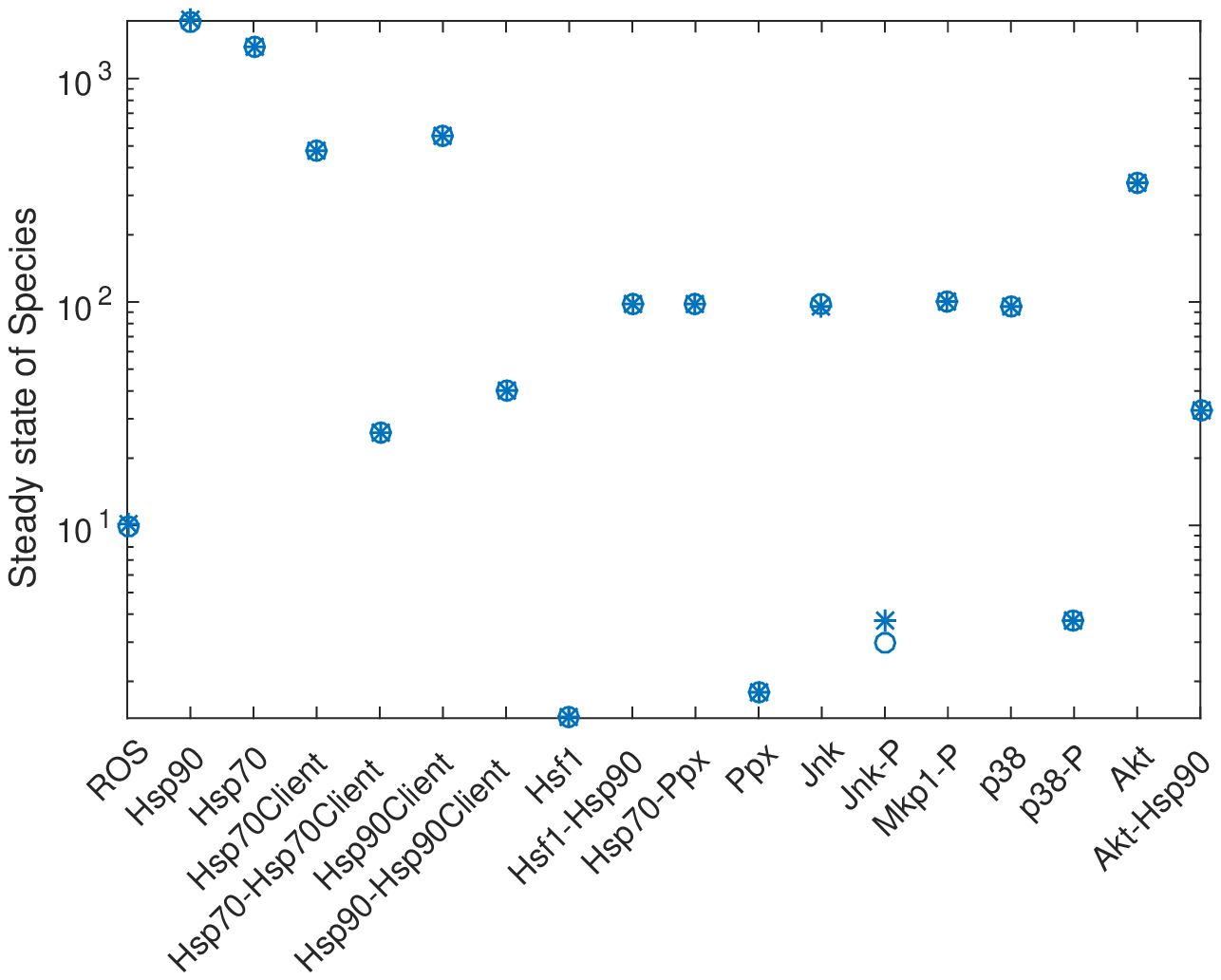}
\end{minipage}
\caption{\label{fig:PH-ss-99} Mean-field trajectories and approximate steady state of species indexed by $\mathcal{S}_\mathcal{P}$, at $99\%$ of total information (see last line of Table \ref{tab:PH}). Dynamics of species Jnk-P is different in the reduced model. Compare this reduced model vs the one at $95\%$ of the total information in Figure \ref{fig:PH-trajs-full-red}. Trajectories of species previously included in that model are now correctly represented, but new species emerge. }
\end{figure}
To improve this model, as described in Section \ref{sec:augment}, we further include in the set of selected reaction channels $\mathcal{J}_\mathcal{P}$ all reaction channels in which species Jnk-P affects the stoichiometry. In Table \ref{tab:PH2} we show the results. Two reaction channels are then added, with indexes $j{=}45$ and $j{=}79$, with stoichiometry 
$$
S_{33} + S_{34} \xrightarrow{c_{45}} S_{32} + S_{34} \, , \qquad S_{33} + S_{33} \xrightarrow{c_{79}} S_{50} + S_{51} \, .
$$
Since parameters $c_{45}$ and $c_{79}$ are not pFIM-sensitive (and that is why these two reactions were not included in the original reduced model, plotted in Figure \ref{fig:PH-ss-99}), they will be included in the reduced model as constants and not parameters. The number of species remains unchanged. This new augmented model is shown in Table \ref{tab:PH2}. Notice the large drop in the information loss with respect to the information loss given in the last row of Table \ref{tab:PH}.
\begin{table}[h!]
\centering
\begin{footnotesize}\begin{tabular}{|c|c|c|c|c|c|c|}
\hline \rule{0pt}{2.6ex}
\textbf{pFIM \%}&$\bar{J}$&$\bar{K}$&$\bar{d}$&\textbf{Loss \eqref{eq:losse}}&\textbf{path-dist \eqref{eq:pdist}}&\textbf{SS-dist \eqref{eq:sdist}} \\\hline
99&15&13&18&0.0203601&0.002&0.001\\\hline
\end{tabular}
\end{footnotesize}
\caption{\label{tab:PH2} Improved reduced model at $99\%$ of total information. Notice the large drop in the information loss with respect to the information loss given in the last row of Table \ref{tab:PH}}
\end{table}


\subsubsection*{Discussion}
This example presents a prototypical case of a model in a regime where significant reductions are possible. By considering the parameters that accumulate at least $95\%$ of the total information as per \eqref{eq:pFIMpctg}, we observe a substantial reduction in terms of species and reaction channels while controlling the information loss and obtaining virtually identical mean-field trajectories. Only 12 out of 52 species are kept in the reduced model, and 10 out of 80 reaction channels. The number of parameters decreased from 87 to 10 (see Table \ref{tab:PH}).

We also presented the case that, for a given information threshold, there exists species for which the reduced model mean-field trajectories does not replicate the corresponding ones of the full model (see Figure \ref{fig:PH-ss-99}). We also showed how to improve such a reduced model iteratively (see Table \ref{tab:PH2}). In that respect, we demonstrated that iteration and validation components of our method are relevant for a meaningful reduction. 


\subsection{Mammalian circadian clock model}
\label{ex:circ}

\subsubsection*{Model description}
This reaction network consists of 16 species, 52 reactions and 52 parameters, the initial population and time interval, $[0,T]$, $T{=}72$, are taken from \cite{Leloup03} and \cite{databaseURL}. The variables in this model represent concentrations of species. The purpose of the original authors is to present a \emph{deterministic} model for the mammalian circadian clock.
This model presents oscillatory behavior in most of its species. 

\subsubsection*{Selecting parameters and reaction channels (Step 1)}
We start this example by noticing that 35 out of 52 parameters acccumulate at least $95\%$ (precisely, $95.051\%$) of the total information as per the pFIM diagonal \eqref{eq:pFIMpctg}, i.e., $\kappa{=}0.95$, as shown in Figure \ref{fig:CIRC_pFIM}.

\begin{figure}[h!]
\centering
\begin{minipage}{0.46\hsize}
	\includegraphics[width=\textwidth]{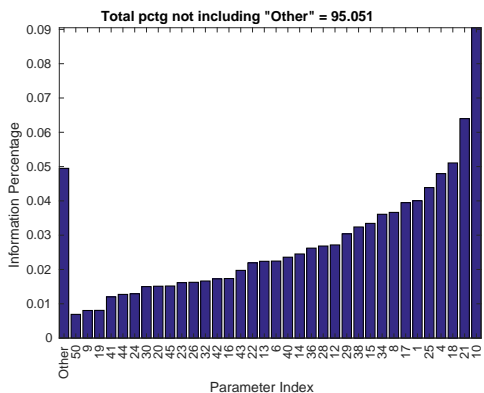}
\end{minipage}
\caption{\label{fig:CIRC_pFIM} Total information as per the pFIM diagonal \eqref{eq:pFIMpctg}. Notice that 35 out of 52 parameters accumulate at least $95\%$ (precisely, $95.051\%$) of the total information as per the pFIM diagonal.}
\end{figure}


\begin{figure}[h!]
\centering
\begin{minipage}{0.48\hsize}
	\includegraphics[width=\textwidth]{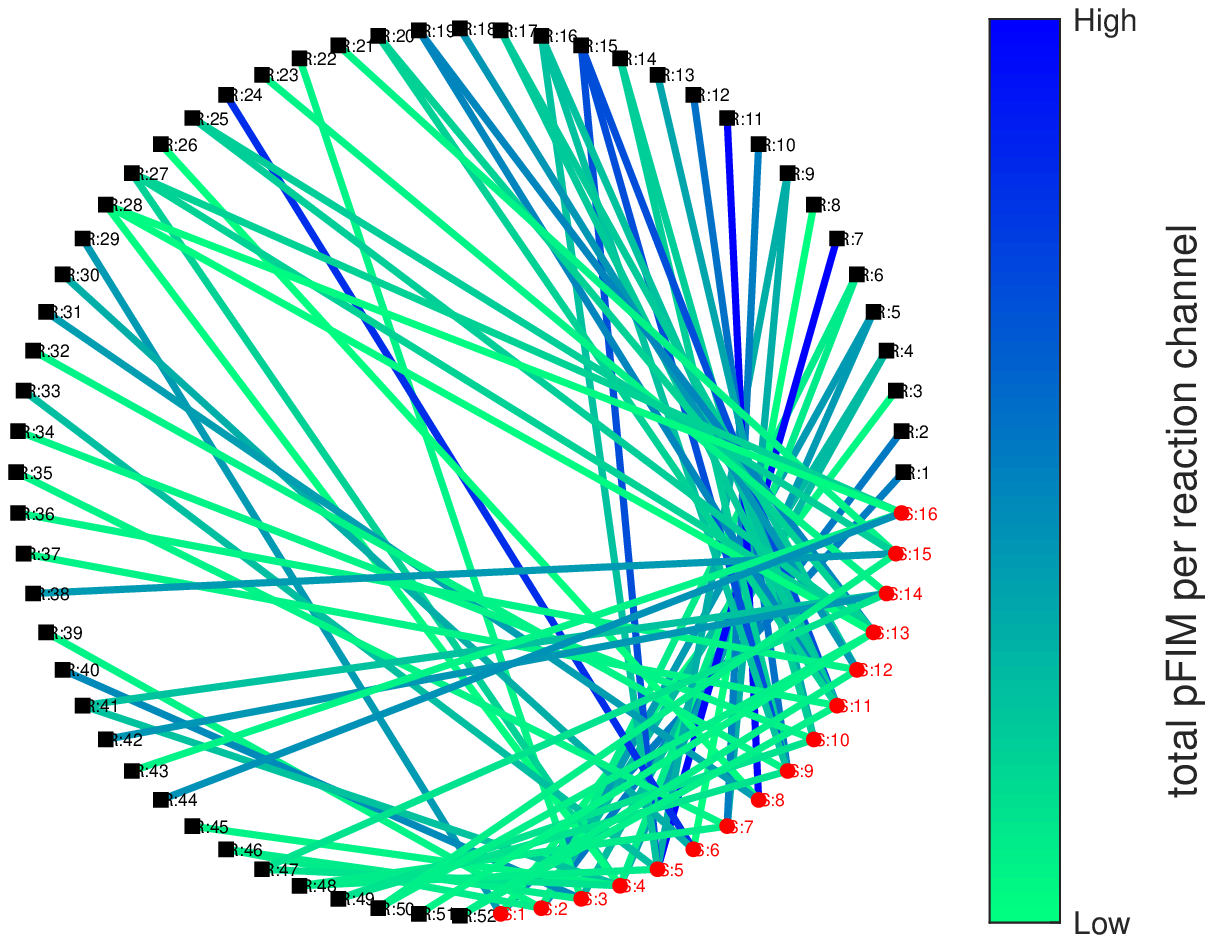}
\end{minipage}
\hfill
\begin{minipage}{0.48\hsize}
	\includegraphics[width=\textwidth]{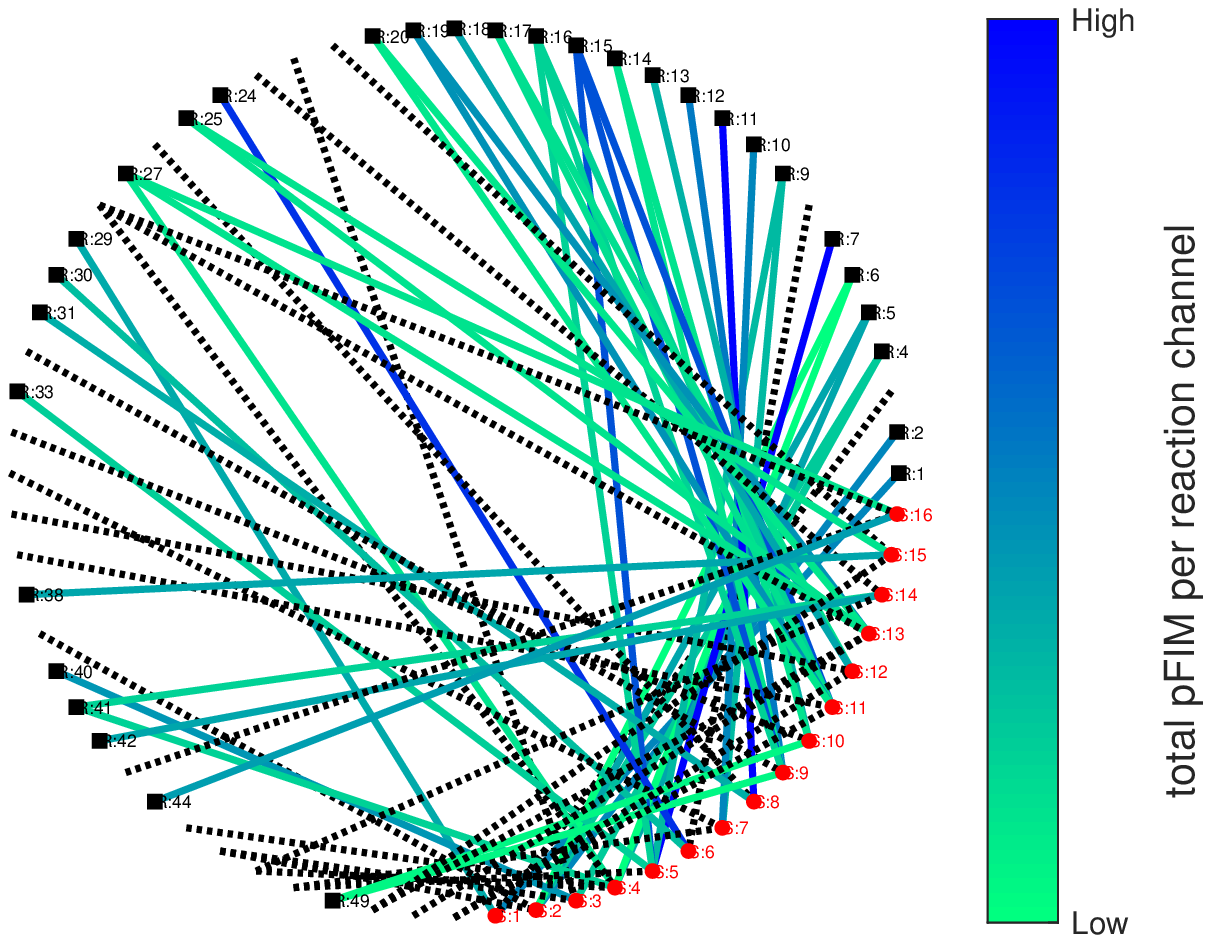}
\end{minipage}
\caption{\label{fig:CIRC_graph} Stoichiometry and information. Black dots: reaction channels, red dots: species. A link shows stoichiometry of each reaction channel. The color of the link shows the total pFIM of the corresponding reaction channel (see \ref{eq:link_color}). Left: Full model. Right: Reduced model, including at least $95\%$ of the total information as per the pFIM diagonal \eqref{eq:pFIMpctg}. In the plot, species indexes are shown, while species names are given in Figure \ref{fig:CIRC-trajs-full-red}.}
\end{figure}

\subsubsection*{Selecting variables (Step 2)}
In Figure \ref{fig:CIRC_graph} (right pane) we show the sensitive reaction channels \eqref{eq:sens_chans} and corresponding species \eqref{eq:sens_spec}, at $\kappa{=}0.95$. Dotted lines shows, for the selected species (i.e., included in the reduced model), the stoichiometry relations with reaction channels considered not sensitive and therefore not included in the reduced model. Some sparsity can be observed in the stoichiometry of this reaction network.

\subsubsection*{Reduced model (Step 3 and Step 4)}
In Figure \ref{fig:CIRC-trajs-full-red} we compare the mean-field trajectories of the species included in the reduced model at $95\%$ of information, and the respective mean-field trajectories of the same species in the full model. We also show a comparison between the time average of the full model versus the corresponding one of the reduced model at $\kappa{=}0.95$. Names of the species included in the reduced model are given in both plots of Figure \ref{fig:CIRC-trajs-full-red}.

\begin{figure}[h!]
\begin{minipage}{0.48\hsize}
	\includegraphics[width=\textwidth]{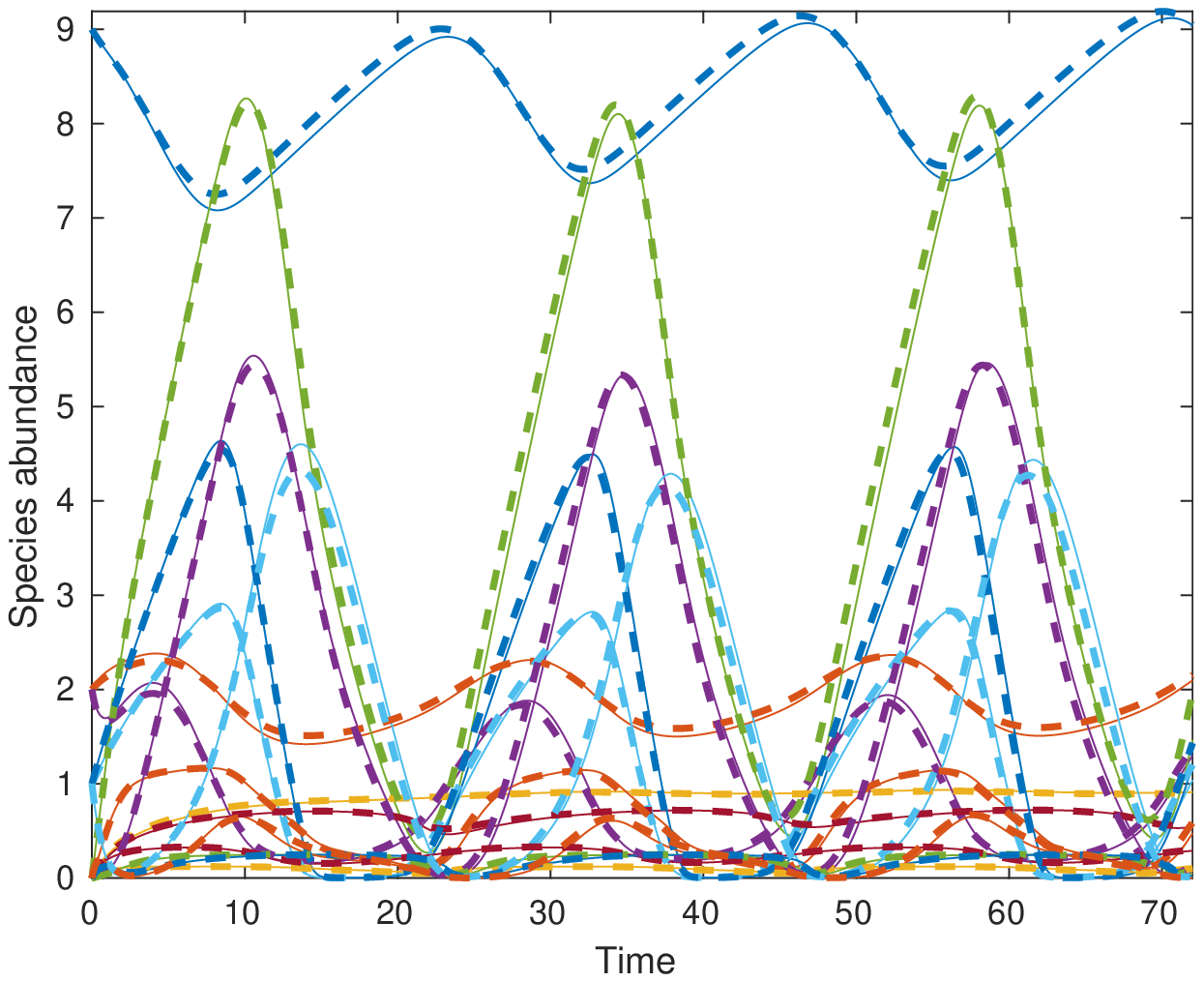}
\end{minipage}
\hfill
\begin{minipage}{0.48\hsize}
	\includegraphics[width=\textwidth]{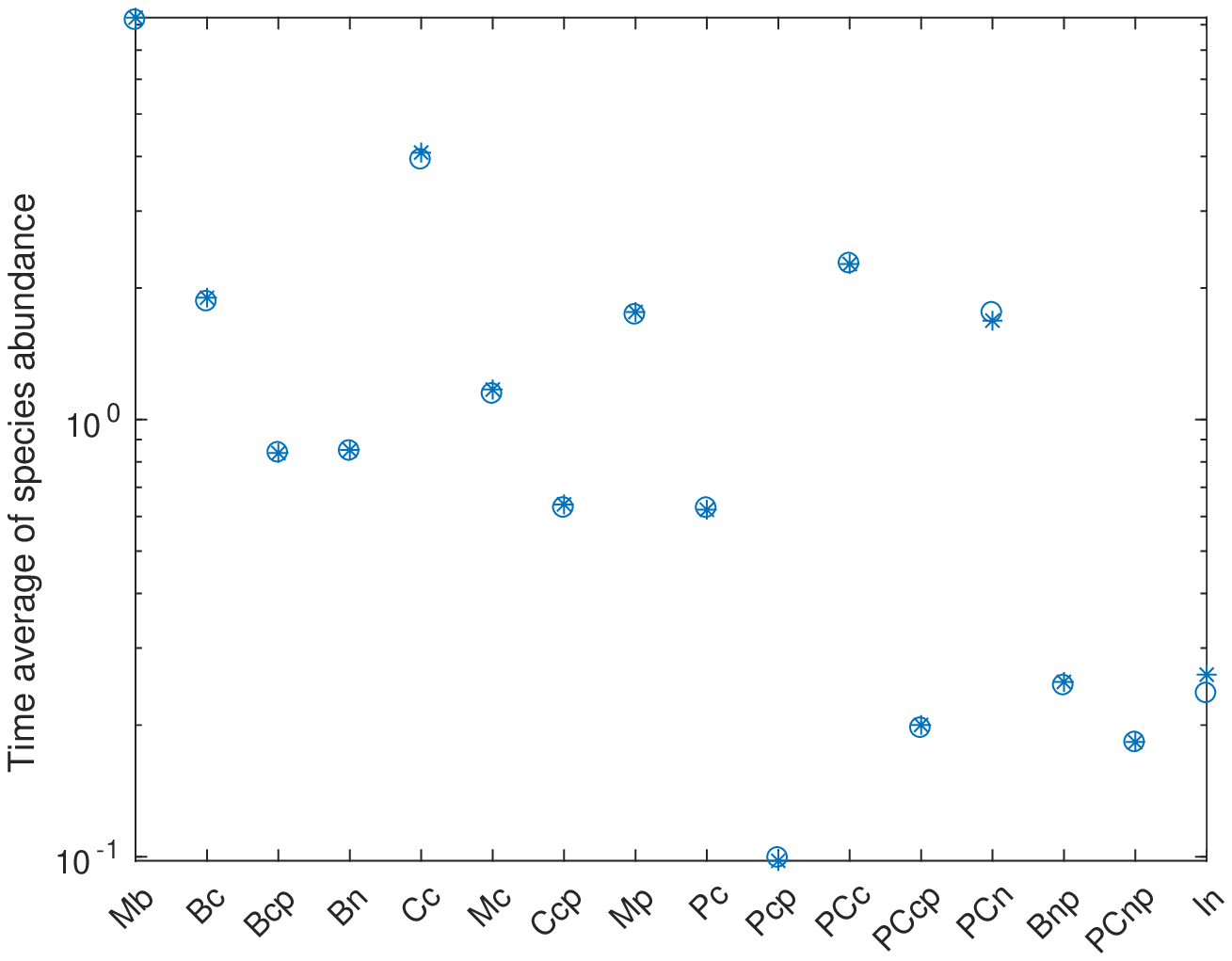}
\end{minipage}
\caption{\label{fig:CIRC-trajs-full-red} Left: Mean-field trajectories of species indexed by $\mathcal{S}_\mathcal{P}$ \eqref{eq:sens_spec}. Solid lines correspond to the reduced model and thick dashed lines to the full model. Right: Approximation of the mean-field steady state of species in the reduced model (circle), vs. the full model (star), for $t\in [0,T]$ at $95\%$ of total information.}
\end{figure}

\subsubsection*{Validation (Step 5)}
In Table \ref{tab:CIRC} we show different reduced models, according to total pFIM information. Every reduced model includes all 16 species, so every reduced model is already  comparable in terms of the validation distances. Notice there are at least two candidate models. At 97\% of total information we have virtually the same trajectories (not shown in the figures) but only 7 out of 52 reactions channels and 13 out of 52 parameters are eliminated. At 95\%, 21 out of 52 reaction channels are reduced and 17 parameters out of 52  are reduced but the mean-field trajectories of the reduced model seems out of phase with respect to the ones of the full model.

%

\begin{table}[h!]
\centering
\begin{footnotesize}\begin{tabular}{|c|c|c|c|c|c|c|}
\hline \rule{0pt}{2.6ex}
\textbf{pFIM \% \eqref{eq:pFIM}}&$\bar{J}$&$\bar{K}$&$\bar{d}$&\textbf{Loss \eqref{eq:losse}}&\textbf{path-dist \eqref{eq:pdist}}&\textbf{SS-dist \eqref{eq:sdist}}\\\hline
93&29&33&16&0.0109498&4.506&0.267\\\hline
95&31&35&16&0.000820265&0.597&0.083\\\hline
97&45&39&16&0.000381857&0.082&0.040\\\hline
99&49&45&16&0.000172741&0.102&0.019\\\hline
\end{tabular}
\end{footnotesize}
\caption{\label{tab:CIRC} Different reduced models as per total pFIM information. Since the number of species in every model is the same, the validation distances of the last two columns are applied to the same set of species. Notice the monotone decrease of the loss function.}
\end{table}

\subsubsection*{Discussion}
This model presents non-trivial oscillatory behaviour and there is no reduction on the number of species. Despite of this, the number of reaction channels and parameters are still reduced from 52 to 31 and from 52 to 35 respectively, by considering $95\%$ of the total information as per the pFIM diagonal (see Table \ref{tab:CIRC}). The mean-field trajectories of this reduced model reasonably replicates the original ones (see Figure \ref{fig:CIRC-trajs-full-red}). Notice, however, a phase shift in some of its trajectories. Increasing the total pFIM to $97\%$, we obtain superior replication of the trajectories, but more parameters and reaction channels are kept in the reduced model.

\subsection{Epidermal Growth Factor Receptor (EGFR) model}
\label{ex:EGFR}
\subsection*{Model description}

This example is a well-studied model that describes signaling phenomena of mammalian cells, regulating its growth, survival and proliferation playing a crucial role in many biological processes. 
This reaction network consists of 24 species, 47 reaction channels and 50 parameters. The variables account for species concentrations. It has mass-action kinetics type and Michaelis-Menten approximation type of propensities. This model of signalling phenomena has a transient regime that corresponds to the time interval $[0, 50]$ and also a stationary regime for $T{>}50$. Here we consider the transient regime and the reaction constants, the initial population  are taken from \cite{Kholodenko99}.

\subsubsection*{Selecting parameters and reaction channels (Step 1)}
We start the analysis by noticing that that 25 out of 50 parameters accumulate at least $97\%$ (precisely, $97.244\%$) of the total information as per the pFIM \eqref{eq:pFIMpctg}, as shown In Figure \ref{fig:EGFR_pFIM}.

\begin{figure}[h!]
\centering
\begin{minipage}{0.46\hsize}
	\includegraphics[width=\textwidth]{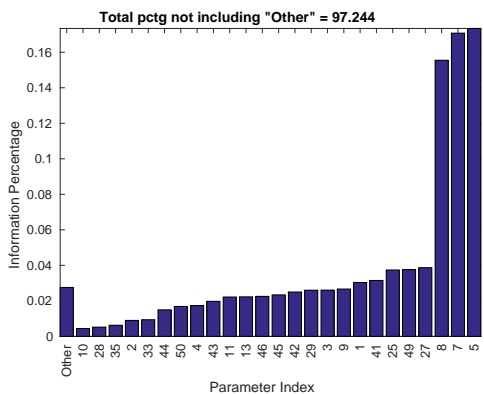}
\end{minipage}
\caption{\label{fig:EGFR_pFIM} Total information as per the pFIM. Notice that 25 out of 50 parameters acccumulate $97.244\%$ of the total information as per the pFIM diagonal (see \eqref{eq:pFIMpctg}).}
\end{figure}

%
\begin{figure}[h!]
\centering
\begin{minipage}{0.48\hsize}
	\includegraphics[width=\textwidth]{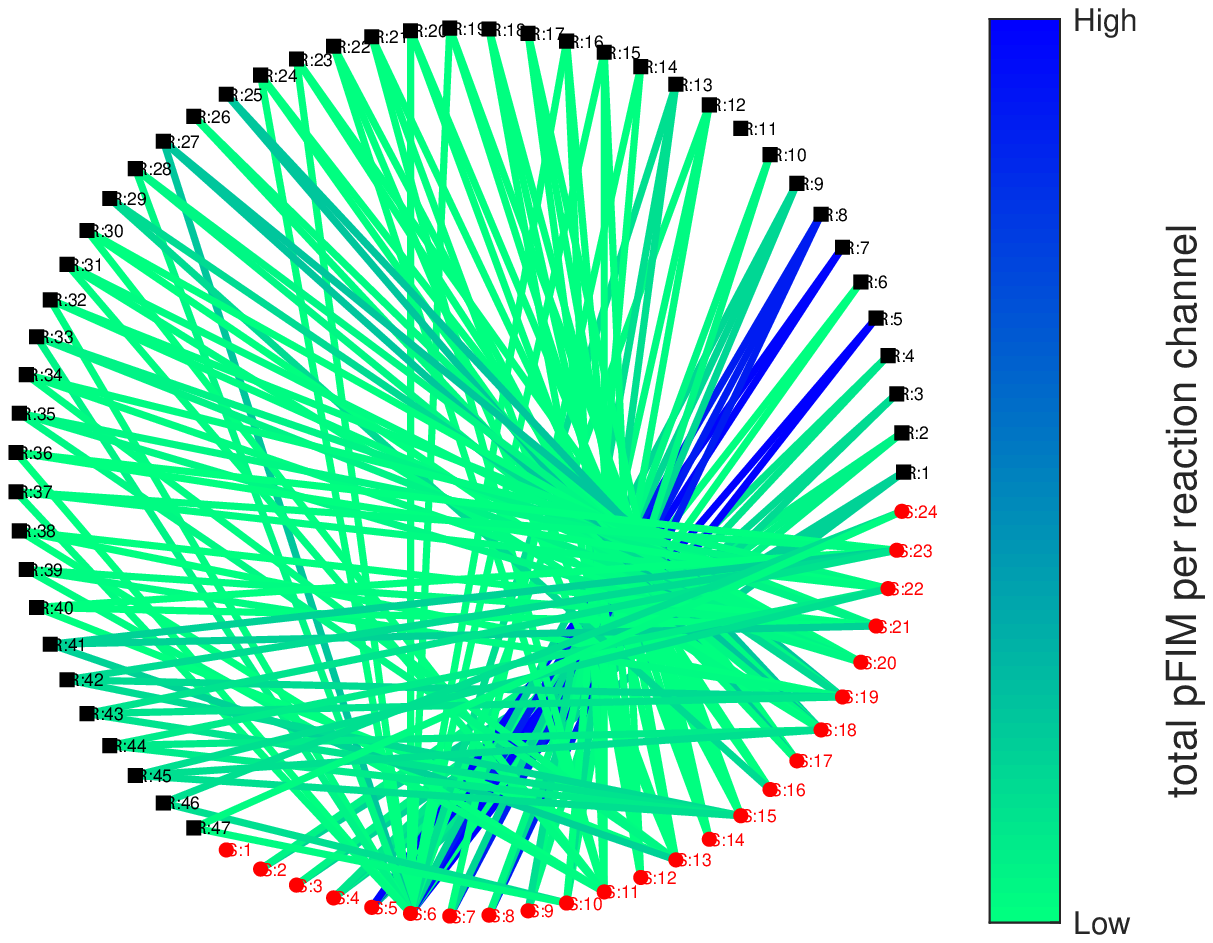}
\end{minipage}
\hfill
\begin{minipage}{0.48\hsize}
	\includegraphics[width=\textwidth]{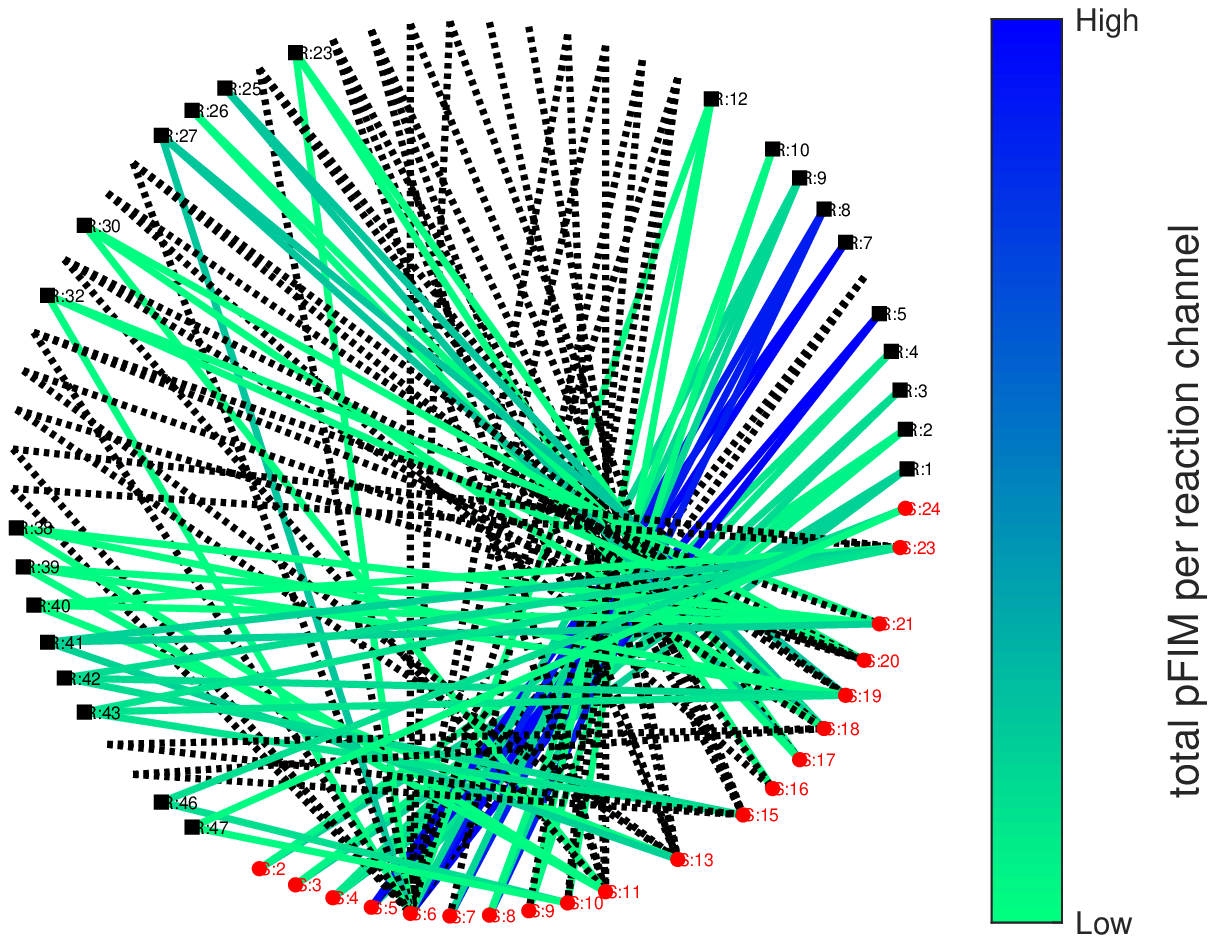}
\end{minipage}
\caption{\label{fig:EGFR_graph} Stoichiometry and information. Black dots: reaction channels, red dots: species. A link shows stoichiometry of each reaction channel. The color of the link shows the total pFIM of the corresponding reaction channel (see \ref{eq:link_color}). Left: Full model. Right: Reduced model, which includes at least $97\%$ of the total information as per the pFIM diagonal \eqref{eq:pFIMpctg}. In the plot, species indexes are shown, while species names are given in Figure \ref{fig:EGFR-trajs-full-red}.}
\end{figure}

\subsubsection*{Selecting variables (Step 2)}
In Figure \ref{fig:EGFR_graph} (right pane) we show the sensitive reaction channels \eqref{eq:sens_chans} and the corresponding species \eqref{eq:sens_spec}, that represent at least $97\%$ of total pFIM diagonal information. Dotted lines shows, for the selected species (i.e., associated with sensitive reaction channels), the stoichiometry relations with reaction channels not considered sensitive, and therefore, not included in the reduced model.  We can observe that many reaction channels not considered sensitive affect many selected species. In principle, for a fixed $\kappa$ this may affect the quality of the reduced model, since many reactions that may increase or decrease the populations/concentrations of the selected species are not included in the reduced model. As seen in Figure \ref{fig:EGFR-trajs-full-red}, those missing reaction channels seems to be non essential. 

\subsubsection*{Reduced model (Step 3 and Step 4)}
In Figure \ref{fig:EGFR-trajs-full-red} we compare the mean-field trajectories of the species included in the reduced model at $\kappa{= }0.97$, and the respective mean-field trajectories of the same species in the full model. We also show a comparison for the time average on $[0,T]$, of the full model versus the reduced model at $\kappa{=}0.97$. Most of the trajectories and time averages are correctly represented in the reduced model. The validation distances \eqref{eq:pdist} and \eqref{eq:sdist}, which are properly normalized, are shown in Table \ref{tab:EGFR}.

\begin{figure}[h!]
\begin{minipage}{0.48\hsize}
	\includegraphics[width=\textwidth]{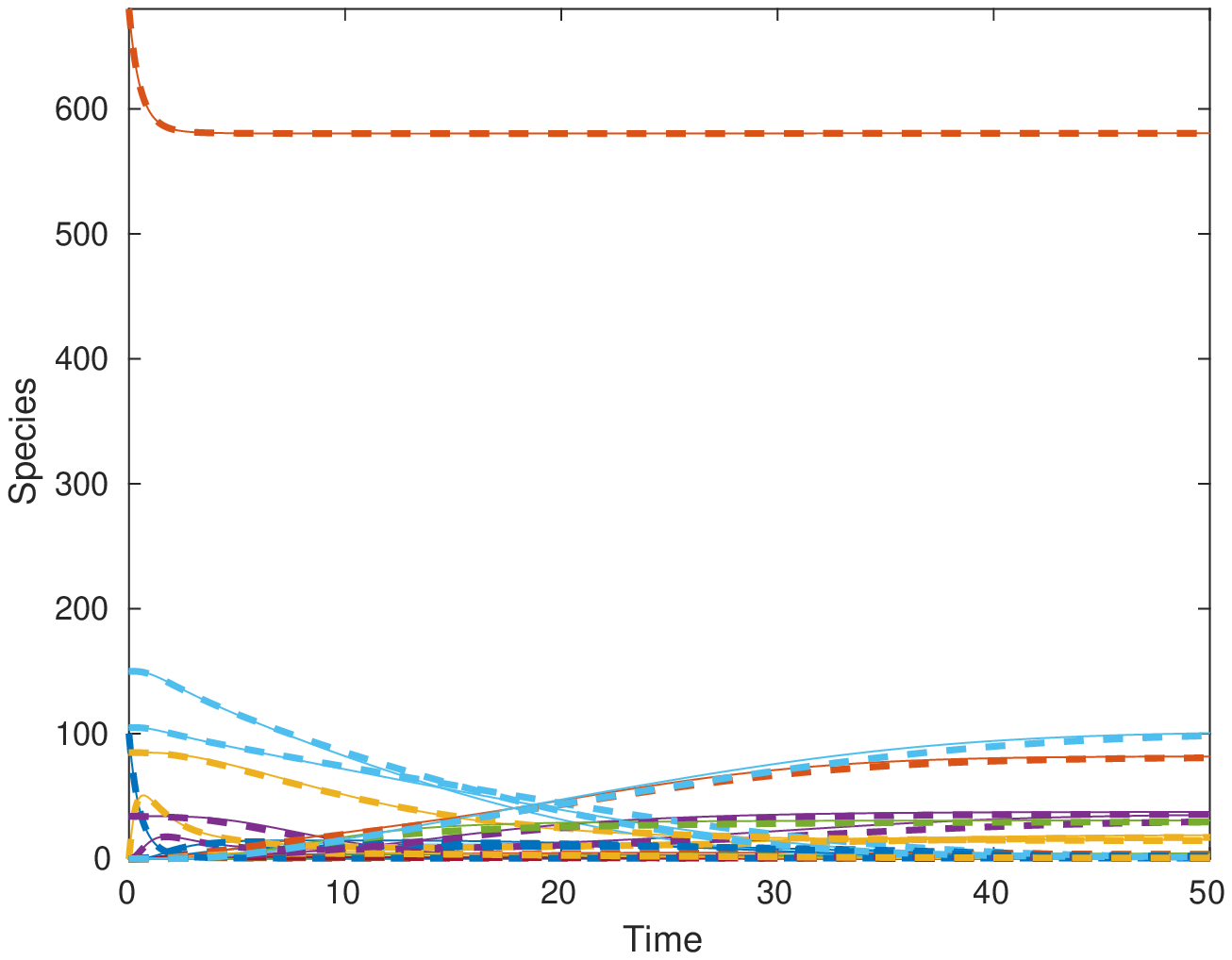}
\end{minipage}
\hfill
\begin{minipage}{0.46\hsize}
	\includegraphics[width=\textwidth]{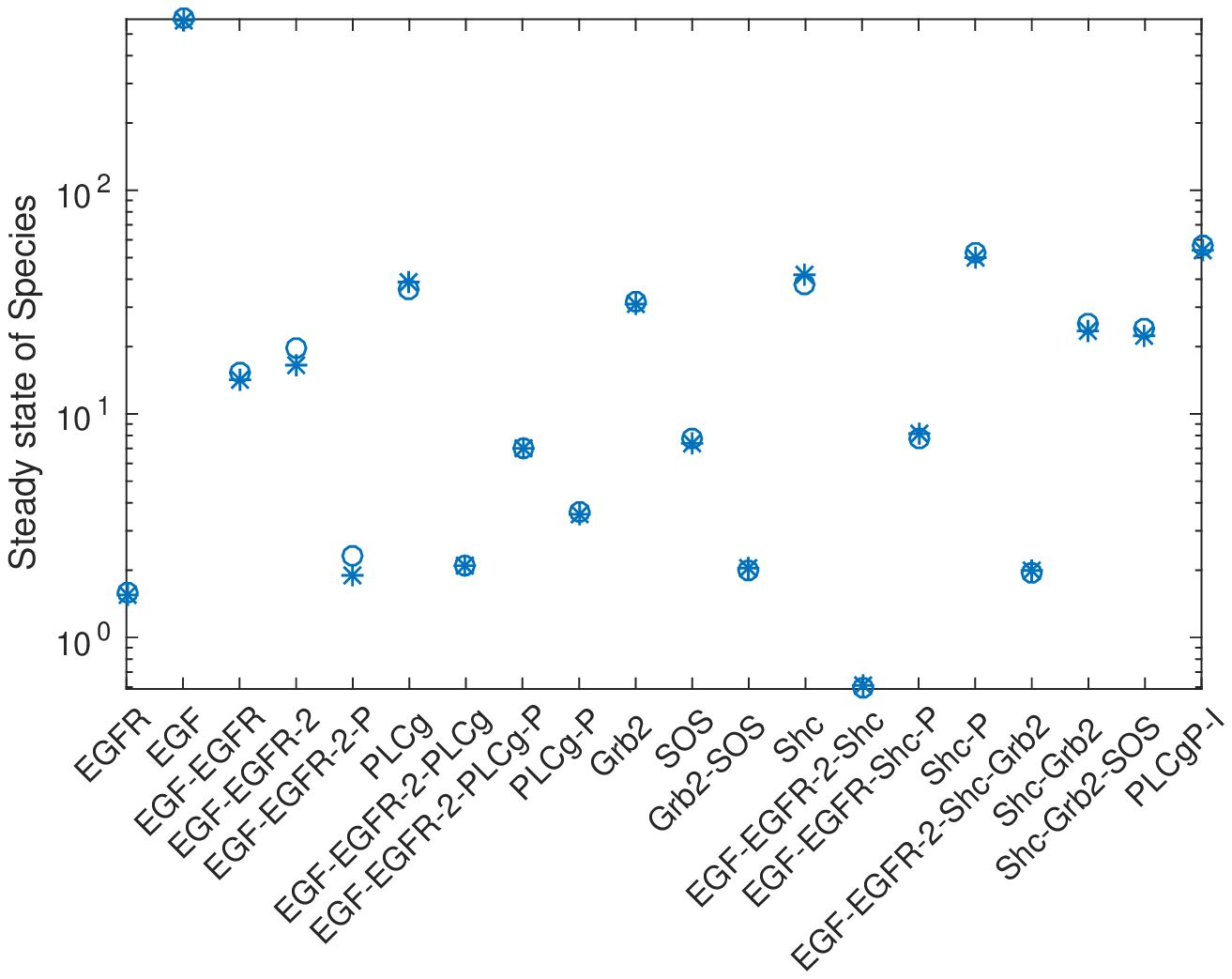}
\end{minipage}
\caption{\label{fig:EGFR-trajs-full-red} Left: Mean-field trajectories of species indexed by $\mathcal{S}_\mathcal{P}$ \eqref{eq:sens_spec}. Solid lines correspond to the reduced model and thick dashed lines to the full model. Right: Time average on $[0,T]$ for the reduced model (circle), vs. the full model (star) at $97\%$ of total information per pFIM diagonal.}
\end{figure}

\subsubsection*{Validation (Step 5)}
In Table \ref{tab:EGFR} we show different reduced models, according to total information as per the pFIM diagonal. Column ``Loss'' shows the value of the loss function \eqref{eq:losse} at which $\theta_{0:T}^*$ is achieved, by numerically solving problem \eqref{eq:opt_e}. Observe the monotone reduction of the information loss, it has two large decays when going from $93\%$ to $95\%$ and from $97\%$ to $99\%$.
Distances shown in the last two columns are given by \eqref{eq:pdist} and \eqref{eq:sdist} respectively, where the set $\mathcal{O}$ is given by the set of species associated with each respective reduced model. That is, the set $\mathcal{O}$ associated with the second line ($95\%$) correspond to the 20 species selected in that reduced model.
Notice that the number of selected species remain unchanged  when increasing the total information from $95\%$ to $97\%$, but three additional reaction channels are added to the reduced model. This results in lower information loss  and also lower pathwise and steady state distance of the species trajectories. 
Model at $97\%$ of total information as per the pFIM seems a good compromise between information loss and pathwise distances in comparison the reduction of  species and reaction channels.
\begin{table}[h!]
\centering
\begin{footnotesize}\begin{tabular}{|c|c|c|c|c|c|c|}
\hline \rule{0pt}{2.6ex}
\textbf{pFIM \% \eqref{eq:pFIM}}&$\bar{J}$&$\bar{K}$&$\bar{d}$&\textbf{Loss \eqref{eq:losse}}&\textbf{path-dist \eqref{eq:pdist}}&\textbf{SS-dist \eqref{eq:sdist}}\\\hline
93&19&20&19&1.45047&0.681&0.353\\\hline
95&21&22&20&0.464843&7.776&3.824\\\hline
97&24&25&20&0.16691&0.638&0.214\\\hline
99&31&32&21&0.0469859&0.437&0.093\\\hline
\end{tabular}
\end{footnotesize}
\caption{\label{tab:EGFR} Different reduced models as per total pFIM information for the EGFR model. Observe the monotone decrease of the information loss even though model has different number of species. Distances shown in the last two columns are given by \eqref{eq:pdist} and \eqref{eq:sdist} respectively, where the set $\mathcal{O}$ is given by the set of species associated with each respective reduced model.}
\end{table}

\subsubsection*{Discussion}
We observe a significant reduction in terms of species, reaction channels and parameters, while controlling the entropy loss and obtaining virtually identical mean-field trajectories for almost all the species. The number of species in the reduced model (at $97\%$ of total information) is 20 (out of 24) and the number of reaction channels is 24 (out of 47) while the number of parameters is 25 (out of 50). Further details can be found on Table \ref{tab:EGFR}.

\section{Numerical experiments: Stochastic models}
\label{sec:stoc_exp}
In this section we present two stochastic pure jump models: a circadian clock model which is based on \cite{Leloup03}, and a growth factor receptor, based on \cite{Kholodenko99}. By scaling the original models, we obtain stochastic pure jump representations that manifest non-Gaussian distributions of the time average of species counts for many of the species. We show in this section that our method is robust in the sense that it can be used to either reduce a model which is close to a Langevin or mean-field regime and also to reduce a pure jump model with low count in many of its species. Therefore, once the reduced network $(\bar\nu_j,\bar a_j)_{j=1}^{\bar J}$ is constructed, a pure jump stochastic, a Langevin diffusion or a deterministic representation can be readily obtained up to the corresponding scalings.

\subsection{A stochastic circadian clock model}
\label{ex:stoc_CIRC}
In this section we focus on a stochastic circadian model which is derived  from the deterministic model presented in the previous section, to  demonstrate the robustness of the method in a challenging non-Gaussian regime. 
 Using Kurtz's classical scaling (\cite{KURTZ1978223}) we obtain a stochastic pure jump model from the original deterministic one to perform a statistical validation. We first observe that the original model can be written in integral form (see \eqref{eq:rrODE} for the differential form) as
$$
z(t) = z(0) + \sum_{j=1}^J  \nu_j \int_0^t a_j(z(s);c) ds  \, , 
$$
and $z(0) = z_0$ where $z(t) \in \mathbb{R}^d$ represents the concentration of the species. Now given a parameter that measures the size of the system, $N$, we obtain the following counting process that \emph{approximates} $z(t)$
$$
X(t) = X(0) + \sum_{j=1}^J  \nu_j Y_j \Big(N \int_0^t \bar{a}_j(N^{-1}X(s);c) ds \Big) \, ,
$$
and $X(t)=N z_0$ where $\mathcal{P}_j$ are independent unit-rate Poisson processes and $\bar{a}_j$ are the kinetic propensity functions, independent of $N$. When $N \rightarrow \infty$ it can be shown that the mean of $N^{-1}X$ converge pathwise to  $z$ (for further details and more advanced scaling limits we refer to \cite{kang2013}). In this example, for $N{=}10^5$ a stochastic trajectory is indistinguishable from $z$. In figure \ref{fig:CIRCs} (left) we show one stochastic trajectory for $N{=}2$. 

We compute the pFIM for this stochastic network (see \eqref{eq:pFIM_SRN}) by sampling $M{=}10^3$ trajectories of $X(t)$, for $t\in [0,T]$, $T{=}72$. Using the same trajectories, we compute the mean time average of each species $S_i$
\begin{equation}
\label{eq:mean_time_avg_spec}
\mathcal{A}(X_i):=\frac{1}{M} \sum_{m=1}^M  \frac{1}{T} \sum_{k=1}^{T^{(m)}}x_{k,m} \Delta_{t_{k,m}} \, , \quad i{=}1,2,..., d \, ,
\end{equation}
where $T^{(m)}$ is the number of jumps of the $m-$th trajectory. 
Using this sample, we construct a bootstrap estimation of the mean of the time average and the corresponding 95\%  confidence interval (together for comparison with the time average of $z$), as shown in Figure \ref{fig:CIRCs} (left). We observe that, for many species, the time average of the deterministic model is not included in a 95\% confidence interval of the mean time average of the stochastic model. This is reasonable because in the small particle count case, species may get extinct (hit the zero boundary), among many other possible stochastic behaviour not captured by the deterministic model. It is clear that this stochastic model, derived from the original deterministic one, is not in a Langevin or mean-field regime.

\begin{figure}[h!]
\centering
\begin{minipage}{0.48\hsize}
	\includegraphics[width=\textwidth]{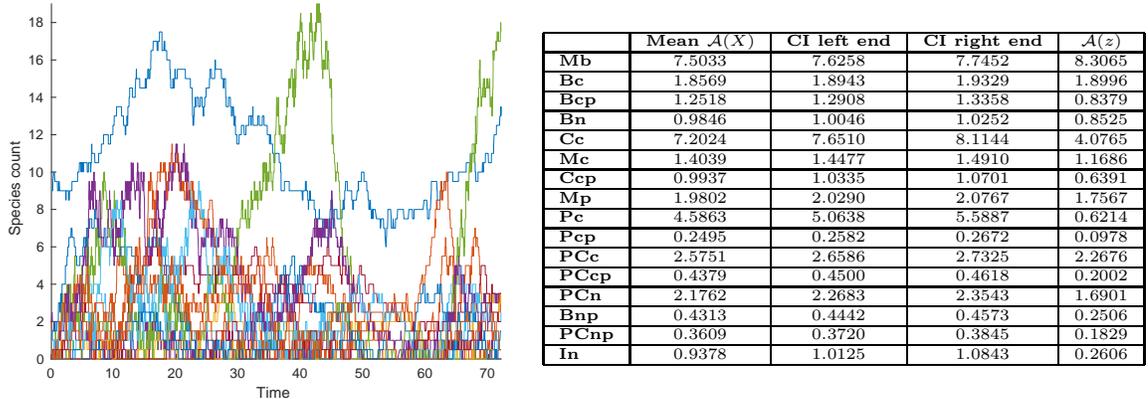}
\end{minipage}
\hfill
\begin{minipage}{0.48\hsize}
\begin{tiny}\begin{tabular}{|l|c|c|c|c|}
\hline
&\textbf{Mean $\mathcal{A}(X)$}&\textbf{CI left end}&\textbf{CI right end}&$\mathcal{A}(z)$\\\hline
\textbf{Mb}&7.5033&7.6258&7.7452&8.3065\\\hline
\textbf{Bc}&1.8569&1.8943&1.9329&1.8996\\\hline
\textbf{Bcp}&1.2518&1.2908&1.3358&0.8379\\\hline
\textbf{Bn}&0.9846&1.0046&1.0252&0.8525\\\hline
\textbf{Cc}&7.2024&7.6510&8.1144&4.0765\\\hline
\textbf{Mc}&1.4039&1.4477&1.4910&1.1686\\\hline
\textbf{Ccp}&0.9937&1.0335&1.0701&0.6391\\\hline
\textbf{Mp}&1.9802&2.0290&2.0767&1.7567\\\hline
\textbf{Pc}&4.5863&5.0638&5.5887&0.6214\\\hline
\textbf{Pcp}&0.2495&0.2582&0.2672&0.0978\\\hline
\textbf{PCc}&2.5751&2.6586&2.7325&2.2676\\\hline
\textbf{PCcp}&0.4379&0.4500&0.4618&0.2002\\\hline
\textbf{PCn}&2.1762&2.2683&2.3543&1.6901\\\hline
\textbf{Bnp}&0.4313&0.4442&0.4573&0.2506\\\hline
\textbf{PCnp}&0.3609&0.3720&0.3845&0.1829\\\hline
\textbf{In}&0.9378&1.0125&1.0843&0.2606\\\hline
\end{tabular}
\end{tiny}
\end{minipage}
\caption{\label{fig:CIRCs} Table: Mean time average of each species with the corresponding 95\% empirical confidence interval and the time average of $z$ (the time average denoted by $\mathcal{A}$), for the stochastic circadian model. We can observe that, for some species, there are significant differences (e.g. species \emph{Pc}).}
\end{figure}

We now reduce this stochastic model by using our method. We show in Table \ref{tab:CIRCs} the results. For an equivalent number of parameters, we have consistent pFIM and validation distances in comparison with the deterministic model. This is a remarkable robustness of our reduction method, and stress the point that it can be applied to deterministic, Langevin or pure jump reaction networks. For example, we compare the reduced model obtained by applying our method to the deterministic model versus the reduced model of the stochastic circadian model. Compare for instance the model with 35 parameters and 31 reaction channels that accumulates $95.05\%$ of the deterministic pFIM (Table \ref{tab:CIRC}) versus the model with 34 parameters and 30 reaction channels that accumulates $95.47\%$ of the stochastic pFIM (Table \ref{tab:CIRCs}). In fact, we obtained a slightly more parsimonious model (less parameters and reaction channels) with smaller validation distances. We notice that the validation distances are mean-field based (see \eqref{eq:pdist} and \eqref{eq:sdist}).

\begin{table}[h!]
\centering
\begin{footnotesize}\begin{tabular}{|c|c|c|c|c|c|c|}
\hline \rule{0pt}{2.6ex}
\textbf{pFIM \% \eqref{eq:pFIM}}&\textbf{$\bar{J}$}&\textbf{$\bar{K}$}&\textbf{$\bar{d}$}&\textbf{Loss \eqref{eq:losse}}&\textbf{path-dist \eqref{eq:pdist}}&\textbf{SS-dist \eqref{eq:sdist}}\\\hline
91.96&28&30&16&0.119118&2.204&0.268\\\hline
92.94&29&31&16&0.0570935&3.125&0.268\\\hline
93.91&30&32&16&0.00611969&0.681&0.051\\\hline
94.69&30&33&16&0.00542096&0.698&0.055\\\hline
95.46&30&34&16&0.00533519&0.421&0.061\\\hline
96.08&30&35&16&0.00535939&0.663&0.050\\\hline
96.64&42&36&16&0.0053336&0.179&0.038\\\hline
97.20&43&37&16&0.00490005&0.215&0.043\\\hline
\end{tabular}
\end{footnotesize}

\caption{\label{tab:CIRCs} Different reduced stochastic models as per total pFIM information, for the stochastic circadian model. Since the number of species in every model is the same, the validation distances of the last two columns are applied to the same set of species.}
\end{table}

In Figure \ref{fig:CIRCs-box} we show a box-plot comparison between the full stochastic model and the reduced model at 95.46\% of total information. On each x-axis we show a box for the species in the full model and next a box for the same species in the reduced model. We can observe a tight agreement on the vast majority of the resolved species. In Figure \ref{fig:CIRCs-hist} we show histograms for 3 representative species with non-Gaussian behaviour showing tight agreements.

\begin{figure}[h!]
\begin{minipage}{0.325\hsize}
	\includegraphics[width=1.1\textwidth]{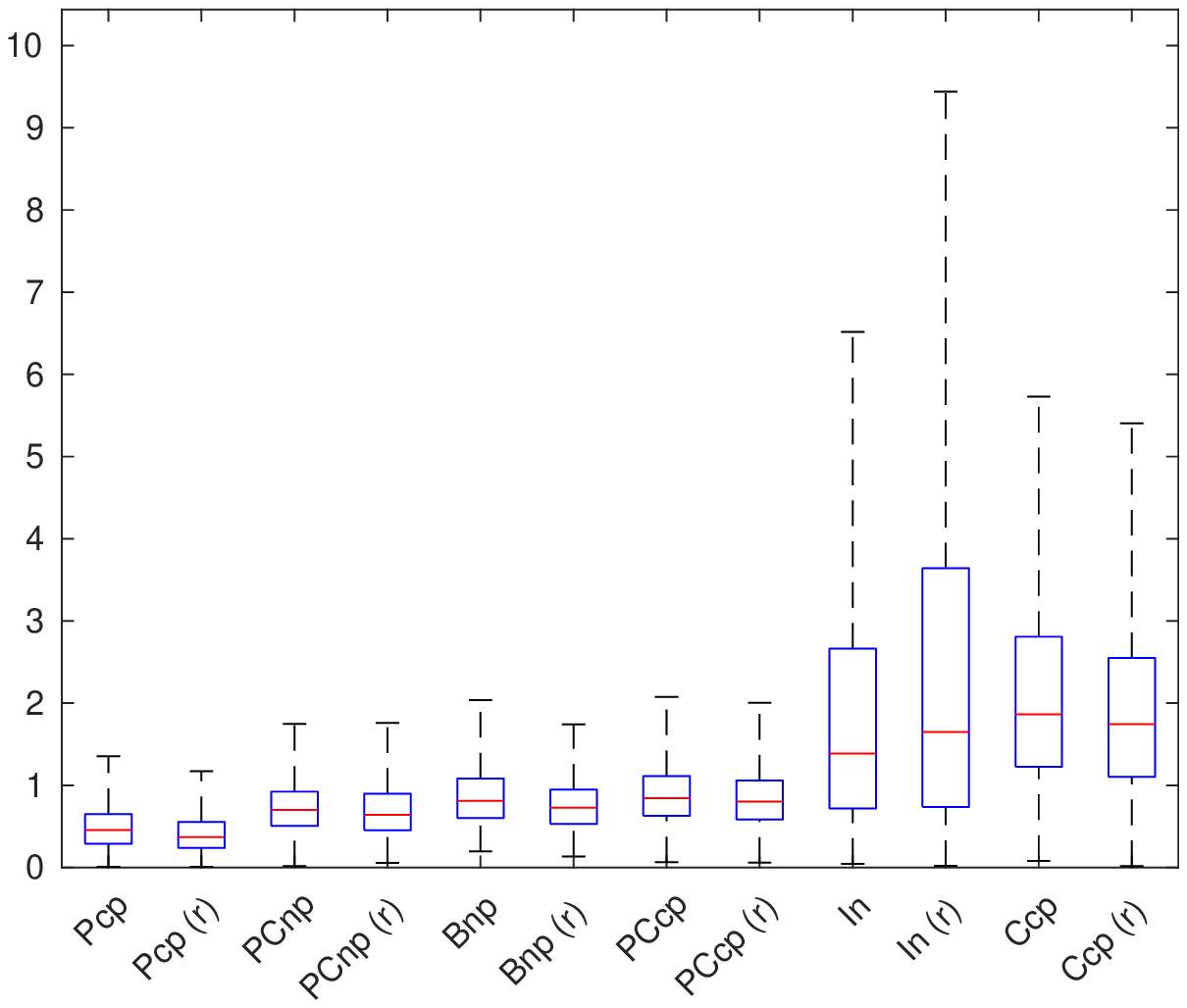}
\end{minipage}
\begin{minipage}{0.325\hsize}
	\includegraphics[width=1.1\textwidth]{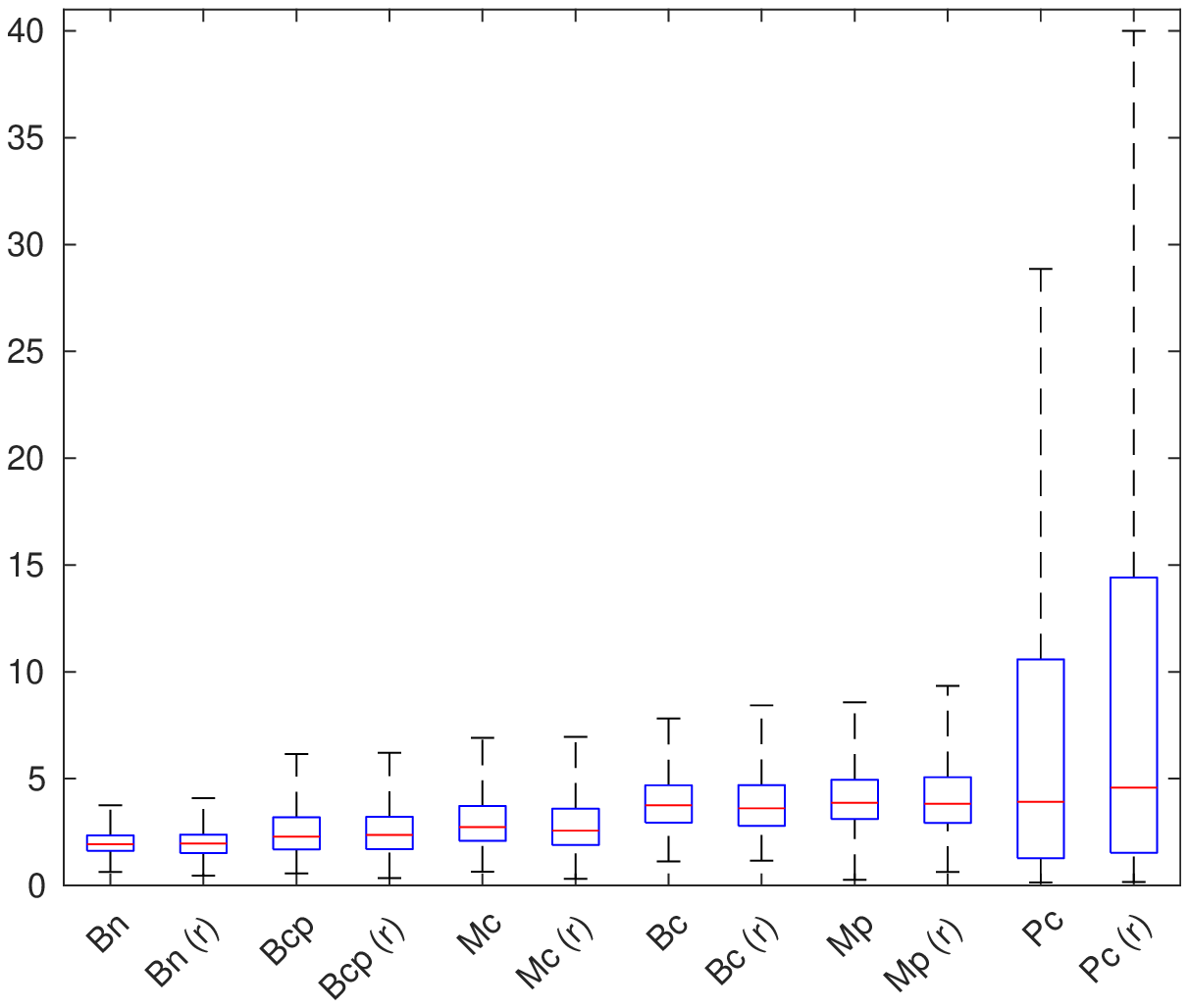}
\end{minipage}
\begin{minipage}{0.325\hsize}
	\includegraphics[width=1.1\textwidth]{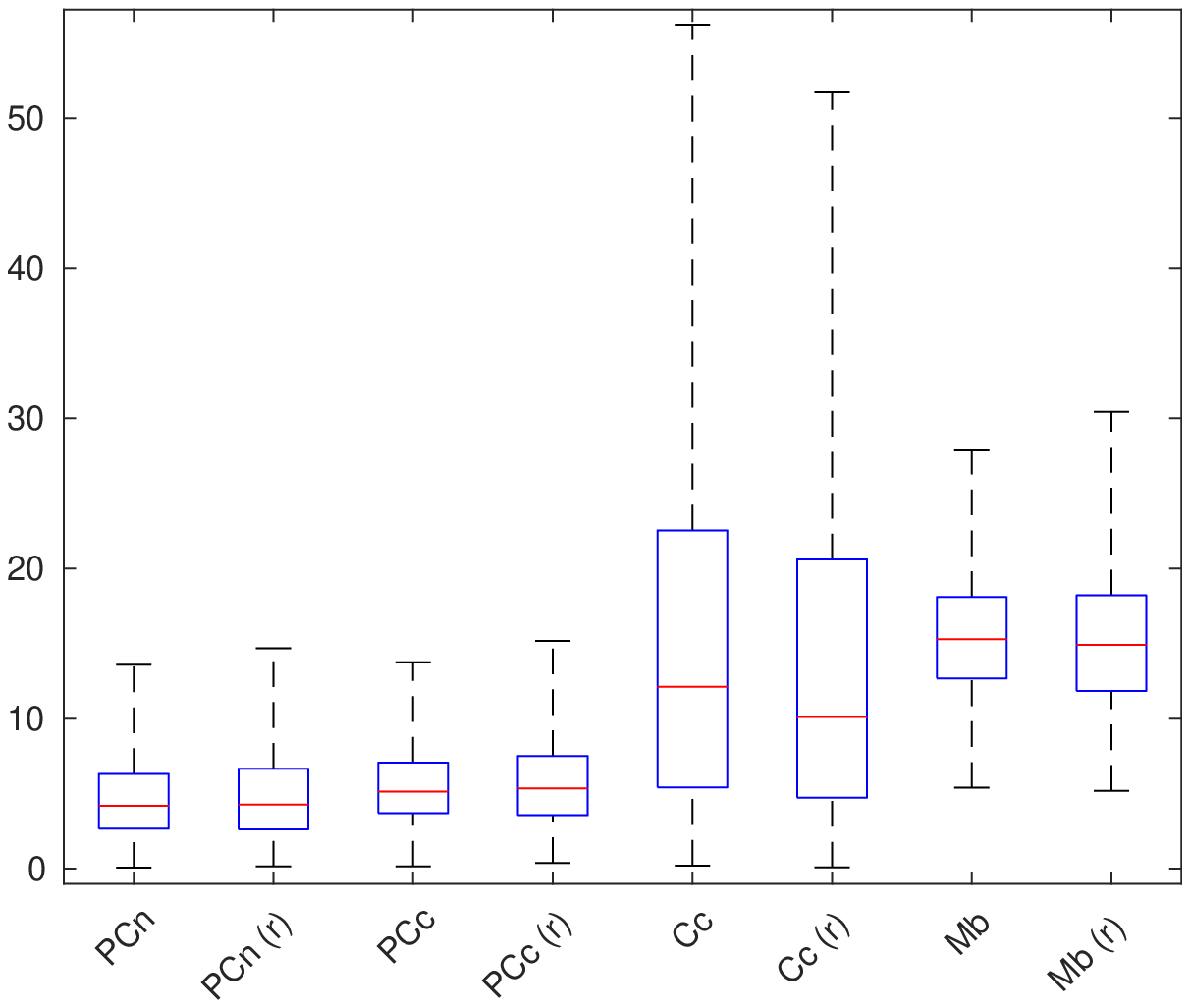}
\end{minipage}
\caption{\label{fig:CIRCs-box} Box plot comparison of time average of species,  for the stochastic circadian model. On each box, the central mark indicates the median, and the bottom and top edges of the box indicate the 25th and 75th percentiles, respectively of the time average. The whiskers extend to the most extreme data points not considering outliers. Species name with suffix ``(r)'' correspond to the reduced model at 95.46\%. }
\end{figure}

\begin{figure}[h!]
\begin{minipage}{0.325\hsize}
	\includegraphics[width=1.1\textwidth]{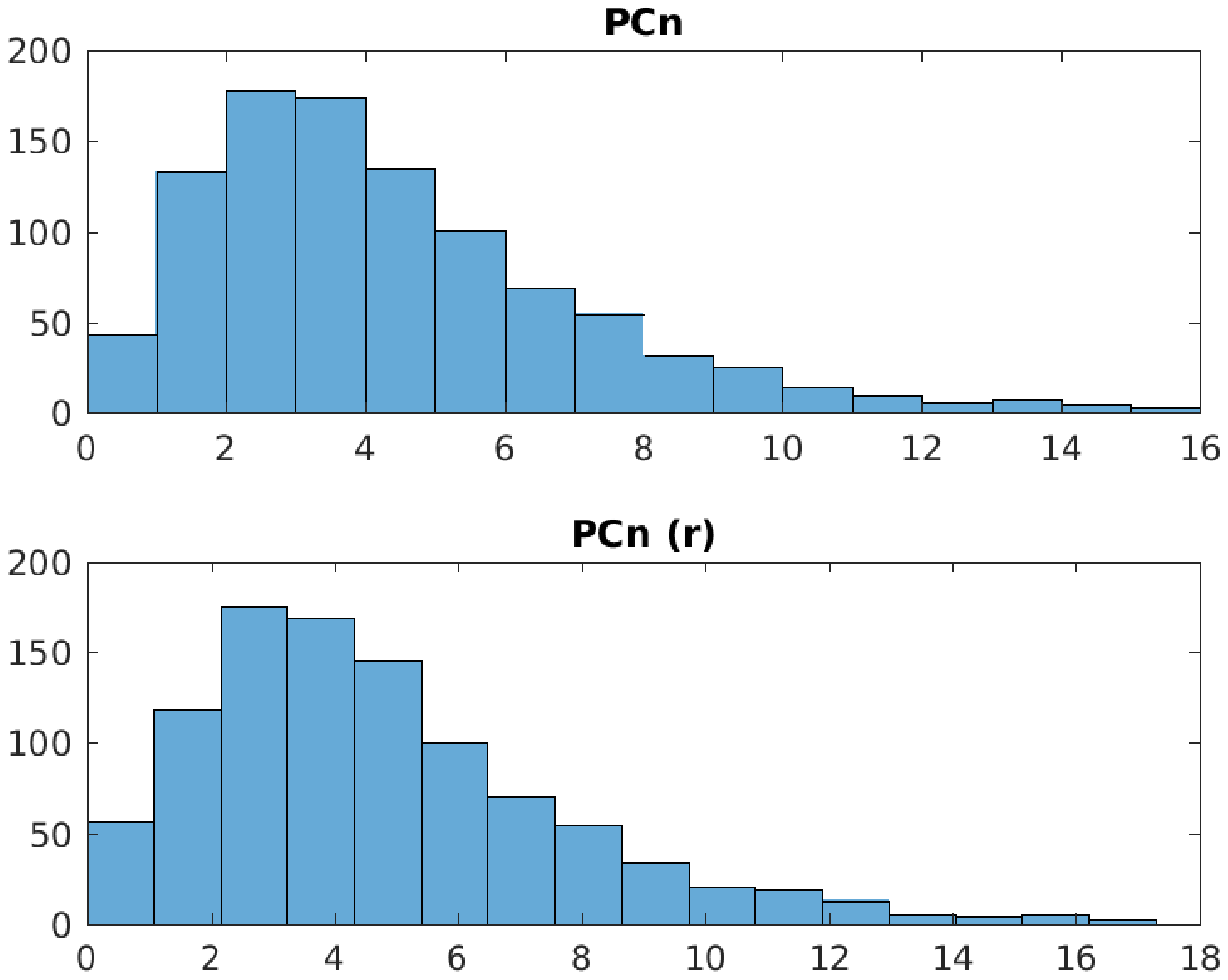}
\end{minipage}
\begin{minipage}{0.325\hsize}
	\includegraphics[width=1.1\textwidth]{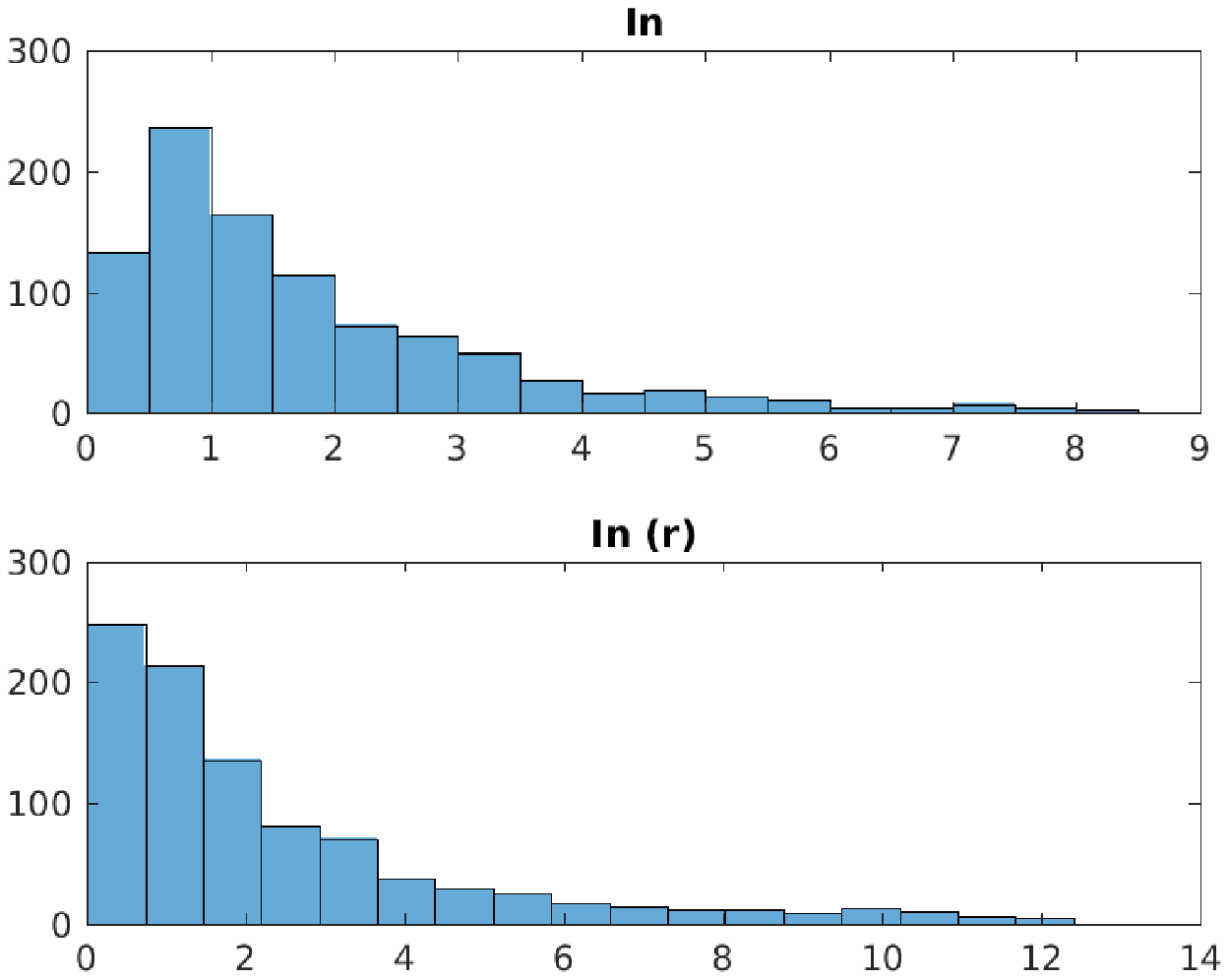}
\end{minipage}
\begin{minipage}{0.325\hsize}
	\includegraphics[width=1.1\textwidth]{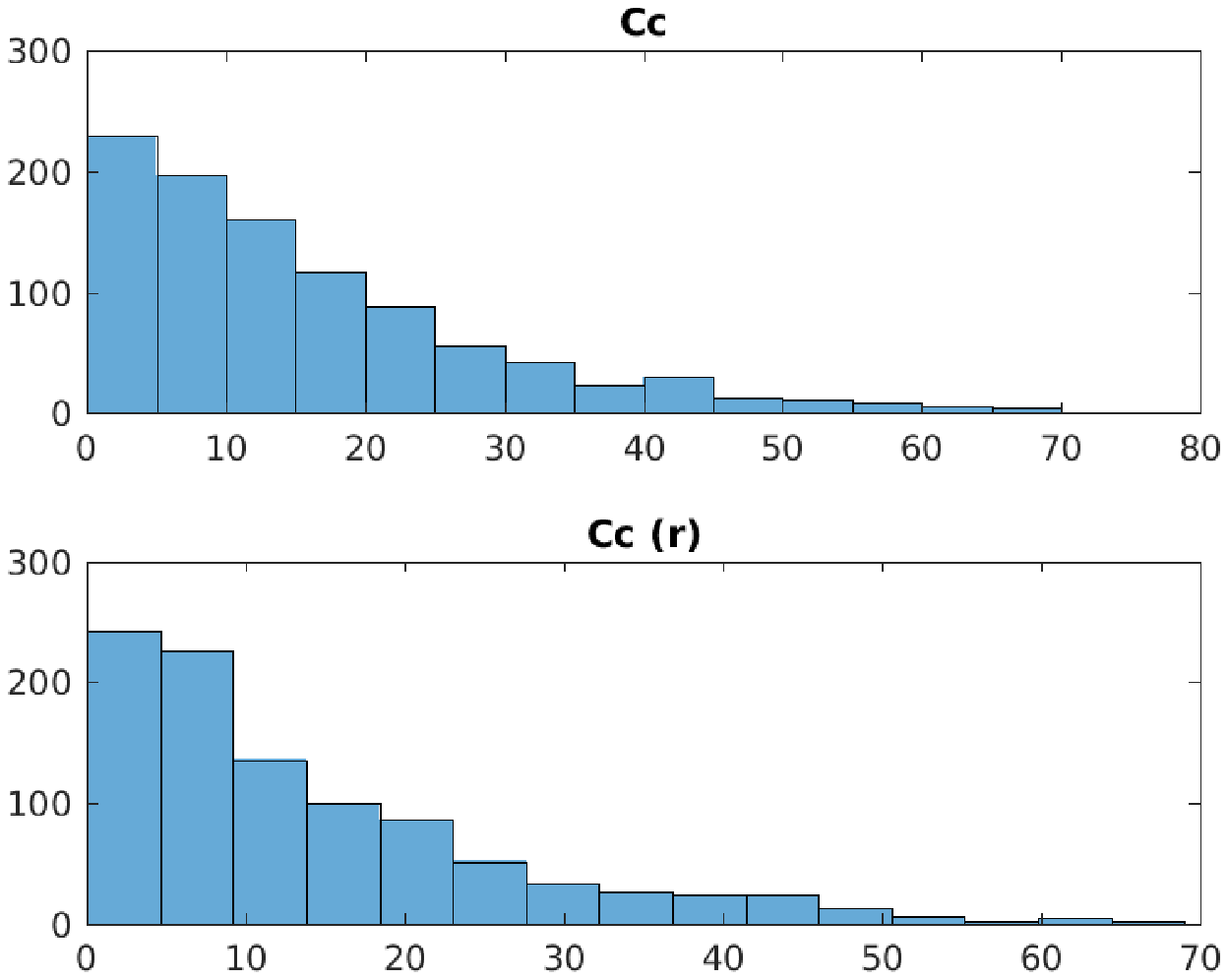}
\end{minipage}
\caption{\label{fig:CIRCs-hist} Histograms for 3 representative species with non-Gaussian behaviour showing tight agreements, for the stochastic circadian model.}
\end{figure}

\subsection{A stochastic growth factor receptor model}
In this section we focus on a stochastic growth factor receptor model which is derived  from the deterministic model presented in Section \ref{ex:EGFR}, by using Kurtz's classical scaling (\cite{KURTZ1978223}) (for more details see Section \ref{ex:stoc_CIRC}).

We compute the pFIM for this stochastic network (see \eqref{eq:pFIM_SRN}) by sampling $M{=}10^3$ trajectories of $X(t)$, for $t\in [0,T]$, $T{=}50$. Using that sample, we construct a bootstrap estimation of the mean of the time average and the corresponding 95\% confidence interval (together for comparison with the time average of $z$), as shown in Figure \ref{fig:EGFRs}. 
We observe that, for many species, the time average of the deterministic model is a good approximation of the mean time average. In this respect, we can consider this stochastic model close to a mean-field regime.

\begin{figure}[h!]
\centering
\begin{tiny}\begin{tabular}{|l|c|c|c|c|}
\hline
&\textbf{Mean $\mathcal{A}(X)$}&\textbf{CI left end}&\textbf{CI right end}&$\mathcal{A}(z)$\\\hline
\textbf{EGFR}&2.0997&2.1353&2.1715&2.2993\\\hline
\textbf{EGF}&60.0997&60.1353&60.1715&60.2993\\\hline
\textbf{EGF-EGFR}&3.7125&3.7529&3.7945&3.9655\\\hline
\textbf{EGF-EGFR-2}&1.1980&1.2183&1.2375&1.0995\\\hline
\textbf{EGF-EGFR-2-P}&0.1318&0.1344&0.1370&0.1200\\\hline
\textbf{PLCg}&9.6485&9.6979&9.7485&9.8726\\\hline
\textbf{EGF-EGFR-2-PLCg}&0.0602&0.0627&0.0652&0.0563\\\hline
\textbf{EGF-EGFR-2-PLCg-P}&0.1770&0.1852&0.1944&0.1718\\\hline
\textbf{PLCg-P}&0.0786&0.0826&0.0864&0.0727\\\hline
\textbf{Grb2}&8.4036&8.4293&8.4555&8.5213\\\hline
\textbf{EGF-EGFR-2-Grb2}&0.0293&0.0368&0.0443&0.0290\\\hline
\textbf{SOS}&2.7198&2.7376&2.7553&2.7840\\\hline
\textbf{EGF-EGFR-2-Grb2-SOS}&0.0057&0.0082&0.0111&0.0060\\\hline
\textbf{Grb2-SOS}&0.1124&0.1245&0.1381&0.1070\\\hline
\textbf{Shc}&12.3810&12.4518&12.5232&12.7312\\\hline
\textbf{EGF-EGFR-2-Shc}&0.0236&0.0244&0.0251&0.0229\\\hline
\textbf{EGF-EGFR-Shc-P}&0.3506&0.3619&0.3734&0.3390\\\hline
\textbf{Shc-P}&1.7070&1.7607&1.8174&1.5703\\\hline
\textbf{EGF-EGFR-2-Shc-Grb2}&0.0162&0.0185&0.0207&0.0191\\\hline
\textbf{Shc-Grb2}&0.2400&0.2530&0.2667&0.2147\\\hline
\textbf{EGF-EGFR-2-Shc-Grb2-SOS}&0.0036&0.0053&0.0073&0.0041\\\hline
\textbf{Shc-Grb2-SOS}&0.1147&0.1244&0.1342&0.0989\\\hline
\textbf{PLCgP-I}&0.9308&0.9715&1.0134&0.8265\\\hline
\end{tabular}
\end{tiny}
\caption{\label{fig:EGFRs} Table: Mean time average of each species with the corresponding 95\% empirical confidence interval and the time average of $z$ (the time average denoted by $\mathcal{A}$), stochastic growth factor receptor model. We can observe that, for many  species, the time average of the deterministic model is a good approximation of the mean time average. We notice that only 23 species are shown in the table because in the original model a dummy species is used to model the creation of particles.}
\end{figure}

We now reduce this stochastic model by using our method. We show in Table \ref{tab:EGFRs} the results. For an equivalent number of parameters, we have consistent pFIM and validation distances in comparison with the deterministic model. We stress again this robustness feature of our reduction method. For example, we compare the reduced model obtained by applying our method to the deterministic model versus the reduced model of the stochastic circadian model. Compare for instance the model with 25 parameters and 24 reaction channels that accumulates $97.244\%$ of the deterministic pFIM (Table \ref{tab:EGFR}) versus the model with 24 parameters and 22 reaction channels that accumulates $99.5\%$ of the stochastic pFIM (Table \ref{tab:EGFRs}). In fact, we obtained a slightly more parsimonious model (less parameters and reaction channels) with similar validation distances. We notice that the validation distances are mean-field based (see \eqref{eq:pdist} and \eqref{eq:sdist}).

\begin{table}[h!]
\centering
\begin{footnotesize}\begin{tabular}{|c|c|c|c|c|c|c|}
\hline \rule{0pt}{2.6ex}
\textbf{pFIM \% \eqref{eq:pFIM}}&\textbf{$\bar{J}$}&\textbf{$\bar{K}$}&\textbf{$\bar{d}$}&\textbf{Loss \eqref{eq:losse}}&\textbf{path-dist \eqref{eq:pdist}}&\textbf{SS-dist \eqref{eq:sdist}}\\\hline
98.97&20&21&18&0.0033021&2.443&0.477\\\hline
99.50&22&24&19&0.000889939&0.627&0.191\\\hline
99.73&24&26&20&0.000326358&0.359&0.106\\\hline
\end{tabular}
\end{footnotesize}

\caption{\label{tab:EGFRs} Different reduced stochastic models as per total pFIM information, stochastic growth factor receptor model. Since the number of species in every model is the same, the validation distances of the last two columns are applied to the same set of species.}
\end{table}

In Figure \ref{fig:EGFRs-box} we show a comparison between the full stochastic model and the reduced model at 99.5\% of total information by using a box plot. On each x-axis we show a box for the species in the full model and next a box for the same species in the reduced model. We can observe a tight agreement on the vast majority of the species. In Figure \ref{fig:EGFRs-hist} we show histograms for 3 representative species with non-Gaussian behaviour showing tight agreements.

\begin{figure}[h!]
\begin{minipage}{0.325\hsize}
	\includegraphics[width=1.1\textwidth]{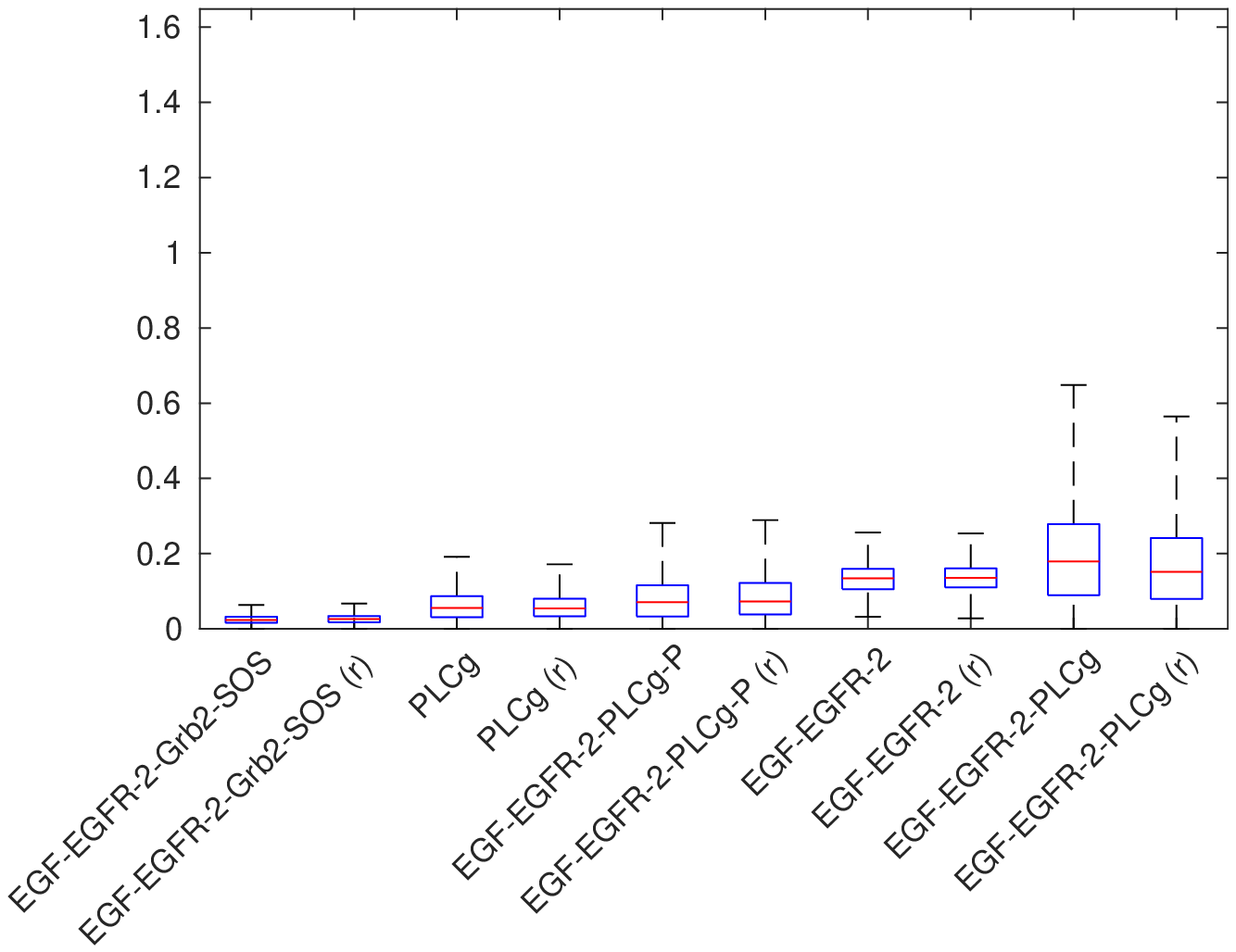}
\end{minipage}
\begin{minipage}{0.325\hsize}
	\includegraphics[width=1.1\textwidth]{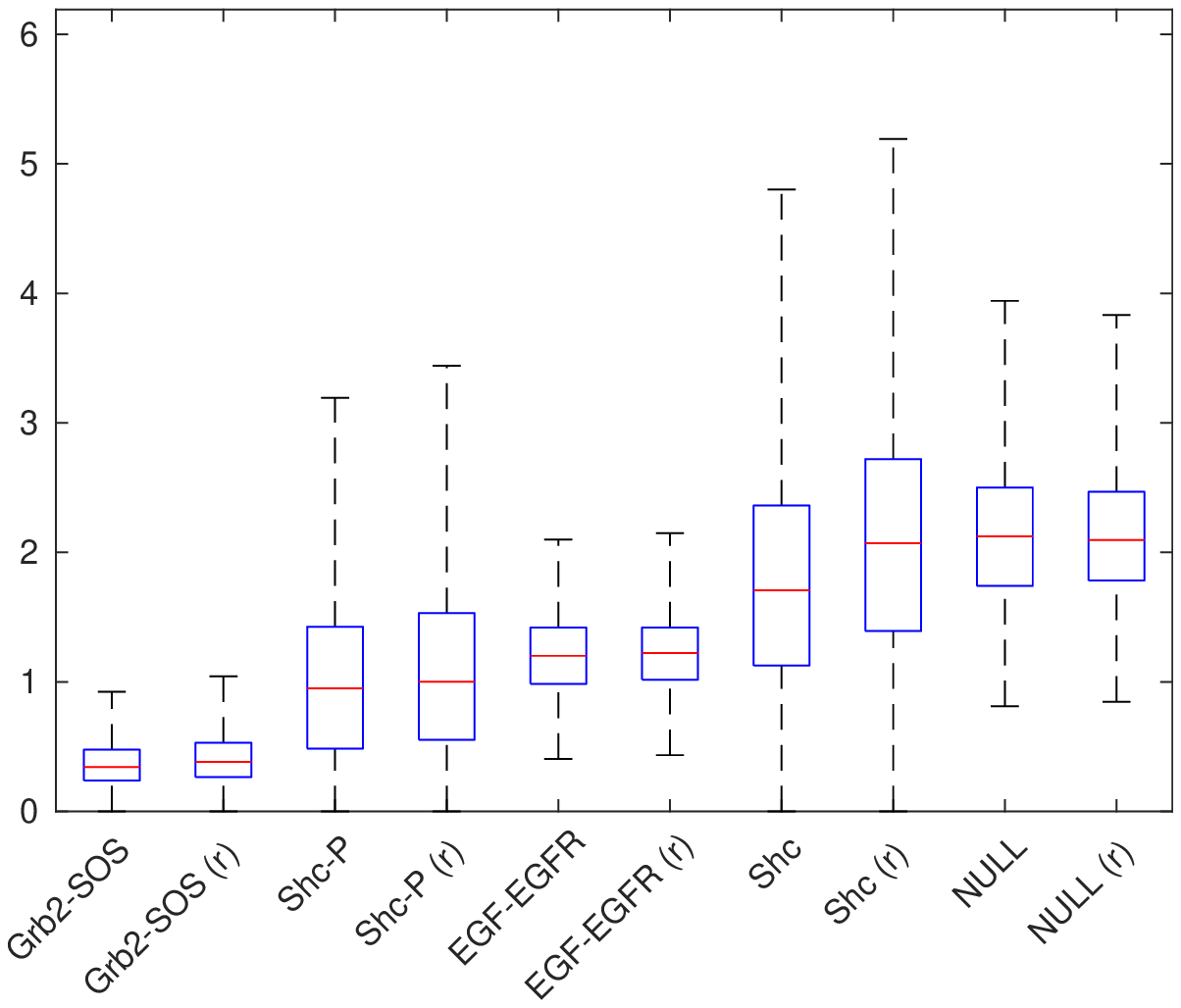}
\end{minipage}
\begin{minipage}{0.325\hsize}
	\includegraphics[width=1.1\textwidth]{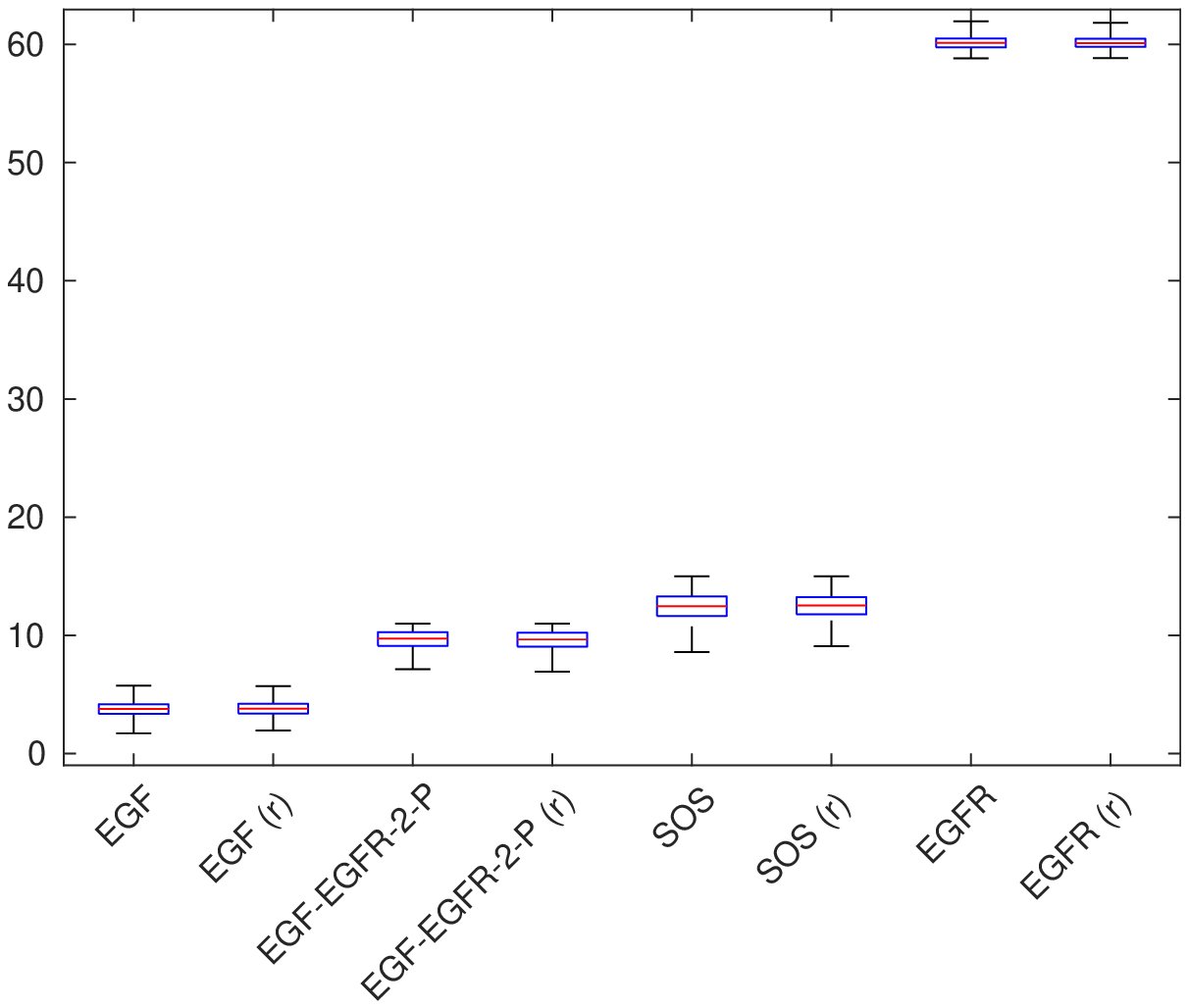}
\end{minipage}
\caption{\label{fig:EGFRs-box} Box plot comparison of time average of resolved species, stochastic growth factor receptor model. On each box, the central mark indicates the median, and the bottom and top edges of the box indicate the 25th and 75th percentiles, respectively of the time average. The whiskers extend to the most extreme data points not considering outliers. Resolved species name with suffix ``(r)'' correspond to the reduced model. }
\end{figure}

\begin{figure}[h!]
\begin{minipage}{0.325\hsize}
	\includegraphics[width=1.1\textwidth]{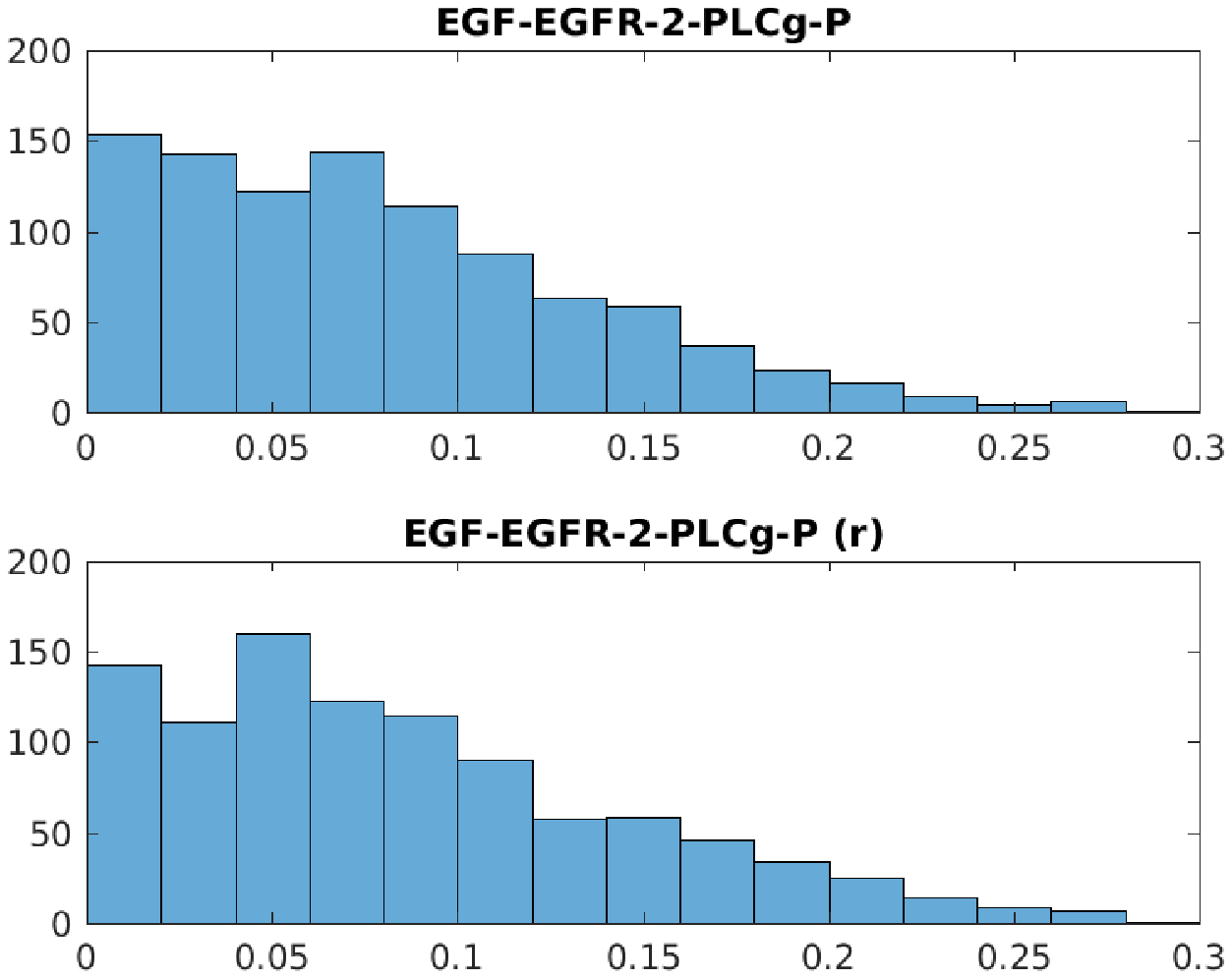}
\end{minipage}
\begin{minipage}{0.325\hsize}
	\includegraphics[width=1.1\textwidth]{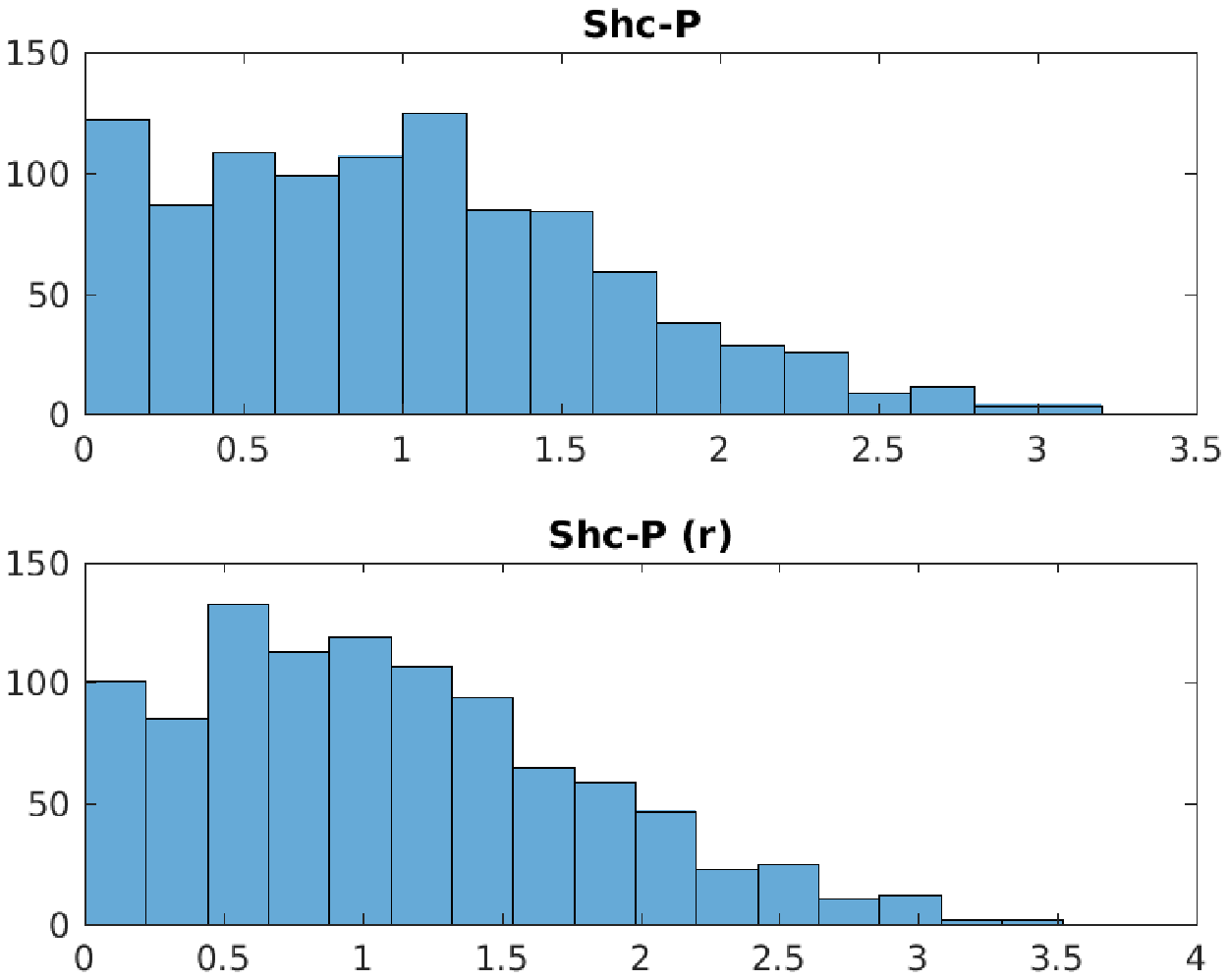}
\end{minipage}
\begin{minipage}{0.325\hsize}
	\includegraphics[width=1.1\textwidth]{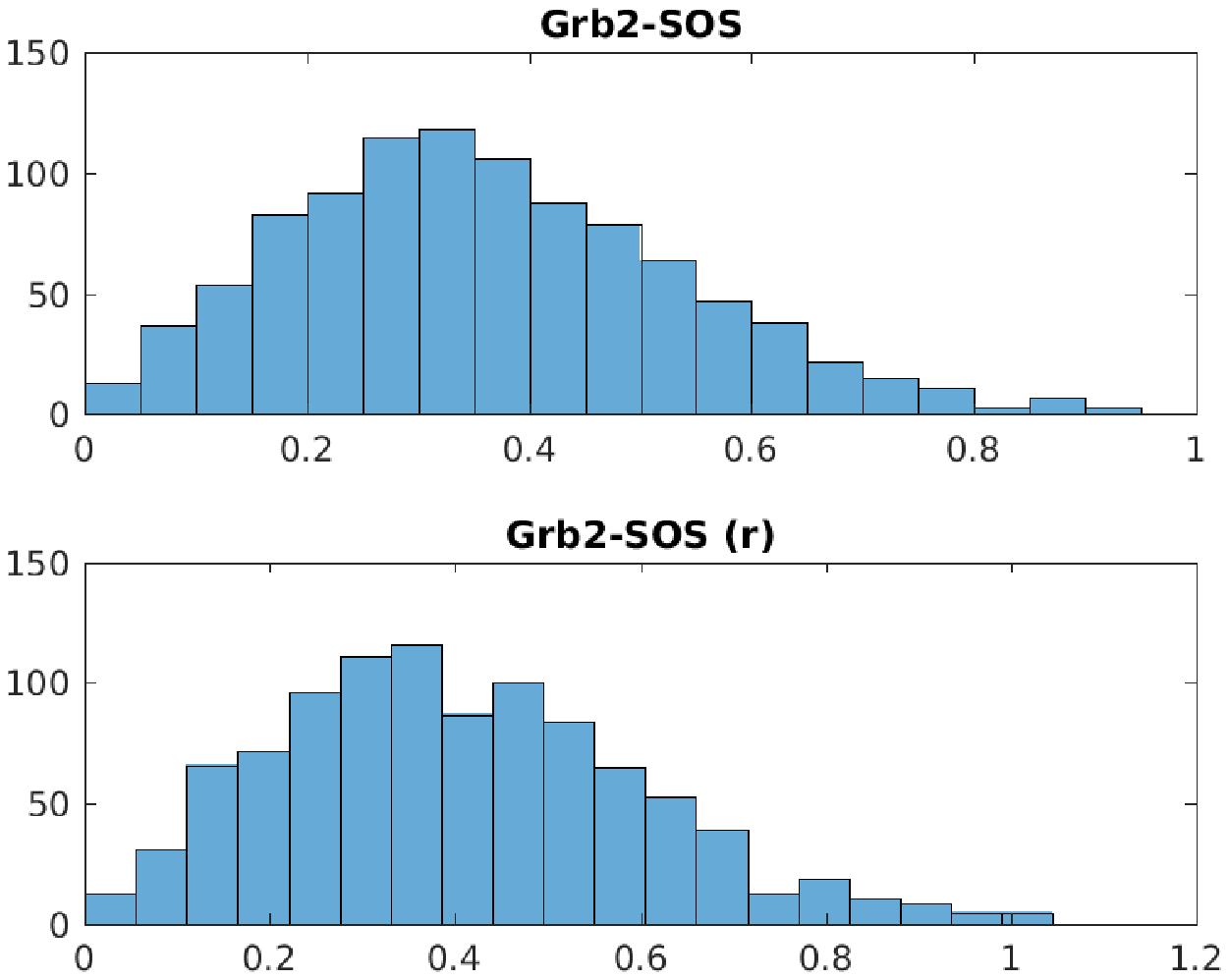}
\end{minipage}
\caption{\label{fig:EGFRs-hist} Histograms for 3 representative resolved species with non-Gaussian behaviour showing tight agreements, stochastic growth factor receptor model.}
\end{figure}

\section{Connections with related reduction methods}
\label{sec:comparison}
In this section we briefly discuss related state of the art on model reduction for biochemical reaction networks.

\emph{Sensitivity-based reduction methods. }
The goal of sensitivity analysis is to determine how certain quantities of interest of the system vary under perturbations to  model parameters and/or state variables. To reduce a given system by using this technique, the most common approach is to eliminate the least sensitive parameters and/or state variables. 
One of the advantages of sensitivity-based methods is the meaning preservation of the state variables and reaction channels.

\noindent
Local sensitivity-based reduction methods \cite{liu2004sensitivity,DEGENRING2004729, apri2012complexity, turanyi1990sensitivity, tomlin1995reduced, liu2004sensitivity} usually require to solve the system \eqref{eq:adjoint}, and therefore are not suitable for high-dimensional networks (see Remark \ref{rem:adjoint}) not only due to the high dimension (number of parameters times number of state variables) but also the stiffness of the system to solve. Our method only requires to compute the pathwise FIM which is of the order of the number of parameters. A typical drawback of sensitivity-based model reduction approaches is that the elimination of low -sensitivity parameters may lead to unsatisfactory results. An example of this issue is presented in \cite{chassagnole2002dynamic}, for which the sensitivities of some of the reaction channels are close to zero, however the removal of these reactions would result in the shutdown of the whole reaction network. We applied our method to this model obtaining a substantial reduction without compromising the dynamics of any of the species present in the original model. 
If a matrix of sensitivity indexes is feasible to compute, principle component analysis (PCA) is a commonly used method to rank reaction channels and then determine which ones can be eliminated \cite{turanyi1989reaction,degenring2004sensitivity,smets2002sensitivity}. As we previusly mentioned, our method is amenable to combine with PCA by applying this analysis to extract further information from the pFIM, especially taking into account its block diagonal structure. We do not pursue this direction any further here.

\noindent
Finally, the use of global sensitivity analysis for model reduction of biochemical reaction networks is still an open problem, due to the extremely high computational work requirements of these methods (see \cite{maurya2005reduced,jayachandran2014optimal}). We anticipate that the main blocks of our method (i.e. model selection by means of the pFIM and its stoichiometry matrix and the use of a loss function in the macroscopic space for data training) can be of great help in applying global sensitivity analysis for model reduction of high-dimensional systems.

\emph{Optimization approaches. }
A technique that is usually related with sensitivity-based reduction methods is the optimization approach. The aim is to reduce a system by testing candidate reduced models while minimizing an error metric to choose  the best one.
The optimization techniques vary in construction process of the set of candidate reduced models. 

\noindent
In \cite{maurya2005reduced, maurya2009mixed}, the authors present a method that combines a model reduction technique together with a parameter estimation algorithm. In this work, in order to reduce the number of parameters, the least influential reaction rates are just set to zero. In this optimisation problem, the authors use a genetic algorithm to
simultaneously identify parameter values and further eliminate unimportant reactions. 
A similar approach is used in \cite{hangos2013model}, for the case of  polynomial or rational propensity functions, based on a integer quadratic programming optimization. This approach also contains rate coefficients estimation in the reduced model to minimize the defined model error. This approach is demonstrated by the authors by reducing a model of the Arabidopsis thaliana circadian clock involving 7 state variables and 27 reactions, from \cite{locke2005extension}. Even though we do not consider this model as high-dimensional, we applied our reduction method and obtained the same reduced model as in \cite{locke2005extension}. 

\noindent
In \cite{anderson2011model,prescott2012guaranteed}, the authors proposed a 
method to calculate the error associated with a model reduction
algorithm. This approach is based on the error between observables of the original and the reduce system.  A worst-case error type of bound between the original and the reduced system in form of sum of squares are used for developing an optimization-based method for model reduction.

\emph{Timescale-based reduction methods. }
Alternatively, model reduction can be carried out by by exploiting timescale differences that are often present in biochemical systems
\cite{III201501}. Timescale analysis methods are one of the most widely used approaches for model reduction in the literature. These methods aim to partition the reaction network into different timescales by exploiting the several orders of magnitude difference that usually exist between reaction channel rates. This timescale difference allows to reduce a given model as certain species or reaction channels can be assumed to be constant with respect to the timescale of interest.

\noindent
A group of timescale exploitation methods that is close to our method is based on identifying species or reactions which can be classified to be on a fast dynamics regime in comparison with the other ones, and therefore partitioning the reaction network into fast and slow components. Once such a partition has been found, the reduction of the system is achieved by means of the application of singular perturbation techniques. These methods are based on Tikhonov's singular perturbation theory for the reduction of first order ODE's going back at least to \cite{tikhonov1952systems}. 
The aforementioned partition could be done in terms of species or reaction channels.
Among the species-based partitioning methods the quasi-steady-state approximation (QSSA) is well-known for its application to the reduction of the Michaelis–Menten equation \cite{briggs1925note}. This method is limited to models in which species exhibit a clear separation in timescales. For low-dimensional systems, it is feasible to search for such a partition, by the use of additional intuitive information about the system. Unfortunately, for high-dimensional systems these methods are usually prohibitive due to the combinatorial nature of different
model representations \cite{rou01, kooi02, radulescu2008robust, petrov2007reduction,Schneider2000}. Other works aim to provide algorithmic procedures for determining which species can be considered fast (see for instance  \cite{choi2008new,west2015method}).
The reaction-based partitioning methods assume that certain reactions occur fast enough such that it can be assumed that they reach an equilibrium immediately, so these methods are referred as rapid equilibrium approximation. This is de idea behind the Michaelis-Menten original approximation \cite{menten1913kinetik}. Related works include \cite{gerdtzen2004non, prescott2012guaranteed, noel2014tropicalization}.
The main difficulty associated with the previous timescale partitioning method is to find a formulation such that the timescales differences between species or reaction channels are clearly exposed. 

\noindent
A second group of timescale methods aims to obtain a transformation of the state variables to obtain a reduced model where timescale separations are clear.
Such approaches often lead to more accurate and substantial model reductions than the previous methods but transformations often difficult the biochemical interpretation of the reduced model.
An example of these methods is the eigenbasis transformation of the state variables, as in the intrinsic low-dimensional manifold method (ILDM), originally developed in \cite{maas1992simplifying}. A brief review for the biochemical reaction network case is given in \cite{vallabhajosyula2005conservation}. Other applications of this method is given in \cite{zobeley2005new, surovtsova2009accessible}.
Another transformation method that is widely used is the computational singular perturbation (CSP), originally published in \cite{lam1985singular} and further developed in \cite{lam1991conventional, lam1994csp}. A rigorous analysis of this method and its comparison with the ILDM is given in \cite{zagaris2004analysis}. This aims to transform the set of reaction channels into a different basis to clearly enhance timescale differences between the transformed reaction channels. The fast transformed reaction channels are assumed to equilibrate instantaneously, and then their dynamical contribution is neglected  in the reduced model. Some works in this direction include \cite{surovtsova2012simplification,kourdis2010physical}.

\emph{Lumping-based reduction methods. }
Lumping is another wide class of model reduction methods, usually applied to linear systems. How a proper lumping can be formulated for a nonlinear system
is still an open problem. Methods presented in the literature for non-linear systems usually are based on trial and error approaches which may be computationally prohibitive for high-dimensional systems. This class of methods originated in the dynamical systems literature \cite{wei1969lumping, kuo1969lumping}. The idea is to remove sets of state variables and replace them with
 new lumped variables that represent mappings from the original ones. Applications of lumping to biochemical reaction networks are \cite{dano2006reduction,dokoumetzidis2009proper,koschorreck2007reduced,sunnaaker2011method}

\section{Conclusions}
\label{sec:conc}
In this work we presented an efficient and principled model reduction method for high-dimensional deterministic and stochastic reaction networks. 
The goal of model reduction is to construct a simpler model in terms of a reduced set of state variables, parameters and reaction channels. In general, there is no universal model reduction method  which can be considered superior for every reaction network, especially in the high-dimensional case. The appropriateness of each method is entangled to the nature of the model that is intended to reduce. Our method is particularly suited for high-dimensional biochemical systems, applicable to any smooth propensity function and no equilibrium assumptions are required. 
Despite our method is designed for high-dimensional models, it is well suited for many low-dimensional examples found in the literature although alternative reduction methods may be better suited for low dimensional models, such as transforming methods or lumping approaches (see Section \ref{sec:comparison} for brief discussions of these methods). Our method is also well suited for pure jump models, in Langevin regimes or deterministic ones.

Our method is based on pathwise information metrics to screen-out insensitive parameters via the analysis of the pathwise Fisher Information Matrix. 
By means of a simple loss function for time-series data based on path-space information theory, candidate reduced models can be fitted to full model time series data and therefore the resulting reduced model dynamics conform an accurate approximation to the full model dynamics on a given time interval. 
The main features of this approach are summarized as follows:
i) scalable method targeted to high-dimensional deterministic and stochastic biochemical systems; ii) applicable to stiff and non-linear models; iii) independent of particular quantities of interest; iv) given a user-defined information threshold, the information loss of the reduced model is controlled; v) applicable to any smooth propensity function form and no information is required concerning which reactions are in partial equilibrium nor which species are assumed to be in steady state; vi) no biochemical knowledge is required to reduced the model; vii) intuitive and automatizable method; viii) it is \emph{not} a data-hungry method.

As previously mentioned, our method has a clear advantage with respect to sensitivity-based and optimization-based methods that require to solve classical sensitivity systems (e.g. \ref{eq:adjoint}), since our method only requires to compute the pathwise FIM \eqref{eq:pFIM}. We recall that the computational work in the first case is of the order of the number of state variables \emph{times} parameters while in the second is of the order of the number of parameters only (see Remark \ref{rem:adjoint}). 
One of the features of our method is that it preserves the biochemical meaning of the variables and reaction channels. However, when this requirement is not relevant for a particular application, transformation-based reduction methods like lumping or timescale transforming methods may be a superior alternative depending on the dimension of the system (see Section \ref{sec:comparison}).

Even though our method uses the stoichiometry of the network to choose sensitive species, further information may be extracted from it as well. For example, graph-theoretic tools may allow to dissect the stoichiometry graph into functional modules and determine different types of interactions between these modules (see for instance \cite{hartwell1999molecular,saez2005dissecting,saez2004modular}). Biochemical information can also be included in this analysis, as for example retroactivity, i.e., which sub-modules influences (or not) each other (see for instance \cite{conzelmann2004reduction,wolf2003motifs}). In \cite{craciun2008identifiability}, the authors use graph-theoretic tools to characterize conditions for two different reaction networks to have the same dynamics, in the case of mass action kinetics dynamics.

We finally notice that in our method, the reduced model depends on the time interval $[0,T]$ considered. However, our methodology can be also applied to long-time horizons by working with the relative entropy rate
$$
\rer{P}{Q} = \lim_{T \rightarrow \infty} \frac{1}{T} \re{P_{[0,T]}}{Q_{[0,T]}}\, ,
$$
where $P$ and $Q$ denote the corresponding stationary processes. We refer to \cite{DKPP16} and references therein.
 
\section*{Acknowledgments}

\noindent
The research of M.K. was partially supported by the Defense Advanced Research Projects Agency (DARPA) EQUiPS program under the grant W911NF1520122. The research of P.V. was supported by the Defense Advanced Research Projects Agency (DARPA) EQUiPS program under the grant W911NF1520122.

%
%

\clearpage
\section*{Appendix: Relative entropy and pFIM decomposition}
\label{sec:app}

\def\PATHS{Q_{0:T}^\theta}
\def\IPATHS#1{Q_{#1}^\theta}
\def\PATHSAPP{Q_{0:T}^{\theta+\epsilon}}
\def\IPATHSAPP#1{Q_{#1}^{\theta+\epsilon}}

\begin{thm}
\label{thm:pRE}
The pathwise relative entropy for a discrete-time Markov chain
can be decomposed as
\begin{equation}
\re{P_{0:T}}{Q^\theta_{0:T}}= \re{\INIT}{\INITAPP} + \sum_{i=1}^T \re{P_i}{Q^\theta_i} \ ,
\label{Rel:entr:DTMC}
\end{equation}
where the quantity 
\begin{equation}
\re{P_i}{Q^\theta_i}
= \mathbb E_{\nu_{i-1}}\Big[\int_\latt p(x,x')\log \frac{p(x,x')}{q^{\theta}(x,x')}dx'\Big] \, .
\label{inst:Rel:entr:DTMC}
\end{equation}
can be interpreted as the instantaneous relative entropy.
\end{thm}

\begin{proof}
The proof of this theorem can be found in \cite[Ch. 2]{Cover:91}, 
but for the sake of completeness we present it here. The Radon-Nikodym derivative
of $P_{0:T}$ w.r.t. $Q^\theta_{0:T}$ takes the form
\begin{equation*}
\frac{dP_{0:T}}{dQ^\theta_{0:T}}\big( (x_i)_{i=0}^T \big) =
\frac{\INIT(x_0)\prod_{i=0}^{T-1}p(x_i,x_{i+1})}{\INITAPP(x_0)\prod_{i=0}^{T-1}q^{\theta}(x_i,x_{i+1})}\, ,
\end{equation*}
which is well-defined since the transition probabilities are always positive. Then,
\begin{equation*}
\begin{aligned}
&\re{P_{0:T}}{Q^\theta_{0:T}} \\
&= \int_\latt\cdots\int_\latt \INIT(x_0)\prod_{j=1}^T p(x_{j-1},x_j)
\log \frac{\nu(x_0)\prod_{i=1}^T p(x_{i-1},x_i)}
{\INITAPP(x_0)\prod_{i=1}^T q^{\theta}(x_{i-1},x_i)} dx_0\ldots dx_T \\
&= \int_\latt\cdots\int_\latt \INIT(x_0)\prod_{j=1}^T p(x_{j-1},x_j) \log \frac{\INIT(x_0)}{\INITAPP(x_0)} dx_0\ldots dx_T \\
& \quad + \sum_{i=1}^T \int_\latt\cdots\int_\latt \INIT(x_0)\prod_{j=1}^T p(x_{j-1},x_j) 
\log \frac{p(x_{i-1},X_i)}{q^{\theta}(x_{i-1},x_i)}  dx_0\ldots dx_T \\
%
%
&= \re{\INIT}{\INITAPP} + \sum_{i=1}^T \re{P_i}{Q^\theta_i} \ ,
\end{aligned}
\end{equation*}
where $\re{\INIT}{\INITAPP}=\mathbb E_{\INIT}\Big[\log \frac{\INIT(x)}{\INITAPP(x)}\Big]$ is the relative entropy of the initial distributions, while 
 the instantaneous relative entropy is
\begin{equation*}
\re{P_i}{Q^\theta_i} =
\mathbb E_{\nu_{i-1}}\Big[\int_\latt p(x,x')\log \frac{p(x,x')}{q^{\theta}(x,x')}dx'\Big] \ .
\end{equation*}
\end{proof}

\begin{thm}
\label{thm:pFIM}
Under smoothness assumption on the transition probability function
w.r.t. the parameter vector $\theta$, the pFIM can be also decomposed as
\begin{equation}
\FISHER{\PATHS} = \FISHER{\nu^\theta} + \sum_{i=1}^T  \FISHERR{\IPATHS{i}} \ ,
\label{path:FIM:DTMC}
\end{equation}
where $ \FISHER{\nu^\theta}$ is the FIM of the initial distribution and the instantaneous FIM is given by
\begin{equation}
\FISHERR{\IPATHS{i}} =
\expt{\int_\latt p^\theta(x,x')\nabla_\theta \log p^\theta(x,x')\nabla_\theta \log p^\theta(x,x')^{tr} dx'}{\nu_{i-1}} \, .
\label{inst:FIM:DTMC}
\end{equation}
\end{thm}
\begin{proof}
This proof is similar to the proof presented in \cite{PK2013,PLOS:AKP}. We recap it here with minor but
necessary adaptations. 

Let $\Delta p(x,x') := p^{\theta+\epsilon}(x,x') - p^\theta(x,x')$. Thus, we have that
\begin{equation*}
\begin{aligned}
&\re{\IPATHS{i}}{\IPATHSAPP{i}}
= - \int_\latt\int_\latt \nu_{i-1}^\theta(x) p^\theta(x,x') \log \left(1+ \frac{\Delta p(x,x')}{p^\theta(x,x')}\right) dx dx' \\
&= - \int_\latt\int_\latt \left[\nu_{i-1}^\theta(x)\Delta p(x,x')
-  \frac{1}{2}\nu_i^\theta(x)\frac{\Delta p(x,x')^2}{p^\theta(x,x')}
+ O(|\Delta p(x,x')|^3) \right] dx dx' \, .
\end{aligned}
\end{equation*}
Moreover, for all $x\in \latt$, it holds that
\begin{equation*}
\int_\latt \Delta p(x,x')dx' = \int_\latt p^{\theta+\epsilon}(x,x')dx'
- \int_\latt p^{\theta}(x,x')dx' = 0.
\end{equation*}
By using the smoothness assumption and Taylor-expanding $\Delta p$ we obtain
\begin{equation*}
\delta p(x,x') = \epsilon^{tr} \nabla_\theta p^\theta(x,x') + O(|\epsilon|^2)\ .
\end{equation*}
Finally, we have that 
\begin{equation*}
\begin{aligned}
&\re{\IPATHS{i}}{\IPATHSAPP{i}} 
= \frac{1}{2} \int_\latt\int_\latt  \nu_{i-1}^\theta(x) \frac{(\epsilon^{tr} \nabla_\theta p^\theta(x,x'))^2}{p^\theta(x,x')} dx dx' + O(|\epsilon|^3) \\
&= \frac{1}{2}\epsilon^{tr}\Big( \int_\latt\int_\latt \nu_{i-1}^\theta(x) p^\theta(x,x) \nabla_\theta \log p^\theta(x,x')
 \nabla_\theta \log p^\theta(x,x')^{tr} dx dx' \Big) \epsilon + O(|\epsilon|^3) \\
&= \frac{1}{2} \epsilon^{tr} \FISHERR{\IPATHS{i}} \epsilon + O(|\epsilon|^3)
\end{aligned}
\end{equation*}
where
\begin{equation*}
\FISHERR{\IPATHS{i}} =  \expt{\int_\latt p^\theta(x,x') \nabla_\theta 
\log p^\theta(x,x') \nabla_\theta \log p^\theta(x,x')^{tr} dx'}{\nu_{i-1}^\theta}
\end{equation*}
is the instantaneous FIM associated to the instantaneous relative entropy.
Consequently, the pFIM $\FISHER{\PATHS}$, i.e., the Hessian of the
pathwise relative entropy at point $\theta$, is given by
\begin{equation*}
\FISHER{\PATHS} = \FISHER{\nu^\theta} + \sum_{i=1}^T  \FISHERR{\IPATHS{i}} \ ,
\end{equation*}
where $\FISHER{\nu^\theta} = \expt{\nabla_\theta \log\nu^\theta(x)\nabla_\theta \log\nu^\theta(x)^{tr}}{\nu^\theta}$
is the FIM of the initial distribution.
\end{proof}

\clearpage
\section*{Appendix: Proof of Theorem \ref{thm:argmin}}
\label{sec:thm_min_proof}
The following Lemma is an instrumental result for proving this theorem.
\begin{lemm}
\label{lemm:argmin}
For $\Dt_i > 0$ and $i \in \{1,2,...,T\}$ we have 
$$
\nabla_\theta \rer{P_i}{Q_i^\theta} = \nabla_\theta \left[R_i(\theta) + \Dt_i M_i(\theta)\right] \, ,
$$
where $R_i(\theta)$ and $M_i(\theta)$ are defined as in Theorem \ref{thm:argmin}.
\end{lemm}
\begin{proof}
In order to compute $\rer{P_i}{Q_i^\theta}$ recall that the transition probability density of the Euler discretization of the microscopic CLE for $x' \in \mathbb{R}^d$ is given by
\begin{equation}
\label{eq:cle_trans}
p(x,x') = \frac{1}{Z_\Dt(x)} \exp \left\{ -\frac{1}{2\Dt}(x'-m_\Dt(x))^{tr} \Sigma^{-1}(x)(x'-m_\Dt(x)) \right\} \, ,
\end{equation}
where $m_\Dt(x) := x-b(x) \Dt$, $Z_\Dt(x) := \sqrt{(2 \pi \Dt)^d \det(\Sigma(x))}$ and $\Sigma(x) := \sigma(x) \sigma^{tr}(x)$.

Let $\CG: \mathbb{R}^d \rightarrow \mathbb{R}^{\bar d}$ be a linear  map as in \eqref{eq:var_map} such that,
$\Sigma(x)$ can be partitioned as
$$\left( \begin{array}{cc} \Sigma_{1,1}(x)& \Sigma_{1,2}(x) \\ \Sigma_{2,1}(x)& \Sigma_{2,2}(x) \end{array} \right) \, ,$$
where $\Sigma_{1,1}(x)= \CG \Sigma(x) \CG^{tr} \in \mathbb{R}^{\bar d}{\times} \mathbb{R}^{\bar d}$, $\Sigma_{2,2}(x)=\CG^\comp \Sigma(x) \CG^{\comp,tr} \in \mathbb{R}^{d-\bar d}{\times}\mathbb{R}^{d-\bar d}$ and $\Sigma_{2,1}(x)=\Sigma_{1,2}^{tr}(x)$.

Considering \eqref{eq:cle_trans} for $x' \in \mathbb{R}^d$ let $x' \mapsto (\CG x',\CG^\comp x') = (\bar x', \hat x')$, then the conditional distribution of $\hat x'$ given $\bar x '$ can be written as 
$$\hat x' | \bar x' \sim \mathcal{N}(\CG m_\Dt(x) + \Sigma_{2,1}(x) \Sigma_{1,1}^{-1}(x)(\bar x' {-} \CG^\perp m_{\Dt}(x))  , \,\,  \Sigma_{2,2}(x)-  \Sigma_{2,1}(x) \Sigma_{1,1}^{-1}(x) \Sigma_{1,2}(x)  ) \, .$$

This allows us to split the microscopic transition probability as 
\begin{equation}
\label{eq:micro_split}
p(x,x')=p^{(1)}(x,\bar x') p^{(2)}(x,\hat x'|\bar x') \, ,
\end{equation}
where 
$$
p^{(1)}(x,\bar x') = \frac{1}{Z^{(1)}_\Dt(x)} \exp \left\{ -\frac{1}{2\Dt}(\bar x'-\CG m_\Dt(x))^{tr} \Sigma_{1,1}^{-1}(x)(\bar x'-\CG m_\Dt(x)) \right\} \, ,
$$
with $Z^{(1)}_\Dt(x) := \sqrt{(2 \pi \Dt)^{\bar d} \det(\Sigma_{1,1}(x))}$, assuming the matrix $\Sigma_{1,1}$ is non-singular without losing generality.

We now describe the macroscopic CLE (corresponding to the reduced model), which is an approximation of the projected time series $(\CG x_i)_{i=0}^T$. Its Euler discretization is given by 
$$
\bar x_{k+1} = \bar x_k + \bar b(\bar x_k; \theta) \Dt + \bar \sigma(\bar x_k;\theta) \Delta \bar W_k \, ,
$$
where 
$$\bar b(\bar x;\theta)=\bar \nu \bar a(\bar x;\theta) \quad \text{and} \quad \bar \sigma(\bar x;\theta) = \bar \nu \sqrt{ \diag (\bar a(\bar x;\theta)}  \, ,$$
with $\Delta \bar W_k \sim \mathcal{N}(0,\Dt I)$ independent Gaussian increments.

This macroscopic (reduced) model is a parameterized approximation of the projected microscopic (full) model, with the following conditional density for $\bar x' \in \mathbb{R}^{\bar d}$ given $\bar x \in \mathbb{R}^{\bar d}$
\begin{equation*}
\label{eq:tran_macro}
p^\theta(\bar x,\bar x') = \frac{1}{\bar Z_\Dt(\bar x;\theta)} \exp \left\{ -\frac{1}{2\Dt}(\bar x'-\bar m_\Dt(\bar x;\theta))^{tr} \Sigma^{-1}(\bar x;\theta)(\bar x'-\bar m_\Dt(\bar x;\theta)) \right\} \, ,
\end{equation*}
where $\bar m_\Dt(\bar x; \theta) := \bar x- \bar b(\bar x;\theta) \Dt$, $\bar Z_\Dt(\bar x; \theta) := \sqrt{(2 \pi \Dt)^{\bar d} \det(\bar \Sigma(\bar x;\theta))}$ and $\bar \Sigma(\bar x;\theta) = \bar \sigma(\bar x;\theta)\bar \sigma^{tr}(\bar x;\theta)$, assuming $\bar \Sigma$ is non-singular. In the singular case, a Moore-Penrose pseudo-inverse must be used instead of $\bar \Sigma^{-1}$, and a pseudo determinant instead of the determinant.

In order to compare the microscopic density $p$ with the approximating macroscopic density $p^\theta$, we consider a reconstructed transition probability $q^\theta$ in the microscopic space, written in terms of $p^\theta$, as follows
\begin{equation}
\label{eq:macro_split}
q^\theta(x,x') = r(x'|\bar x') p^\theta(\bar x,\bar x') \, ,
\end{equation}
where $r(x'|\bar x')$ is any probability associated with the reconstruction, independent of $\theta$.
Now we can write 
\begin{equation}
\label{fact_log1}
\tag{*}
\log \frac{p(x,x')}{q^\theta(x,x')} = \log \frac{p^{(1)}(x,\bar x') p^{(2)}(x,\hat x'|\bar x')}{ p^\theta(\bar x,\bar x')r(x'|\bar x') } = \log \frac{p^{(1)}(x,\bar x')}{p^\theta(\bar x,\bar x')} + \log\frac{p^{(2)}(x,\hat x'|\bar x')}{r(x'|\bar x')} \, ,
\end{equation}
where the second term does not depend on $\theta$, so it can be ignored.

Furthermore we have,
\begin{align}
\tag{**}
\label{fact_log2}
\log \frac{p^{(1)}(x,\bar x')}{p^\theta(\bar x,\bar x')} &=-\frac{1}{2\Dt_i}(\bar x'-\CG m_{\Dt_i}x))^{tr} \Sigma_{1,1}^{-1}(x)(\bar x'-\CG m_{\Dt_i}(x)) \\ \nonumber & + \log \frac{\bar Z_{\Dt_i}(\bar x;\theta)}{Z^{(1)}_{\Dt_i}(x)}+ \frac{1}{2\Dt_i}(\bar x'-\bar m_{\Dt_i}(\bar x;\theta))^{tr} \Sigma^{-1}(\bar x;\theta)(\bar x'-\bar m_{\Dt_i}(\bar x;\theta)) \, .
\end{align}
The first term does not depend on $\theta$ so it can be also ignored.

Then using \eqref{eq:micro_split}, \eqref{eq:macro_split}, and taking into account \eqref{fact_log1} and \eqref{fact_log2}, we have
\begin{align*}
\nabla_\theta \rer{P_i}{Q_i^\theta} &= \nabla_\theta \expt{\int p(x,x') \log \frac{p(x,x')}{q^\theta(x,x')}dx'}{\nu_{i-1}} \\ &=\nabla_\theta \expt{\int \int p^{(1)}(x,\bar x') p^{(2)}(x,\hat x'|\bar x') \left(
 \log \frac{\bar Z_{\Dt_i}(\bar x;\theta)}{Z^{(1)}_{\Dt_i}(x)}+ K_{\Dt_i}(\bar x, \bar x';\theta)\right)  d\hat x' d\bar x'}{\nu_{i-1}} \\
 &= \nabla_\theta \expt{\int p^{(1)}(x,\bar x') \left(
 \log \frac{\bar Z_{\Dt_i}(\bar x;\theta)}{Z^{(1)}_{\Dt_i}(x)}+ K_{\Dt_i}(\bar x, \bar x';\theta) \right) d \bar x'}{\nu_{i-1}}
  \, ,
\end{align*}
where $K_{\Dt_i}(\bar x, \bar x';\theta) := \frac{1}{2\Dt_i}(\bar x'-\bar m_{\Dt_i}(\bar x;\theta))^{tr} \bar \Sigma^{-1}(\bar x;\theta)(\bar x'-\bar m_{\Dt_i}(\bar x;\theta))  $. The last equality is because $p^{(2)}$ is a density on $\hat x'$.

\begin{align*}
\nabla_\theta \rer{P_i}{Q_i^\theta} = \nabla_\theta & \expt{\int p^{(1)}(x,\bar x') \log \frac{\bar Z_{\Dt_i}(\bar x;\theta)}{Z^{(1)}_{\Dt_i}(x)} d \bar x' }{\nu_{i-1}} 
\\
  &+ \nabla_\theta \expt{\int p^{(1)}(x,\bar x') K_{\Dt_i}(\bar x, \bar x';\theta)  d \bar x'}{\nu_{i-1}}
\end{align*}

By replacing $Z^{(1)}_{\Dt_i}$, $\bar Z_{\Dt_i}$ and using that $\int p^{(1)}(x,\bar x') d\bar x' = 1$ the first term is equal to
\begin{align*}
\nabla_\theta \expt{\frac{1}{2} \log \det \left( \bar \Sigma(\bar x;\theta) \Sigma_{1,1}^{-1}(x) \right) }{\nu_{i-1}} &= \frac{1}{2} \nabla_\theta \expt{\log \det \left( \bar \Sigma(\bar x;\theta) (\CG \Sigma(x) \CG^{tr})^{-1} \right) }{\nu_{i-1}} \\
&= -\frac{1}{2} \nabla_\theta \expt{ \log \det \left(  (\CG \Sigma(x) \CG^{tr}) \bar \Sigma^{-1}(\bar x;\theta) \right) }{\nu_{i-1}}\, .
\end{align*}

Consider the second term. We can split $K_{\Dt_i} = K_{\Dt_i}(\bar x, \bar x';\theta)$ as follows
\begin{align*}
K_{\Dt_i} &=
  \frac{1}{2{\Dt_i}}(\CG m_{\Dt_i}(x){ - }\CG m_{\Dt_i}(x) {+} \bar x'{-}\bar m_{\Dt_i}(\bar x;\theta))^{tr} \bar \Sigma^{-1}(\bar x;\theta)
  \\ & \quad \quad \cdot (\CG m_{\Dt_i}(x) {-} \CG m_{\Dt_i}(x){+}\bar x'{-}\bar m_{\Dt_i}(\bar x;\theta))  \\
& =  \frac{1}{2{\Dt_i}}(\CG m_{\Dt_i}(x) -\bar m_{\Dt_i}(\bar x;\theta))^{tr} \bar \Sigma^{-1}(\bar x;\theta)(\CG m_{\Dt_i}(x) -\bar m_{\Dt_i}(\bar x;\theta))  \\
& + \frac{1}{2{\Dt_i}} ( \bar x' - \CG m_{\Dt_i}(x))^{tr} \bar \Sigma^{-1}(\bar x;\theta)( \bar x'- \CG m_{\Dt_i}(x)) \\
&= \frac{{\Dt_i}}{2} (\bar b(\bar x;\theta) -\CG b(x))^{tr} \bar \Sigma^{-1}(\bar x;\theta) (\bar b(\bar x;\theta) -\CG b(x)) \\
& + \frac{1}{2{\Dt_i}} ( \bar x' - \CG m_{\Dt_i}(x))^{tr} \bar \Sigma^{-1}(\bar x;\theta)( \bar x'- \CG m_{\Dt_i}(x)) 
\, .
\end{align*}

\noindent
Now use $\int p^{(1)}(x,\bar x')d\bar x'=1$ to get
\begin{align}
\nonumber
 \nabla_\theta & \expt{\int p^{(1)}(x,\bar x') K_{\Dt_i}(\bar x, \bar x';\theta) d \bar x'}{\nu_{i-1}} = \\
\tag{I} \label{eq:summ1}
&\frac{{\Dt_i}}{2} \nabla_\theta \expt{ (\bar b(\bar x;\theta) -\CG b(x))^{tr} \bar \Sigma^{-1}(\bar x;\theta)((\bar b(\bar x;\theta) -\CG b(x)))}{\nu_{i-1}}\\
\tag{II} \label{eq:summ2}
&+\frac{1}{2{\Dt_i}} \nabla_\theta \expt{\int  p^{(1)}(x,x') ( \bar x' - \CG m_{\Dt_i}(x))^{tr} \bar \Sigma^{-1}(\bar x;\theta)( \bar x'- \CG m_{\Dt_i}(x))d \bar x' }{\nu_{i-1}} \, .
\end{align}

For the second summand, \eqref{eq:summ2}, consider that $\bar \Sigma(\bar x;\theta)$ is a covariance matrix, so exists a matrix $S=S(\bar x;\theta)$ such that 
$$\bar \Sigma^{-1}(\bar x;\theta) = S^{tr}(\bar x;\theta)S(\bar x;\theta) \, .$$
 Let $z := S(\bar x;\theta) (\bar x' - \CG m_{\Dt_i}(x))$, so we have
$$
\bar x' - \CG m_{\Dt_i}(x) = S^{-1}(\bar x;\theta)z \, ,
$$
and therefore
$$
( \bar x' - \CG m_{\Dt_i}(x))^{tr} \bar \Sigma^{-1}(\bar x;\theta)( \bar x'- \CG m_{\Dt_i}(x))  =  z^{tr} (S^{-1})^{tr}S^{tr}S S^{-1}z = z^{tr}z \, .
$$
Moreover,
\begin{align*}
p^{(1)}(x,\bar x') &= \frac{1}{Z^{(1)}_{\Dt_i}(x)} \exp \left\{ -\frac{1}{2{\Dt_i}}(\bar x'-\CG m_{\Dt_i}(x))^{tr} \Sigma_{1,1}^{-1}(x)(\bar x'-\CG m_{\Dt_i}(x)) \right\} \\ &= \frac{1}{Z^{(1)}_{\Dt_i}(x)}  \exp \left\{  -\frac{1}{2{\Dt_i}} z^{tr}(S^{-1})^{tr} \Sigma_{1,1}^{-1}(x) S^{-1}z  \right\} \\
&= \frac{1}{\sqrt{(2 \CG {\Dt_i})^{\bar d} \det(\Sigma_{1,1}(x))}} \exp \left\{  -\frac{1}{2{\Dt_i}} z^{tr} \tilde \Sigma^{-1}(x) z  \right\}  \, ,
\end{align*}
where $\tilde \Sigma(x) := S \, \Sigma_{1,1}(x) S^{tr}$.

Now perform the change of variables $\det(S) d\bar x' = dz$ to get that 
\begin{align*}
\eqref{eq:summ2} &= \frac{1}{2{\Dt_i}} \nabla_\theta \expt{\frac{1}{\det(S)}\frac{1}{\sqrt{(2 \CG {\Dt_i})^{\bar d} \det(\Sigma_{1,1}(x))}} \exp \left\{  -\frac{1}{2{\Dt_i}} z^{tr} \tilde \Sigma^{-1}(x) z  \right\} z^{tr}z dz }{\nu_{i-1}} \\
& = \frac{1}{2{\Dt_i}} \nabla_\theta \expt{\frac{1}{\sqrt{(2 \CG {\Dt_i})^{\bar d} \det(\tilde \Sigma(x))}} \exp \left\{  -\frac{1}{2{\Dt_i}} z^{tr} \tilde \Sigma^{-1}(x) z  \right\} z^{tr}z dz }{\nu_{i-1}} \\
&= \frac{1}{2{\Dt_i}} \nabla_\theta \expt{\expt{z^{tr}z}{z}}{\nu_{i-1}} \\
&= \frac{1}{2{\Dt_i}} \nabla_\theta \expt{\tr{\expt{zz^{tr}}{z}} }{\nu_{i-1}}\\
&= \frac{1}{2{\Dt_i}} \nabla_\theta \expt{\tr{{\Dt_i} \tilde \Sigma(x)} }{\nu_{i-1}}\\
&= \frac{1}{2} \nabla_\theta \expt{\tr{S \Sigma_{1,1}(x) S^{tr}}}{\nu_{i-1}} \\
&= \frac{1}{2} \nabla_\theta \expt{\tr{\CG \Sigma(x) \CG^{tr} \bar \Sigma^{-1}(\bar x;\theta)}}{\nu_{i-1}}
\, ,
\end{align*}
where $z \sim \mathcal{N}(0,{\Dt_i}\, \tilde{\Sigma}(x))$.

\noindent
Summarizing we have 
\begin{align*}
\nabla_\theta \rer{P_i}{Q_i^\theta} &=-\frac{1}{2} \nabla_\theta \expt{ \log \det \left(  (\CG \Sigma(x) \CG^{tr}) \bar \Sigma^{-1}(\bar x;\theta) \right) }{\nu_{i-1}}\\
&+ \frac{1}{2} \nabla_\theta \expt{\tr{\CG \Sigma(x) \CG^{tr} \bar \Sigma^{-1}(\bar x;\theta)}}{\nu_{i-1}}\\
&+ \frac{{\Dt_i}}{2} \nabla_\theta \expt{ (\bar b(\bar x;\theta) -\CG b(x))^{tr} \bar \Sigma^{-1}(\bar x;\theta)(\bar b(\bar x;\theta) -\CG b(x))}{\nu_{i-1}} \\
& =:\nabla_\theta \left[R_i(\theta) +  M_i(\theta) \Dt_i \right] 
\end{align*}
\end{proof}

\begin{proof}(of Theorem \ref{thm:argmin}).
Apply Lemma \ref{lemm:argmin} and linearity on \eqref{eq:pRE}. Then we have the same minimizers.
\end{proof}

\label{sec:thm_bound_proof}
We first prove the following Lemma which is instrumental for the proof of this theorem. 
\begin{prop}
\label{prop:ineq_R}
Let 
$$B:= (\CG \Sigma(x) \CG^{tr}) \bar \Sigma^{-1}(\bar x;\theta) \quad \text{ be a $\bar d {\times} \bar d$ matrix} \, .$$ 
Then 
$$\tr{B} - \log \det ( B )  \geq \bar d\, .$$
Moreover, the equality is attained if and only if $B=I_{\bar d \times \bar d}$.
\end{prop}
\begin{proof}
Notice that $B$ is a diagonalizable matrix for any $x\in \mathbb{R}^d$, $\bar x \in \mathbb{R}^{\bar d}$, $\theta \in \Theta$ and therefore the problem is reduced to state that 
$$
- \log \lambda_k + \lambda_k \geq 1\, , \,\, k{=}1,2,...,\bar d\, ,
$$
where $\lambda_k$ are the positive eigenvalues of $B$. By simple optimality arguments we get that $1 - 1/\lambda_k \geq 0$ with equality if and only if $\lambda_k=1$ for $k \in \{1,2,...,\bar d\}$.
\end{proof}
\bibliographystyle{elsarticle-num}
\bibliography{refs}

\end{document}